\documentclass[10pt]{article}
\usepackage[utf8]{inputenc} 
\usepackage[T1]{fontenc}
\usepackage{lmodern}

\usepackage{amsthm,amsmath,amsfonts,amssymb,stmaryrd,bbm}
\usepackage{mathtools}
\usepackage{enumitem}
\usepackage{mathrsfs} 

\usepackage{graphicx}
\usepackage[dvipsnames]{xcolor}
\usepackage{caption,tikz}
\usepackage{dsfont}
\usepackage{amsthm}

\usepackage{anyfontsize}

\usepackage[english]{babel}
\usepackage[colorlinks=true, linkcolor=blue, urlcolor=black, citecolor=blue,pdfstartview=FitH]{hyperref}

\numberwithin{equation}{section}

\let\originalleft\left
\let\originalright\right
\renewcommand{\left}{\mathopen{}\mathclose\bgroup\originalleft}
\renewcommand{\right}{\aftergroup\egroup\originalright}

\newlength{\bibitemsep}
\newlength{\bibparskip}
\let\oldthebibliography\thebibliography
\renewcommand\thebibliography[1]{\oldthebibliography{#1}
  \setlength{\parskip}{\bibitemsep}
  \setlength{\itemsep}{\bibparskip}}

\DeclareMathOperator{\im}{Im}
\DeclareMathOperator{\Tr}{Tr}
\DeclareMathOperator{\tr}{tr}
\DeclareMathOperator{\Var}{Var}
\DeclareMathOperator{\Cov}{Cov}

\DeclareMathOperator{\OO}{O}
\DeclareMathOperator{\oo}{o}

\DeclareMathOperator{\Id}{Id}

\DeclareMathOperator{\supp}{supp}
\DeclareMathOperator{\Spec}{Spec}
\DeclareMathOperator{\rank}{rank}

\newcommand{\mc}[1]{\mathcal{#1}}
\newcommand{\mf}[1]{\mathfrak{#1}}
\newcommand{\ms}[1]{\mathscr{#1}}

\newcommand{\rd}{{\rm d}}

\newcommand{\ii}{\mathrm{i}}

\renewcommand{\epsilon}{\varepsilon}
\newcommand{\e}{{\varepsilon}}
\renewcommand{\leq}{\leqslant}
\renewcommand{\geq}{\geqslant}

\renewcommand{\P}{\mathbb{P}}
\newcommand{\E}{\mathbb{E}}
\newcommand{\R}{\mathbb{R}}
\newcommand{\C}{\mathbb{C}}
\newcommand{\N}{\mathbb{N}}

\newcommand{\1}{\mathbbm{1}}

\newcommand{\abs}[1]{\left\lvert #1 \right\rvert}

\newcommand{\vertiii}[1]{{\left\vert\kern-0.25ex\left\vert\kern-0.25ex\left\vert #1 
    \right\vert\kern-0.25ex\right\vert\kern-0.25ex\right\vert}}
\newcommand{\ip}[1]{\left\langle #1 \right\rangle}
\newcommand{\diff}{\mathop{}\!\mathrm{d}}

\theoremstyle{plain} 
\newtheorem{thm}{Theorem}[section]
\newtheorem{lem}[thm]{Lemma}
\newtheorem{cor}[thm]{Corollary}

\newtheorem{prop}[thm]{Proposition}

\newtheorem*{theorem*}{Theorem}

\newtheorem{rem}[thm]{Remark}

\theoremstyle{definition}
\newtheorem{assump}{Assumption}

\makeatletter
\newenvironment{subcorollary}[1]{%
  \def\subcorollarycounter{#1}%
  \refstepcounter{#1}%
  \protected@edef\theparentnumber{\csname the#1\endcsname}%
  \setcounter{parentnumber}{\value{#1}}%
  \setcounter{#1}{0}%
  \expandafter\def\csname the#1\endcsname{\theparentnumber.\Alph{#1}}%
  \expandafter\def\csname theH#1\endcsname{thm.\theparentnumber.\Alph{#1}}%
  \unskip\ignorespaces
}{%
  \setcounter{\subcorollarycounter}{\value{parentnumber}}%
  \ignorespacesafterend
}
\makeatother
\newcounter{parentnumber}

\makeatletter
\renewcommand{\subsection}{\@startsection
{subsection}
{2}
{0mm}
{-\baselineskip}
{0 \baselineskip}
{\normalfont\bf\itshape}} 
\makeatother

\makeatletter
\renewcommand{\subsubsection}{\@startsection
{subsubsection}
{3}
{0mm}
{-\baselineskip}
{0 \baselineskip}
{\normalfont\bf\itshape}} 
\makeatother

\def\author#1{\par
    {\centering{\authorfont#1}\par\vspace*{0.05in}}
}

\def\titlefont{\fontsize{13}{15}\bfseries\boldmath\selectfont\centering{}}
\def\authorfont{\fontsize{13}{15}}

\let\affiliationfont\rhfont

\def\address#1{\par
    {\centering{\affiliationfont#1\par}}\par\vspace*{11pt}
}

\def\title#1{
    \thispagestyle{plain}
    \vspace*{-14pt}
    \vskip 79pt
    {\centering{\titlefont #1\par}}%
    \vskip 1em
}

\makeatletter
\newcommand{\setword}[2]{%
  \phantomsection
  #1\def\@currentlabel{\unexpanded{#1}}\label{#2}%
}
\makeatother

\makeatletter
\renewcommand{\section}{\@startsection
{section}
{1}
{0mm}
{-2\baselineskip}
{1\baselineskip}
{\normalfont\large\scshape\centering}} 
\makeatother

\begin{document}

~\vspace{-1.4cm}

\title{Exponential Growth of Random Determinants Beyond Invariance}

\vspace{1cm}
\noindent
\begin{minipage}[c]{0.33\textwidth}
\author{G\'{e}rard Ben Arous}
\address{Courant Institute\\
    New York University\\
    E-mail: benarous@cims.nyu.edu}
\end{minipage}
\begin{minipage}[c]{0.33\textwidth}
\author{Paul Bourgade}
\address{Courant Institute\\
    New York University\\
    E-mail: bourgade@cims.nyu.edu}
\end{minipage}
\begin{minipage}[c]{0.33\textwidth}
\author{Benjamin M\textsuperscript{c}Kenna}
\address{Courant Institute\\
    New York University\\
    E-mail: mckenna@cims.nyu.edu}
\end{minipage}

\begin{abstract}
We give simple criteria to identify the exponential order of magnitude of the absolute value of the determinant for wide classes of random matrix models, not requiring the assumption of invariance. These include Gaussian matrices with covariance profiles, Wigner matrices and covariance matrices with subexponential tails, Erd{\H o}s-R\'enyi and $d$-regular graphs for any polynomial sparsity parameter, and non-mean-field random matrix models, such as random band matrices for any polynomial bandwidth. 
The proof builds on recent tools, including the theory of the Matrix Dyson Equation as  developed in \cite{AjaErdKru2019}. 

We use these asymptotics as an important input to identify the complexity of classes of Gaussian random landscapes in our companion papers \cite{BenBouMcK2021II, McK2021}.
\end{abstract}

{
	\hypersetup{linkcolor=black}
	\tableofcontents
}


\vspace{-0.5cm}

\section{Introduction}


\subsection{Overview.}\
In this paper, our goal is to study the expected absolute values of the determinants of general $N \times N$ real symmetric random matrices $H_N$, specifically at exponential scale in the large-$N$ limit:
\begin{equation}
\label{eqn:exponentialdetcon}
	\lim_{N \to \infty} \frac{1}{N} \log \E[\abs{\det(H_N)}].
\end{equation}
We identify two sets of simple criteria that lead to asymptotics of this type (Theorems \ref{thm:convex_functional} and \ref{thm:concentrated_input}), and apply them to a wide variety of matrix models.

Initiated in the 1930s, and developed early on by Tur{\'a}n, Fortet, Tukey, Nyquist, Rice, Riordan, Pr{\'e}kopa, and others, the study of random determinants has focused on three distinct questions: the singularity probability (that the determinant of a discrete random matrix vanishes), Gaussian fluctuations, and asymptotics of the type \eqref{eqn:exponentialdetcon}. We will describe this history below in greater detail. The third direction is useful for the topological ``landscape complexity'' program, which studies the geometry of high-dimensional random functions via the Kac-Rice formula, and which motivates our present work.

Most studies in this direction have focused on the invariant Gaussian ensembles. We study random determinants in contexts where the distribution of the matrix $H_N$ is not necessarily invariant by orthogonal conjugacy,  evaluating \eqref{eqn:exponentialdetcon} for matrix models including Gaussian matrices with variance profiles, large zero blocks, or even correlations; Wigner matrices and sample covariance matrices whose entries have subexponential tails; Erd\H{o}s-R{\'e}nyi graphs with parameter $p \geq N^\epsilon/N$; uniform $d$-regular graphs with parameter $N^\epsilon \leq d \leq N^{2/3-\epsilon}$; band matrices with any  bandwidth $W \geq N^\epsilon$; and the classical free-convolution model $A + OBO^T$ with $O$ uniform on the orthogonal group. For example, denoting $\rho_{\rm sc}$ the semicircle density on $[-2,2]$,  for any $E$ we prove that
\[
	\lim_{N \to \infty} \frac{1}{N} \log \E[\abs{\det(W_N - E)}] = \int \log\abs{\lambda - E} \rho_{\text{sc}}(\lambda) \diff \lambda,
\]
whenever $W_N$ is a Wigner matrix (Corollary \ref{cor:Wigner}) or a random band matrix (Corollary \ref{cor:BM}), under the above decay and bandwidth assumptions.

In the companion papers \cite{BenBouMcK2021II, McK2021}, we use these results to study the landscape complexity of non-invariant random functions. There, we 
prove formulas of Fyodorov and Le Doussal \cite{FyoLeD2020} on the classical ``elastic manifold'' from statistical physics, which models a point configuration with local self-interactions in a disordered environment. We also find a new phase transition, with universal near-critical behavior, for a certain anisotropic signal-plus-noise model. 

In fact, for these geometric applications we need to understand asymptotics like \eqref{eqn:exponentialdetcon} when the matrix $H_N$ has long-range correlations, for example when all the diagonal entries are correlated with each other. In the last section of this paper, we show how to give exact variational formulas for asymptotics of this type, based on the (simpler) formulas for matrices with short-range correlations.

Theorem \ref{thm:convex_functional} and \ref{thm:concentrated_input} below prove that we can obtain the asymptotics \ref{eqn:exponentialdetcon} under three general conditions which do not use invariance, stated informally as follows.
\begin{enumerate}[label=(\arabic*)]
\item We can discard the contribution of extremely large and small eigenvalues (at scales $e^{N^\epsilon}$ and $e^{-N^\epsilon}$).
\item Some form of concentration of the empirical spectral measure $\hat{\mu}_{H_N}$ about its mean $\E[\hat{\mu}_{H_N}]$ holds.
\item There exists a deterministic sequence $(\mu_N)_{N=1}^\infty$ of probability measures, sufficiently regular, that are mildly good approximations for the mean spectral measure $\E[\hat{\mu}_{H_N}]$.
\end{enumerate}
Overall, our proof strategy is to write the determinant as an almost-continuous test function integrated against $\hat{\mu}_{H_N}$, regularize the logarithm using (1), prove concentration of this test statistic about its mean using (2), and relate this mean to something more recognizable using (3). Checking condition (1) is typically model-specific, but conditions (2) and (3) can be discussed in general.

To prove condition (2) on concentration of $\hat{\mu}_{H_N}$, we identify two distinct criteria, corresponding to our general theorems:
\begin{itemize}
\item[--] Either (the \emph{convexity-preserving functional} case, Theorem \ref{thm:convex_functional}) $H_N$ is built in a convexity-preserving and Lipschitz way from arbitrary independent random variables,
\item[--] or (the \emph{concentrated input} case, Theorem \ref{thm:concentrated_input}) linear statistics of $H_N$ are already known to concentrate. This is meant to be applied if, e.g., $H_N$ satisfies log-Sobolev, or Gromov-Milman concentration on compact groups.
\end{itemize}

To prove condition (3) regarding convergence of  $\E[\hat{\mu}_{H_N}]$, in the case of classical random matrices the approximating sequence  $(\mu_N)_{N=1}^\infty$ is well-known (and in fact constant): For example, one should choose the semicircle law for Wigner matrices, or the Mar{\v c}enko-Pastur law for sample covariance models. But the good choice of $\mu_N$ for non-invariant Gaussian ensembles, which are the most important matrices for applications to complexity, has only been understood recently, a consequence of the theory of the Matrix Dyson Equation (MDE) as developed in \cite{AjaErdKru2019, AltErdKruNem2019}. Given nice $H_N$, the MDE produces a probability measure $\mu_N$ found by solving a constrained problem over matrices. The existence, uniqueness, and regularity theory of the MDE is an important input for our work. 

The organization of the paper is as follows: In the rest of this section, we give some history on determinants of random matrices, then state our main results. We prove our general results, Theorems \ref{thm:convex_functional} and \ref{thm:concentrated_input}, in Section \ref{sec:determinantproofs}, then prove our applications to matrix models in Section \ref{sec:modelproofs}. In Section \ref{sec:longrange}, we discuss determinants in the presence of long-range correlations. In Appendix \hyperlink{sec:multipoint}{A} we  extend our results to product of determinants, showing
\begin{equation}
\label{eqn:intro_multipoint}
	\lim_{N \to \infty} \frac{1}{N}\log \E\left[ \prod_{i=1}^\ell |\det(H_N^{(i)})| \right] = \sum_{i=1}^\ell \left( \lim_{N \to \infty} \frac{1}{N} \log \E[|\det(H_N^{(i)})|] \right)
\end{equation}
for any fixed $\ell$ and random matrices $H_N^{(1)}, \ldots, H_N^{(\ell)}$ which may be correlated with each other. This asymptotic factoring holds regardless of the correlation structure between the $H_N^{(i)}$'s.  Finally, in Appendix \hyperlink{sec:secondmoment}{B} we find a transition for the exponential order of determinants of Wigner matrices $W_N$, showing that, for any $p \geq 1$, the quantity $\limsup_{N \to \infty} \frac{1}{N} \log \E[\abs{\det(W_N-E)^p}]$ is finite if and only if the entries have finite $(2p)$-th moment.


\subsection{History.}\
The earliest research on random-matrix determinants covered non-Hermitian matrices with i.i.d. entries, discussing an extremal problem on the determinant of Bernoulli matrices \cite{SzeTur1937} (extended in \cite{Tur1955}) and exact formulas at finite $N$ for small moments of determinants \cite{For1951, ForTuk1952, NyqRicRio1954, Pre1967} (see also Girko's book \cite{Gir1990}). Later in the literature, we identify three main strands of research on determinants.

First, one can ask for the probability that an $N \times N$ discrete matrix (Bernoulli, say) is singular, i.e., that its determinant is zero, for large $N$. In the non-Hermitian case, Koml{\'o}s showed that this probability is $o(1)$ \cite{Kom1967, Kom1968}. Recently K. Tikhomirov established the long-standing conjecture that this probability is $(\frac{1}{2}+o(1))^N$ \cite{Tik2020}; earlier exponential estimates in this direction include \cite{KahKomSze1995, TaoVu2006, TaoVu2007, BouVuWoo2010}.

Second, one can show that the determinant, appropriately normalized, has Gaussian fluctuations. In the non-Hermitian case, if the entries are Gaussian this follows from work of Goodman \cite{Goo1963}. Gaussianity was replaced by an exponential-tails assumption in \cite{NguVu2014} and a fourth-moment assumption in \cite{BaoPanZho2015}. In the Hermitian case, Gaussian matrices were studied in \cite{DelLeC2000}. Gaussianity was relaxed to a four-moment-matching assumption in \cite{TaoVu2012}, then to a two-moment-matching assumption in \cite{BouMod2019,BouModPai2022}. Some other ensembles were treated in \cite{CaiLiaZho2015, Rou2007}, and more about determinants for Gaussian ensembles  was discussed in \cite{BorLaC2015, EdeLaC2015}. 

Third, one can study the same question we do here, namely the asymptotics of $\E[\abs{\det(H_N)}]$, usually in the same context of studying complexity for high-dimensional random fields. Here we just discuss the types of random matrices that have appeared; for a discussion of what these prior results mean for complexity, we refer to the companion paper \cite{BenBouMcK2021II}. Fyodorov \cite{Fyo2004} studied Gaussian matrices of type $\text{GOE} + \mc{N}(0,1/N)\Id$  using supersymmetry, and a similar model was addressed in Auffinger \emph{et al.} \cite{AufBenCer2013} using known large-deviations principles (LDPs) \cite{BenGui1997, BenDemGui2001}. Rank-one perturbations of GOE appeared in \cite{BenMeiMonNic2019}, using an LDP of Ma{\"i}da \cite{Mai2007}. An upper bound for full-rank perturbations of GOE appeared in \cite{FanMeiMon2021}, based on free probability and large deviations. Upper and lower bounds for Gaussian matrices with a certain covariance structure were given in \cite{AufChe2014}. The (Gaussian) real elliptic ensemble was discussed in \cite{BenFyoKho2021}, based on a new result on large deviations for its spectral measure. Baskerville \emph{et al.} cover finite-rank perturbations of GOE in \cite{BasKeaMezNaj2021} and a specific ensemble of Gaussian matrices with a variance profile, inspired by a two-layer spin-glass model, in \cite{BasKeaMezNaj2022}. In both cases the determinant analysis is performed through supersymmetry, for the asymptotic spectral density and for Wegner estimates. Our corollaries \ref{cor:GaussianCovarianceA}, \ref{cor:GaussianCovarianceB}, and \ref{cor:GaussianBlock} about general Gaussian ensembles provide alternative derivations for all these results about Hermitian matrices. These corollaries also make rigorous the analysis of random determinants by Fyodorov and Le Doussal \cite{FyoLeD2020} (see \cite{BenBouMcK2021II} for corresponding complexity results).

Finally, asymptotics for a \emph{pair} of determinants, in the style of \eqref{eqn:intro_multipoint} with $\ell = 2$, appeared for a particular pair of random matrices from spin glasses, closely related to correlated GOE matrices, in \cite{Sub2017, AufGol2020, BenSubZei2020}. These arguments were based on known LDPs for Gaussian ensembles. 


\subsection*{Notations.}\
We write $\|\cdot\|$ for the operator norm on elements of $\C^{N \times N}$ induced by the ${\rm L}^2$ distance on $\C^N$. We let
$
    \|f\|_{\textup{Lip}} = \sup_{x \neq y}\frac{\abs{f(x)-f(y)}_{{\rm L}^2}}{\abs{x-y}_{{\rm L}^2}}
$
for  functions $f : \R^m \to \R^n$, and consider the following three distances on probability measures on the real line (called bounded-Lipschitz, Wasserstein-$1$, and Kolmogorov-Smirnov, respectively):
\begin{align*}
    d_{\textup{BL}}(\mu,\nu) &= \sup\left\{\abs{\int_\R f(x) (\mu - \nu)(\diff x)} : \|f\|_{\text{Lip}} + \|f\|_{L^\infty} \leq 1\right\}, \\
    {\rm W}_1(\mu,\nu) &= \sup\left\{\abs{\int_\R f(x) (\mu - \nu)(\diff x)} : \|f\|_{\text{Lip}} \leq 1\right\}, \\
    d_{\textup{KS}}(\mu,\nu) &= \sup\{ \abs{\mu((-\infty,x]) - \nu((-\infty,x])} : x \in \R\}.
\end{align*}
We normalize the semicircle law as
$
	\rho_{\text{sc}}(\diff x) = \frac{\sqrt{4-x^2}}{2\pi} \, \mathbf{1}_{x \in [-2,2]} \diff x.
$
We write $\mathtt{l}(\mu)$ for the left edge (respectively, $\mathtt{r}(\mu)$ for the right edge) of a compactly supported measure $\mu$. For an $N \times N$ Hermitian matrix $M$, we write $\lambda_{\min{}}(M) = \lambda_1(M) \leq \cdots \leq \lambda_N(M) = \lambda_{\max{}}(M)$ for its eigenvalues and 
$
    \hat{\mu}_M = \frac{1}{N}\sum_{i=1}^N \delta_{\lambda_i(M)}
$
for its empirical measure. We write $\ms{S}_N$ for the set of all $N \times N$ real symmetric matrices, which we often identify with the space $\R^{\frac{N(N+1)}{2}}$, and on which we therefore put the norm 
\begin{equation}
\label{eqn:frobenius-like-norm}
	\|T\|_{\widetilde{F}}^2 = \sum_{1 \leq i \leq j \leq N} T_{ij}^2
\end{equation}
(slightly different from the Frobenius norm since we only sum on and above the diagonal; but we will also use the Frobenius norm $\|T\|_F^2 = \sum_{1 \leq i, j \leq N} T_{ij}^2$). We also write $\boxplus$ for the free (additive) convolution of probability measures. 

We write $B_R$ for the ball of radius $R$ around 0 in the relevant Euclidean space. We use $(\cdot)^T$ for the matrix transpose, which is distinguished both from  the matrix conjugate transpose $(\cdot)^\ast$, and from the matrix trace $\Tr(\cdot)$.


\subsection{General theorem for convexity-preserving functional.}\
The following Theorem \ref{thm:convex_functional} is our first general result. It applies to random matrices without any a priori concentration hypothesis, but requires the tools of convex analysis, in particular results of Talagrand.

To state the hypotheses, we denote $\kappa>0$ an arbitrarily small control parameter which does not depend on $N$. Let $M = M_N \geq 1$. Consider
$X=(X_1,\dots,X_M)$ a random vector. We now consider the following set of assumptions.

\begin{enumerate}
\item[\hypertarget{assn:I}{(I)}]  The $X_i$'s are independent and real-valued.  
\item[\hypertarget{assn:M}{(M)}] Matrix model. Let $H=H_N=\Phi(X)$ where $\Phi:\mathbb{R}^M\to\mathscr{S}_N$ is deterministic and Lipschitz (with respect to the norm \eqref{eqn:frobenius-like-norm}), and $\Phi^{-1}(A)$ is convex for any convex set $A$. 
\item[\hypertarget{assn:E}{(E)}] Expectation. A sequence of probability measures $\mu_N$ exists satisfying the following properties. First, 
\begin{equation}
\label{eqn:dBL}
    d_{\textup{BL}}(\E\hat\mu_{\Phi(X)},\mu_N)\leq N^{-\kappa}.
\end{equation}
Moreover, the $\mu_N$'s are supported in a common compact set, and each has a density $\mu_N(\cdot)$ in the same neighborhood $(-\kappa,\kappa)$ around $0$, which satisfies $\mu_N(x)<\kappa^{-1}|x|^{-1+\kappa}$ for all $|x|<\kappa$ and all $N$.
\item[\hypertarget{assn:C}{(C)}] Coarse bounds. Write $(\lambda_i)_{i=1}^N$ for the eigenvalues of $\Phi(X)$. For every $\epsilon > 0$,
\begin{align}
&\lim_{N \to \infty} \frac{1}{N} \log \E\left[ \prod_{i=1}^N (1+\abs{\lambda_i}\mathds{1}_{\abs{\lambda_i} > e^{N^\epsilon}})\right] = 0, \label{eqn:detconub_coarse}
\\
&\lim_{N \to \infty} \P(\Phi(X) \text{ has no eigenvalues in }[-e^{-N^\epsilon},e^{-N^\epsilon}]) = 1. \label{eqn:wegner}
\end{align}
In addition, there exists $\delta > 0$ such that
\begin{equation}
\label{eqn:detconub_NlogN} 
    \limsup_{N \to \infty} \frac{1}{N \log N} \log \E[\abs{\det(H_N)}^{1+\delta}] < \infty.
\end{equation}
\item[\hypertarget{assn:S}{(S)}] Spectral stability. Let  $(X_{\rm cut})_i=X_i\mathds{1}_{|X_i|<N^{-\kappa}/\|\Phi\|_{\textup{Lip}}}$ (recalling again the norm \eqref{eqn:frobenius-like-norm}). We have
\begin{equation}
\label{eqn:ss}
\lim_{N \to \infty} \frac{1}{N\log N}\log\mathbb{P}\left(d_{\textup{KS}}(\hat\mu_{\Phi(X)},\hat\mu_{\Phi(X_{\rm cut})})>N^{-\kappa}\right)=-\infty.
\end{equation}  \label{assn:S}
\end{enumerate}

\begin{thm}
\label{thm:convex_functional}
\textbf{(Convexity-preserving functional)}
Under the assumptions \hyperlink{assn:I}{(I)}, \hyperlink{assn:M}{(M)}, \hyperlink{assn:E}{(E)}, \hyperlink{assn:C}{(C)}, \hyperlink{assn:S}{(S)}, we have
\begin{equation}\label{eqn:conv}
\lim_{N\to\infty}\left(\frac{1}{N}\log\E[\abs{\det(H_N)}]-\int_\R \log|\lambda| \mu_N(\diff \lambda)\right)=0.
\end{equation}
\end{thm}

\noindent {\it Comments on the result.}\ (i) 
A polynomial rate in \eqref{eqn:detconub_coarse} is enough to give a polynomial rate of convergence $$\abs{\frac{1}{N} \log \E[\abs{\det(H_N)}] - \int_\R \log\abs{\lambda} \mu_N(\diff \lambda)} \leq N^{-\epsilon}$$ for some $\epsilon > 0$ and $N \geq N_0(\epsilon)$. Indeed, an examination of the proof shows that $\epsilon$ depends only on $\kappa$ and the polynomial rate in \eqref{eqn:detconub_coarse}, but $N_0(\epsilon)$ also depends on the rates of convergence in \eqref{eqn:wegner} and \eqref{eqn:ss}, and on the permissible values of $\delta$ and the value of the $\limsup$ in \eqref{eqn:detconub_NlogN}. 

(ii) The matrix $H_N$ does not need to be centered. As an elementary example, we can choose $H_N = W_N - E$ for $W_N$ a Wigner matrix and obtain concentration around $\int_\R \log\abs{\lambda-E} \rho_{\text{sc}}(\lambda) \diff \lambda$; see Corollary \ref{cor:Wigner} below.

(iii) The proof uses Talagrand's classic concentration inequality for product measures. We want to recognize the determinant almost as a Lipschitz, convex function of independent, bounded random variables. Ideally these would be the $X_i$'s, but they are not bounded; however, we truncate them using assumption \hyperlink{assn:S}{(S)}. The functional $H = \Phi(X)$ gives the Lipschitz, convex condition, after regularizing the logarithm using assumption \hyperlink{assn:C}{(C)}.

\bigskip

\noindent {\it Comments on the assumptions.}\ We discuss briefly why our assumptions are reasonable and close to optimal. In our applications, $\Phi$ is linear so assumption \hyperlink{assn:M}{(M)} is trivially satisfied, but $\Phi$ is also allowed to create correlations between the entries in a nonlinear fashion. Equation \eqref{eqn:wegner} avoids a non-trivial kernel, an obviously necessary condition for \eqref{eqn:conv}. Equation \eqref{eqn:detconub_NlogN} asks for slightly more integrability than finiteness of $\limsup N^{-1}\log\E[|\det H_N|]$ which is implied by the result and assumption \hyperlink{assn:E}{(E)}. In Section \ref{sec:assumptions} we show the importance \eqref{eqn:detconub_coarse} (which is a constraint on large eigenvalues) and assumption \hyperlink{assn:S}{(S)} (which essentially states that the spectrum should not depend too much on a small number of $X_i$'s): for each of these, we give an example of a distribution on matrices satisfying every \emph{other} assumption but not this one, for which the result of the theorem fails.


\subsection{General theorem for concentrated input.}\
Here we consider the problem of exponential growth for random matrices $H_N$ that already satisfy some concentration property directly, without having to cut the tails and apply a result of Talagrand as in (the proof of) Theorem \ref{thm:convex_functional}. For example, in applications we will take matrices whose upper triangles satisfy a log-Sobolev inequality (even if correlated), or Gromov-Milman-type concentration. We remark that the dichotomy in Theorems \ref{thm:convex_functional} and \ref{thm:concentrated_input} -- namely, proving the similar results, once under product-measure assumptions and once under log-Sobolev-style assumptions -- first appeared in the classic concentration paper of Guionnet-Zeitouni \cite{GuiZei2000}. We have termed these models ``concentrated input,'' to contrast with the previous section's ``convexity-preserving functional'' where $H_N$ is written as $\Phi(X)$ and concentration is provided by convexity-preserving properties of $\Phi$ (and tail bounds). In this section, we will therefore consider $H_N$ directly. We will also replace some of the assumptions above with the following.

\begin{enumerate}
\item[\hypertarget{assn:W}{(W)}] Wasserstein-$1$.
A sequence of probability measures $\mu_N$ exists satisfying the following properties. First, 
\begin{equation}
\label{eqn:wasserstein}
    {\rm W}_1(\E\hat\mu_{H_N},\mu_N)\leq N^{-\kappa}.
\end{equation}
Moreover, the $\mu_N$'s are supported in a common compact set, and each has a density $\mu_N(\cdot)$ in the same neighborhood $(-\kappa,\kappa)$ around $0$, which satisfies $\mu_N(x)<\kappa^{-1}|x|^{-1+\kappa}$ for all $|x|<\kappa$ and all $N$.
\item[\hypertarget{assn:L}{(L)}] Concentration for Lipschitz traces.
There exists $\epsilon_0 > 0$ with the following property: For every $\zeta > 0$, there exists $c_\zeta > 0$ such that, whenever $f : \R \to \R$ is Lipschitz, we have for every $\delta > 0$
\begin{equation}
\label{eqn:lipschitztrace}
    \P\left(\abs{\frac{1}{N}\Tr(f(H_N)) - \frac{1}{N}\E[\Tr(f(H_N))]} > \delta\right) \leq \exp\left(-\frac{c_\zeta}{N^\zeta} \min\left\{ \left(\frac{N\delta}{\|f\|_{\textup{Lip}}}\right)^2, \left(\frac{N\delta}{\|f\|_{\textup{Lip}}}\right)^{1+\epsilon_0}\right\}\right).
\end{equation}
\end{enumerate}

On a first pass readers can drop the $N^{-\zeta}$ factor in \eqref{eqn:lipschitztrace}. It is included because, for Gaussian matrices as in Section \ref{sec:GausCov}, our assumption on the correlation structure implies \eqref{eqn:lipschitztrace} for every $\zeta > 0$ but not necessarily for $\zeta = 0$.

\begin{thm}
\label{thm:concentrated_input}
\textbf{(Concentrated input)} Under the assumptions \hyperlink{assn:W}{(W)}, \hyperlink{assn:L}{(L)}, and the gap assumption \eqref{eqn:wegner}, we have
\[
    \lim_{N \to \infty} \left( \frac{1}{N}\log \E[\abs{\det(H_N)}] - \int_\R \log\abs{\lambda} \mu_N(\diff \lambda) \right) = 0.
\]
\end{thm}

As in Theorem \ref{thm:convex_functional}, by examining the proof one can find a small polynomial rate $N^{-\epsilon}$ in Theorem \ref{thm:concentrated_input}. 

Compared to \cite{GuiZei2000}, we do not require bounded entries in Theorem \ref{thm:convex_functional}, our matrix models are more general, and we consider logarithmic singularities. On the other hand, \cite{GuiZei2000} identifies the correct \emph{scale} of fluctuations, analogous to a rate of convergence of order $N^{-1}$  in \eqref{eqn:conv}, for test functions without singularities.


\subsection{Wigner matrices.}\
\label{subsec:wignerstatement}
We now discuss determinant asymptotics for Wigner matrices $W_N$ whose entries have subexponential tails. (In Appendix \hyperlink{sec:secondmoment}{B} below, we consider what can be said when the entries only have finite second moment.)

Let $\mu$ be a centered probability measure with unit variance that has subexponential tails, in the sense that there exist constants $\alpha, \beta > 0$ such that, if $X \sim \mu$, then
\begin{equation}
\label{eqn:subexponentialtails}
    \P(\abs{X} \geq t^\alpha) \leq \beta e^{-t}
\end{equation}
for all $t > 0$. Let $W_N$ be a real symmetric $N \times N$ Wigner matrix associated with $\mu$, by which we mean that the entries of $\sqrt{N}W_N$ are independent up to symmetry and each distributed according to $\mu$. The following corollary uses Theorem \ref{thm:convex_functional}.

\begin{cor}\label{cor:Wigner}
\textbf{(Wigner matrices with exponential tails)}
For every $E \in \R$ we have
\[
    \lim_{N \to \infty} \frac{1}{N}\log \E[\abs{\det(W_N - E)}] = \int_\R \log\abs{\lambda - E} \rho_{\text{sc}}(\lambda) \diff \lambda.
\]
\end{cor}

An examination of the proof shows local uniformity in $E$, meaning that for every compact $K \subset \R$ we have
\[
	\lim_{N \to \infty} \sup_{E \in K} \left( \frac{1}{N} \log \E[\abs{\det(W_N - E)}] - \int_\R \log\abs{\lambda-E} \rho_{\text{sc}}(\lambda) \diff \lambda \right) = 0.
\]

\begin{rem}
One would also be interested in results of the form
\begin{equation}
\label{eqn:deformedwigner}
	\lim_{N \to \infty} \frac{1}{N} \log \E[\abs{\det(W_N + D_N)}] = \int_\R \log\abs{\lambda} (\rho_{\text{sc}} \boxplus \mu_D)(\diff \lambda),
\end{equation}
where $(D_N)_{N=1}^\infty$ is a sequence of deterministic matrices whose empirical measures tend to some compactly-supported $\mu_D$ (at some polynomial speed and without outliers, say). Our techniques could likely be extended to prove such a result under the assumption of subexponential tails on the Wigner matrices. We do not pursue this direction further here; however, in the companion paper \cite{BenBouMcK2021II}, we prove \eqref{eqn:deformedwigner} with a different approach when $W_N$ is a GOE matrix. For a related problem, see the free-addition model below, in Corollary \ref{cor:freeaddition}.
\end{rem}


\subsection{Erd\H{o}s-R{\'e}nyi matrices.}\
We now consider Erd\H{o}s-R{\'e}nyi matrices with near-optimal sparsity parameter $p \geq N^\epsilon/N$, i.e., when each vertex has expected degree $N^\epsilon$. It is classical that the limiting spectral distribution of such matrices is semicircular as long as $p = \omega(1/N)$ (see, e.g., \cite{TraVuWan2013}), but not semicircular anymore if $p = \alpha/N$ for $\alpha$ fixed (see, e.g., \cite{BauGol2001}).

Fix some $\epsilon > 0$, and let $H_N$ be an $N \times N$ Erd{\H{o}}s-R{\'e}nyi random matrix with parameter $1-\epsilon \geq p_N \geq \frac{N^\epsilon}{N}$. scaled so that the bulk eigenvalues are order one. This means that the entries are independent up to symmetry and
\[
    (H_N)_{ij} = \frac{1}{\sqrt{Np_N(1-p_N)}} \begin{cases} 1 & \text{with probability } p_N, \\ 0 & \text{with probability } 1-p_N. \end{cases}
\]
The following corollary uses Theorem \ref{thm:convex_functional}.

\begin{cor}\label{cor:ER}
\textbf{(Erd\H{o}s-R{\'e}nyi matrices with $p \geq N^\epsilon/N$)}
For any $E \in \R$ with $\abs{E} \neq 2$ we have
\[
    \lim_{N \to \infty} \frac{1}{N}\log \E[\abs{\det(H_N-E)}] = \int_\R \log\abs{\lambda-E} \rho_{\text{sc}}(\lambda) \diff \lambda.
\]
\end{cor}
This result is locally uniform for $E$ away from the edges, meaning $E$ in any compact subset of $\R \setminus \{-2, 2\}$.


\subsection{$d$-regular matrices.}\
We now consider $d$-regular random graphs for $N^\epsilon \leq d \leq N^{2/3-\epsilon}$, i.e., we fix once and for all an $\epsilon > 0$ and let $H'_N$ be the adjacency matrix of a (uniformly) random, simple $d$-regular graph on $N$ vertices for some sequence $d = d_N$ satisfying 
\[
	N^\epsilon \leq d_N \leq N^{\frac{2}{3}-\epsilon}.
\]
Then we consider the normalization
\begin{equation}
\label{eqn:dreg-law}
	H_N = \frac{1}{\sqrt{d(1-\frac{d}{N})}} H'_N.
\end{equation}
Tran, Vu, and Wang \cite{TraVuWan2013} showed that the limiting empirical spectral measure of $H_N$ is semicircular as long as $d_N \to \infty$.

\begin{prop}
\label{prop:dreg}
\textbf{($d$-regular matrices with $N^\epsilon \leq d \leq N^{2/3-\epsilon}$)}
For any $E \in \R$ with $\abs{E} \neq 2$ we have
\[
	\lim_{N \to \infty} \frac{1}{N} \log \E[\abs{\det(H_N-E)}] = \int_\R \log\abs{\lambda-E} \rho_{\text{sc}}(\lambda) \diff \lambda.
\]
\end{prop}

We call this a ``proposition'' rather than a ``corollary'' because it is not a direct consequence of our theorems, but rather can be proved in a similar way. We give details in Section \ref{sec:dreg}.

We note that the assumption $d_N \leq N^{\frac{2}{3}-\epsilon}$ is only to verify Assumption \hyperlink{assn:W}{(W)} and the Wegner estimate \eqref{eqn:wegner} using (much stronger) local laws of Bauerschmidt-Knowles-Yau and Bauerschmidt-Huang-Knowles-Yau \cite{BauKnoYau2017, BauHuaKnoYau2020}. It is likely the result holds up to $d_N \leq N^{1-\epsilon}$ (of course, if $d_N = N$, the determinant is zero).


\subsection{Band matrices.}\
In this section we consider random band matrices $H_N$, i.e., matrices whose $(i,j)$th entry is zero unless $i$ and $j$ are less than some $W$ apart. Many statistics of $H_N$ are believed to undergo a phase transition at $W \sim N^{1/2}$. For example, the eigenvectors are supposed to be localized on $\oo(N)$ sites for $W \ll N^{1/2}$ and delocalized for $W \gg N^{1/2}$. However, we establish that the determinant asymptotics do not see this phase transition: They are the same as long as $W \to +\infty$ polynomially in $N$. For a full discussion, we direct the reader to \cite{Bou2018}.

Let $\mu$ be a centered probability measure with unit variance that has subexponential tails in the sense of \eqref{eqn:subexponentialtails}. Suppose also that $\mu$ has a bounded density $\mu(\cdot)$. Fix any $\epsilon > 0$. Let $H_N$ be an $N \times N$ band matrix with bandwidth $W = W_N \geq N^\epsilon$ corresponding to $\mu$. This means that $H_N$ has independent entries up to symmetry with
\[
    (H_N)_{ij} \begin{cases} = 0 & \text{if } \|i - j\| > W, \\ \sim \frac{X}{\sqrt{2W+1}} & \text{if } \|i - j\| \leq W. \end{cases}
\]
(Here we take periodic distance $\|i-j\| = \min(\abs{i-j},N-\abs{i-j})$.) The following corollary uses Theorem \ref{thm:convex_functional}.

\begin{cor}\label{cor:BM}
\textbf{(One-dimensional band matrices with bandwidth $W \geq N^\epsilon$)}
Under the above assumptions,
\[
    \lim_{N \to \infty} \frac{1}{N} \log \E[\abs{\det(H_N - E)}] = \int_\R \log\abs{\lambda-E} \rho_{\text{sc}}(\lambda) \diff \lambda.
\]
\end{cor}

This result is locally uniform in $E$.

We now comment on the significance of this result. In the companion paper \cite{BenBouMcK2021II}, we solve a problem of Fyodorov-Le Doussal \cite{FyoLeD2020} on a model called the ``elastic manifold.'' They consider a mean-field version of this model, corresponding to block-banded random matrices with bandwidth order $N$, and find the ``Larkin mass'' separating ordered and disordered phases. An important open problem is the behavior of the elastic manifold beyond mean field, when the corresponding random matrix is block-banded with sublinear bandwidth. It does not seem to be clear in which regimes this Larkin transition should persist, but Corollary \ref{cor:BM} may suggest that the transition remains for any polynomial bandwidth.


\subsection{Sample covariance matrices.}\
\label{subsec:covarianceMatstatement}
Let $\mu$ be a centered probability measure on $\R$ with unit variance and subexponential tails in the sense of \eqref{eqn:subexponentialtails}. We assume $\mu$ has density $f=e^{-g}$ with $f$ smooth enough in the sense that, for any $a\geq 1$, there exists $C_a>0$ such that for any $s\in\mathbb{R}$
\begin{equation}\label{eqn:smooth}
|\widehat{f}(s)|+|\widehat{f g''}(s)|\leq \frac{C_a}{(1+s^2)^a}.
\end{equation}

  Let $Y_{p,N}$ be a $p \times N = p_N \times N$ matrix whose entries are independent copies of $\mu$. Suppose that
\[
	\gamma = \lim_{N \to \infty} \frac{p_N}{N} \in (0,1].
\]
If $\gamma < 1$, we require a mild speed-of-convergence assumption
\begin{equation}
\label{eqn:wishartspeed}
	\abs{\gamma - \frac{p_N}{N}} \leq N^{-\epsilon}
\end{equation}
for some $\epsilon > 0$; if $\gamma = 1$, then for technical reasons we require $p_N = N$, i.e., we require the matrices to be exactly square rather than asymptotically square. Write $\mu_{\textup{MP},\gamma}$ for the Mar{\v{c}}enko-Pastur distribution
\begin{equation}\label{eqn:MP}
    \mu_{\textup{MP},\gamma}(\diff x) = \frac{\sqrt{(b_\gamma - x)(x - a_\gamma)}}{2\pi \gamma x} \mathds{1}_{[a_\gamma,b_\gamma]} \diff x
\end{equation}
where $a_\gamma = (1-\sqrt{\gamma})^2$, $b_\gamma = (1+\sqrt{\gamma})^2$. 

\begin{prop}\label{prop:Cov}
\textbf{(Sample covariance matrices with subexponential tails)} Under the above assumptions, 
for every $E \in \R$, we have
\[
    \lim_{N \to \infty} \frac{1}{p_N}\log\E\left[\abs{\det\left(\frac{1}{N}Y_{p,N}(Y_{p,N})^T - E\right)}\right] = \int \log\abs{\lambda - E} \mu_{\textup{MP},\gamma}(\lambda) \diff \lambda.
\]
\end{prop}

As for $d$-regular matrices, this is called a ``proposition'' rather than a ``corollary'' because it is proved along the same lines as our theorems, rather than following from them in a strict sense. The details are in Section \ref{sec:covarianceMat}. The proof also shows, as usual, that the limit holds uniformly in $E$.

Proposition \ref{prop:Cov} complements a 1989 result of Dembo \cite{Dem1989}, who gave an exact formula at finite $N$ for the averaged determinant in the special case $E = 0$, without requiring the assumption of a bounded density. In our normalization, he showed by a combinatorial method that
\[
	\E\left[\abs{\det\left(\frac{1}{N}Y_{p,N}(Y_{p,N})^T\right)}\right] = \E\left[\det\left(\frac{1}{N}Y_{p,N}(Y_{p,N})^T\right)\right] = \frac{N!}{N^p(N-p)!},
\]
and one can check from the known log-potential of the Mar{\v c}enko-Pastur law that $\lim_{N \to \infty} \frac{1}{N}\log \left( \frac{N!}{N^p(N-p)!} \right)$ is the same as given by our proposition.


\subsection{Gaussian matrices with a (co)variance profile.}\ \label{sec:GausCov}
Let $H_N$ be an $N \times N$ real symmetric Gaussian matrix, possibly with a mean, a variance profile, and/or correlated entries, satisfying the technical assumptions below. These are essentially the assumptions needed for the local law of Erd{\H{o}}s \emph{et al.} \cite{ErdKruSch2019} which we will use in the proof. We first give an easier statement for matrices with independent entries up to symmetry (Corollary \ref{cor:GaussianCovarianceA}), then a more involved statement for matrices with correlations (Corollary \ref{cor:GaussianCovarianceB}). In the statement, we decompose $H_N = A_N + W_N$ where $A_N = \E[H_N]$. These corollaries use Theorem \ref{thm:concentrated_input}.

In the following mean-field conditions, the arbitrary parameter $p > 0$ is fixed.
\begin{itemize}
\item[\hypertarget{assn:B}{(B)}] Bounded mean. We have $\sup_N \|A_N\| < \infty$.
\item[\hypertarget{assn:F}{(F)}] Flatness. For each $N$, 
\[
     T \in \C^{N \times N}, \quad T \text{ positive semi-definite} \quad \implies \quad \frac{1}{p} \frac{\Tr(T)}{N} \leq \E[W_N T W_N] \leq p \frac{\Tr(T)}{N}.
\]
\end{itemize}

Let $\mu_N$ be the measure from the size-$N$ Matrix Dyson Equation, that is, the measure with density $\mu_N(\cdot)$ whose Stieltjes transform at $z \in \mathbb{H}$ is $\frac{1}{N}\Tr(M_N(z))$, where $M_N(z)$ is the (unique, deterministic) solution to the following constrained equation over $\C^{N \times N}$:
\begin{align*}
    &\Id_{N \times N} + (z \Id_{N \times N} - A_N + \E[W_NM_N(z)W_N]) M_N(z) = 0 \\
    &\text{subject to} \quad \im M_N(z) = \frac{M_N(z) - M_N(z)^\ast}{2\ii } > 0 \quad \text{in the sense of quadratic forms}.
\end{align*}

\begin{subcorollary}{thm}
\begin{cor}
\textbf{(Gaussian matrices with a variance profile)}
\label{cor:GaussianCovarianceA}
If $H_N$ has independent entries up to symmetry, then under assumptions \hyperlink{assn:B}{(B)} and \hyperlink{assn:F}{(F)} we have
\[
    \lim_{N \to \infty} \left( \frac{1}{N}\log \E[\abs{\det(H_N)}] - \int_\R \log\abs{\lambda} \mu_N(\lambda) \diff \lambda \right) = 0.
\]
\end{cor}

\noindent The following assumptions are needed if $H_N$ has correlations among its entries beyond the symmetry constraints.
\begin{itemize}
\item[\hypertarget{assn:wF}{(wF)}] Weak fullness. Whenever $T \in \R^{N \times N}$ is real symmetric, 
\[
    \E\left[ (\Tr(BW))^2 \right] \geq N^{-1-p}\Tr(B^2).
\]
(The $p = 0$ case is called ``fullness'' in \cite{AjaErdKru2019}.)
\item[\hypertarget{assn:D}{(D)}] Decay of correlations. Write $\kappa$ for multivariate cumulants (for any number of arguments), and consider the distance on subsets of $\llbracket 1, N \rrbracket^2$ given by $d(A,B) = \min\{ \min\{\abs{\alpha - \beta}, \abs{\alpha^t - \beta}\} : \alpha \in A, \beta \in B\}$ where $(\cdot)^t$ switches the elements of an ordered pair. For the order-two cumulants we assume 
\[
    \abs{\kappa(f_1(W_N), f_2(W_N))} \leq \frac{C}{1+d(\supp f_1, \supp f_2)^s} \|f_1\|_2 \|f_2\|_2
\]
for some $s > 12$ and all $L^2$ functions $f_1, f_2$ on $N \times N$ matrices. For order-$k$ cumulants, $k \geq 3$, we consider, for any $L^2$ functions $f_1, \ldots, f_k$, the complete graph on $\{1, \ldots, k\}$ with the edge-weights $d(\{i,j\}) = d(\supp f_i, \supp f_j)$. Writing $T_{\textup{min}}$ for the minimal spanning tree on this graph (i.e., smallest sum of edge weights) and lifting covariance to edges as $\kappa(\{i,j\}) = \kappa(f_i,f_j)$, we assume
\[
    \abs{\kappa(f_1(W_N), \ldots, f_k(W_N))} \leq C_k \prod_{e \in E(T_{\textup{min}})} \abs{\kappa(e)}.
\]
(In fact, our results hold under some weaker correlation-decay conditions that are longer to state; see \cite[Example 2.12]{ErdKruSch2019}.)
\end{itemize}

\begin{cor}
\textbf{(Gaussian matrices with a (co)variance profile)}
\label{cor:GaussianCovarianceB}
Under assumptions \hyperlink{assn:B}{(B)}, \hyperlink{assn:F}{(F)}, \hyperlink{assn:wF}{(wF)}, and \hyperlink{assn:D}{(D)}, we have
\[
    \lim_{N \to \infty} \left( \frac{1}{N}\log \E[\abs{\det(H_N)}] - \int_\R \log\abs{\lambda} \mu_N(\lambda) \diff \lambda \right) = 0.
\]
\end{cor}
\end{subcorollary}

Corollary \ref{cor:GaussianCovarianceA} is an immediate consequence of Corollary \ref{cor:GaussianCovarianceB}, because it is easy to check that \hyperlink{assn:F}{(F)} implies both \hyperlink{assn:wF}{(wF)} and \hyperlink{assn:D}{(D)} if $H_N$ has independent entries up to symmetry. In Section \ref{sec:GausCovProof} we therefore only prove Corollary \ref{cor:GaussianCovarianceB}.

In some cases one can show that the sequence $(\mu_N)_{N=1}^\infty$ has a limit $\mu_\infty$, and obtain $\lim_{N \to \infty} \frac{1}{N}\log \E[\abs{\det(H_N)}] = \int \log\abs{\lambda} \mu_\infty(\diff \lambda)$. Notice this does not follow from our assumptions, because we do not assume any consistency in $N$. For example, this corollary applies to the contrived example $H_N = \text{GOE} + (-1)^N \Id$. In the companion paper \cite{BenBouMcK2021II}, we show how to use the (well-established) stability theory of the Matrix Dyson Equation to find a limit $\mu_\infty$ when it exists.


\subsection{Block-diagonal Gaussian matrices.}\
In this section, we are interested in Gaussian random matrices with large zero blocks. These are \emph{not} covered by Corollary \ref{cor:GaussianCovarianceB}, since the ``flatness'' assumption there implies that all entries have variance in some $[\frac{c}{N}, \frac{C}{N}]$. In the landscape complexity program, such  block-diagonal matrices describe random functions whose components in certain directions are \emph{independent} of those in other directions. In the companion paper \cite{BenBouMcK2021II}, we study one such random function from statistical physics, called the ``elastic manifold.''

Consider matrices $H_N = A_N + W_N$, with $A_N = \E[H_N]$, that have the following special form. 
Fix once and for all some $K \in \N$ (the number of blocks), and consider matrices in $\R^{K \times K} \otimes \R^{N \times N}$, i.e., matrices with $K^2$ blocks each of size $N \times N$. Write $E_{ii}$ for the matrix with a one in the $(i,i)$th entry and zeros otherwise; depending on the context this will be either an $N \times N$ matrix or a $K \times K$ matrix.
\begin{itemize}
\item[\hypertarget{assn:MS}{(MS)}] Bounded mean structure. Consider a deterministic triangular array $(a_i)_{i=1}^N = (a_{i,N})_{i=1}^N$ with each $a_i \in \R^{K \times K}$, and define 
\[
    A_N = \sum_{i=1}^N a_i \otimes E_{ii}.
\]
In particular $A_N$ can only have nonzero entries on the diagonals of each block. Assume
\[
    \sup_N \|A_N\| < \infty.
\]
\item[\hypertarget{assn:MF}{(MF)}] Mean-field randomness in diagonal blocks. The Gaussian random matrix $W_N$ has the form
\[
    W_N = \sum_{i=1}^K E_{ii} \otimes X_i = \begin{pmatrix} X_1 & 0 & \cdots & 0 \\ 0 & X_2 & \cdots & 0 \\ \vdots & \vdots & \ddots & \vdots \\ 0 & 0 & \cdots & X_K \end{pmatrix},
\]
where the $X_i$'s are independent $N \times N$ Gaussian random matrices, each of which has centered independent entries up to symmetry. Write $x_{jk}^{(i)}$ for the $(j,k)$th entry of $X_i$ and $s_{jk}^{(i)}$ for its variance. For some parameter $p$, each $i \in \llbracket 1, K \rrbracket$, and each $j, k \in \llbracket 1, N \rrbracket$, we have
\[
    s_{jk}^{(i)} \leq \frac{p}{N}, \quad s_{jj}^{(i)} \geq \frac{1}{pN}.
\]
Notice the lower bound is only along the diagonal.
\item[\hypertarget{assn:R}{(R)}] Regularity of MDE solution. Given $\mathbf{r} = (r_1, \ldots, r_N) \in (\C^{K \times K})^N$, define
\begin{equation}
\label{eqn:block_stability_operators}
    \ms{S}_i[\mathbf{r}] = \sum_{k=1}^N \sum_{j=1}^K s_{ik}^{(j)} E_{jj} r_k E_{jj} \in \C^{K \times K}
\end{equation}
for each $i \in \llbracket 1, N \rrbracket$. The MDE in this context is a system of $N$ coupled equations over $K \times K$ matrices; we seek the unique solution $\mathbf{m}(z) = \mathbf{m}^{(N)}(z) = (m_1(z), \ldots, m_N(z)) = (m_1^{(N)}(z), \ldots, m_N^{(N)}(z)) \in (\C^{K \times K})^N$ to
\begin{align}
\label{eqn:block_mde}
\begin{split}
    &\Id_{K \times K} + (z\Id_{K \times K} - a_i + \ms{S}_i[\mathbf{m}(z)])m_i(z) = 0 \\
    &\text{subject to} \quad \im m_i(z) > 0 \quad \text{as a quadratic form.}
\end{split}
\end{align}
Consider the probability measure $\mu_N$ on $\R$ whose Stieltjes transform at the point $z$ is $\frac{1}{NK} \sum_{j=1}^N \Tr m_j(z)$. 

Assume that each $\mu_N$ admits a density with respect to Lebesgue measure, and that these densities are bounded in $L^\infty$, uniformly over $N$.
\end{itemize}

The following corollary uses Theorem \ref{thm:concentrated_input}.
\begin{cor}
\textbf{(Block-diagonal Gaussian matrices)}
\label{cor:GaussianBlock}
Under assumptions \hyperlink{assn:MS}{(MS)}, \hyperlink{assn:MF}{(MF)}, and \hyperlink{assn:R}{(R)}, we have
\[
    \lim_{N \to \infty} \left( \frac{1}{NK} \log \E[\abs{\det(H_N)}] - \int_\R \log\abs{\lambda} \mu_N(\lambda) \diff \lambda \right) = 0.
\]
(The normalization is $\frac{1}{NK}$ because $H_N$ is an $NK \times NK$ matrix.)
\end{cor}

In applications to landscape complexity, the description of these measures $\mu_N$ via the MDE is very important to prove properties of the limit measures $\mu_\infty$. For example, in our companion paper \cite{BenBouMcK2021II}, we use this MDE description to identify a crucial convexity property in a variational problem.


\subsection{Free addition.}\
Let $(A_N)_{N=1}^\infty$, $(B_N)_{N=1}^\infty$ be a sequence of deterministic, $N \times N$, real diagonal matrices, whose empirical measures tend to some $\mu_A$, $\mu_B$ respectively. We will be interested in the random matrix $A_N + O_NB_NO_N^T$, where $O_N$ is sampled from Haar measure on the orthogonal group $\mc{O}_N$. 

We require the following assumptions.
\renewcommand\labelitemi{--}
\begin{itemize}
\item The measures $\mu_A$ and $\mu_B$ admit densities $\rho_A$ and $\rho_B$, respectively. These densities have single nonempty interval supports $[E_-^A,E_+^A]$ and $[E_-^B,E_+^B]$, and each density is strictly positive on the interior of its support.
\item Each measure $\mu_A$ and $\mu_B$ has a power-law behavior with exponent in $(-1,1)$ at each of its edges; that is, there exist $\delta > 0$ and exponents $-1 < t_-^A, t_-^B, t_+^A, t_+^B < 1$ such that, for some $C > 1$,
\begin{align*}
    C^{-1} \leq \frac{\rho_A(x)}{(x-E_-^A)^{t_-^A}} \leq C \quad \text{for all } x \in [E_-^A, E_-^A+\delta], \qquad C^{-1} \leq \frac{\rho_B(x)}{(x-E_-^B)^{t_-^B}} \leq C \quad \text{for all } x \in [E_-^B, E_-^B+\delta], \\
    C^{-1} \leq \frac{\rho_A(x)}{(E_+^A-x)^{t_+^A}} \leq C \quad \text{for all } x \in [E_+^A-\delta, E_+^A], \qquad C^{-1} \leq \frac{\rho_B(x)}{(E_+^B-x)^{t_+^B}} \leq C \quad \text{for all } x \in [E_+^B-\delta, E_+^B].
\end{align*}
\item One of the measures $\mu_A$ and $\mu_B$ has a bounded Stieltjes transform.
\item The eigenvalues $(a_i)_{i=1}^N = (a_i^{(N)})_{i=1}^N$ of $A_N$, ordered increasingly, are close to the classical particle locations $a_i^\ast$ defined by
\[
    a_i^\ast = \inf\left\{s : \int_{-\infty}^s \mu_A(\diff y) = i/N\right\}
\]
in the sense that for any $c > 0$, 
$
    \sup_{1 \leq i \leq N} \abs{a_i - a_i^\ast} \leq N^{-1+c}
$
for $N$ sufficiently large. The analogous condition also holds for the eigenvalues of $B_N$.
\end{itemize}
For example, all of these assumptions are satisfied if $\mu_A$ is the semicircle law and $\mu_B$ is either a uniform measure, the Mar{\v c}enko-Pastur law, or the semicircle law; and if $A_N$ and $B_N$ store the relevant $\frac{1}{N}$-quantiles.

The following corollary uses Theorem \ref{thm:concentrated_input}.
\begin{cor}
\label{cor:freeaddition}
\textbf{(Free addition)}
If $O_N$ is chosen randomly from the Haar measure on the orthogonal group $\mc{O}_N$, then whenever $E$ is not an edge of $\mu_A \boxplus \mu_B$, we have
\[
    \lim_{N \to \infty} \frac{1}{N}\log \E[\abs{\det(A_N + O_NB_NO_N^T - E)}] = \int_\R \log\abs{\lambda-E} (\mu_A \boxplus \mu_B)(\lambda) \diff \lambda. 
\]
\end{cor}

\noindent This result is locally uniform in $E$ away from the edge, meaning in any compact subset of $\R \setminus \{\mathtt{l}(\mu_A \boxplus \mu_B), \mathtt{r}(\mu_A \boxplus \mu_B)\}$.

\bigskip

\emph{Comment on the assumptions.} For the proof, we check the assumptions of the concentrated-input Theorem \ref{thm:concentrated_input} using the local law of Bao-Erd{\H{o}}s-Schnelli \cite{BaoErdSch2020} and the fixed-energy universality of Che-Landon \cite{CheLan2019}. For concise writing, the assumptions we state here are a bit stronger than ``the union of the assumptions of these two papers,'' but in fact this union suffices for Corollary \ref{cor:freeaddition}. In fact, our result likely holds under even weaker assumptions than required in these papers, which handle more fine-grained questions.


\subsection*{Acknowledgements.}\
We wish to thank Nick Cook, Amir Dembo, L\'{a}szl\'{o} Erd\H{o}s, Yan Fyodorov, Torben Kr\"{u}ger, Pierre Le Doussal, Krishnan Mody, and Ofer Zeitouni for helpful discussions. We are also grateful to a referee for pointing out an error in an earlier version of the paper. GBA acknowledges support by the Simons Foundation collaboration Cracking the Glass Problem, PB was supported by NSF grant DMS-1812114 and a Poincar\'e chair, and BM was supported by NSF grant DMS-1812114.


\section{Proofs of determinant asymptotics}
\label{sec:determinantproofs}


\subsection{Proof of Theorem \ref{thm:convex_functional}.}\
The proof depends on a careful tuning of many $N$-dependent parameters; in the next section we define these parameters and prove some estimates that are common to both the upper and lower bounds. In the following subsections we then prove these upper and lower bounds in order.

\subsubsection{Definitions and common estimates.}\
Let $\kappa$ be as in the assumptions (i.e., given to us), and write $K, \eta, t, w_b, p_b$ for some $N$-dependent parameters. In fact we will choose
\begin{equation}
\label{eqn:detcon_parameters}
    \begin{cases}
        K = e^{N^\epsilon} \quad \text{for some $\epsilon$ small enough ($\epsilon = \kappa^2/16$ suffices),} \\
        \eta = N^{-\kappa/2}, \\
        t = N^{-\kappa/4}, \\
        w_b = N^{-\kappa/4}, \\
        p_b = N^{-\kappa^2/8},
    \end{cases}
\end{equation}
but we find it more transparent to work with the names $K$, $\eta$, and so on for the bulk of the proof, checking only at the end that these specific choices make the error estimates useful. We will work with the following regularizations of the logarithm:
\begin{align*}
    \log_\eta(\lambda) &= \log\abs{\lambda + \ii \eta}, \\
    \log_\eta^K(\lambda) &= \min(\log_\eta(\lambda),\log_\eta(K)).
\end{align*}
Let $b = b_N : \R \to \R$ be some smooth, even, nonnegative function that is identically one on $[-w_b,w_b]$, vanishes outside of $[-2w_b,2w_b]$, and is $\frac{1}{w_b}$-Lipschitz. Consider the following events:
\begin{align}
\label{eqn:manyevents}
\begin{split}
    \mc{E}_{\textup{gap}} &= \{\Phi(X) \text{ has no eigenvalues in } [-e^{-N^\epsilon},e^{-N^\epsilon}] \}, \\
    \mc{E}_{\textup{ss}} &= \{d_{\textup{KS}}(\hat{\mu}_{\Phi(X)},\hat{\mu}_{\Phi(X_{\textup{cut}})}) \leq N^{-\kappa}\}, \\
    \mc{E}_{\textup{conc}} &= \left\{\abs{\int \log_\eta^K(\lambda) (\hat{\mu}_{\Phi(X_{\textup{cut}})} - \E[\hat{\mu}_{\Phi(X_{\textup{cut}})}])(\diff \lambda)} \leq t\right\}, \\
    \mc{E}_b &= \left\{ \int b(\lambda) \hat{\mu}_{\Phi(X)}(\diff \lambda) \leq p_b\right\}.
\end{split}
\end{align}
It turns out that all of these events are likely. For $\mc{E}_{\textup{gap}}$ and $\mc{E}_{\textup{ss}}$ this is by assumption; we will prove that $\mc{E}_{\textup{conc}}$ and $\mc{E}_b$ are likely below.

Now we collect some estimates which will be useful for both the upper and lower bounds.
\begin{lem}
\label{lem:borcapcha2011}
We have
\[
    \abs{\int \log_\eta^K(\lambda) (\hat{\mu}_{\Phi(X)} - \hat{\mu}_{\Phi(X_{\textup{cut}})})(\diff \lambda)}\mathds{1}_{\mc{E}_{\textup{ss}}} \leq N^{-\kappa} \log\left(1+\frac{K^2}{\eta^2}\right).
\]
\end{lem}
\begin{proof}
The proof of \cite[Lemma C.2]{BorCapCha2011} shows that, if $\hat{\mu}_A$ and $\hat{\mu}_B$ are empirical measures of matrices $A$ and $B$ (which have the same size as each other) and if $f$ is a test function of bounded variation, then
\[
    \abs{\int f(\lambda) \hat{\mu}_A(\diff \lambda) - \int f(\lambda) \hat{\mu}_B(\diff \lambda)} \leq \|f\|_{\textup{TV}} \cdot d_{\textup{KS}}(\hat{\mu}_A,\hat{\mu}_B).
\]
Then the result follows from the computation $\|\log_\eta^K\|_{\textup{TV}} = \log\left(1+\frac{K^2}{\eta^2}\right)$ and the definition of $\mc{E}_{\textup{ss}}$. 
\end{proof}

\begin{lem}
\label{lem:detcon_eps1}
With
\[
    \epsilon_1(N) := N^{-\kappa} \log\left(1+\frac{K^2}{\eta^2}\right) + 2\|\log_\eta^K\|_\infty \P((\mc{E}_{\textup{ss}})^c) + \left(\frac{1}{2\eta} + \|\log_\eta^K\|_\infty\right) N^{-\kappa},
\]
we have
\[
    \abs{\int \log_\eta^K(\lambda) (\E[\hat{\mu}_{\Phi(X_{\textup{cut}})}] - \mu_N)(\diff \lambda)} \leq \epsilon_1(N).
\]
\end{lem}
\begin{proof}
First, by inserting $\mathds{1}_{\mc{E}_{\textup{ss}}}$ and using Lemma \ref{lem:borcapcha2011}, we find
\begin{align*}
    \abs{\int \log_\eta^K(\lambda) (\E[\hat{\mu}_{\Phi(X_{\textup{cut}})}] - \E[\hat{\mu}_{\Phi(X)}])(\diff \lambda)}
    &\leq N^{-\kappa}\log\left(1+\frac{K^2}{\eta^2}\right) + 2\|\log_\eta^K\|_\infty \P((\mc{E}_{\textup{ss}})^c).
\end{align*}
Next, since $\log_\eta^K$ is $\frac{1}{2\eta}$-Lipschitz, \eqref{eqn:dBL} yields
\[
    \abs{\int \log_\eta^K(\lambda) (\E[\hat{\mu}_{\Phi(X)}] - \mu_N)(\diff \lambda)} \leq \left(\frac{1}{2\eta} + \|\log_\eta^K\|_\infty\right) d_{\textup{BL}}(\E[\hat{\mu}_{\Phi(X)}], \mu_N) \leq \left( \frac{1}{2\eta} + \|\log_\eta^K\|_\infty \right)N^{-\kappa}.
\]
Both equations above conclude the proof.
\end{proof}

\begin{lem}
\label{lem:prob_econc}
Let
$
    t_0(N) = 24\sqrt{2\pi}/(\eta N^{\frac{1}{2}+\kappa}).
$
If $t \geq t_0(N)$, then 
\[
    \P((\mc{E}_{\textup{conc}})^c) \leq 12\exp\left(-\frac{(t-t_0(N))^2\eta^2N^{1+2\kappa}}{288}\right).
\]
\end{lem}
\begin{proof}
The function $\log_\eta^K$ is not convex (it is convex on $[-\eta,\eta]$ and concave outside this interval). But it is a linear combination of three convex functions. Indeed, for $i = 1, 2, 3$, consider $\log_i = \log_{i, \eta,K} : \R \to \R$ given by
\begin{align*}
	\log_1(x) &= \begin{cases} -\frac{x}{2\eta} - \frac{1}{2} + \log_\eta(\eta) & \text{if } x \leq -\eta, \\ \log_\eta(x) & \text{if } -\eta \leq x \leq \eta, \\ \frac{x}{2\eta} - \frac{1}{2} + \log_\eta(\eta) & \text{if } x \geq \eta ,\end{cases} \\
	\log_2(x) &= \begin{cases} \frac{x}{2\eta} & x \leq \eta, \\ \log_\eta^K(x) + \frac{1}{2} - \log_\eta(\eta) & \text{if } x \geq \eta, \end{cases} \\
	\log_3(x) &= \begin{cases} -\frac{x}{2\eta} & \text{if } x \geq -\eta, \\ \log_\eta^K(x) + \frac{1}{2} - \log_\eta(\eta) & \text{if } x \leq -\eta. \end{cases}
\end{align*}
Notice that $\log_\eta^K = \sum_{i=1}^3 \log_i$, that $\log_1$ is convex while $\log_2$ and $\log_3$ are concave, and that each $\log_i$ is $\frac{1}{2\eta}$-Lipschitz. For each $i$, consider the function $f_i : [-\frac{N^{-\kappa}}{\|\Phi\|_{\textup{Lip}}}, \frac{N^{-\kappa}}{\|\Phi\|_{\textup{Lip}}} ]^M \to \R$ given by
\[
	f_i(X_{\textup{cut}}) = (-1)^{\mathds{1}_{i \neq 1}} \frac{1}{N}\tr(\log_i(\Phi(X_{\textup{cut}}))) = (-1)^{\mathds{1}_{i \neq 1}} \int_\R \log_i(\lambda) \hat{\mu}_{\Phi(X_{\textup{cut}})}(\diff \lambda).
\]
The factors of $-1$ are for convenience, so that each $f_i$ will be convex. Notice that
\begin{equation}
\label{eqn:convexdecomposition}
	\P((\mc{E}_{\textup{conc}})^c) = \P\left(\abs{\sum_{i=1}^3 (-1)^{\mathds{1}_{i \neq 1}} (f_i(X_{\textup{cut}}) - \E[f_i(X_{\textup{cut}})])} > t\right) \leq \sum_{i=1}^3 \P\left(\abs{f_i(X_{\textup{cut}}) - \E[f_i(X_{\textup{cut}})]} > \frac{t}{3}\right).
\end{equation}
Each $f_i$ is a Lipschitz, convex function of the many independent compactly supported variables $(X_{\textup{cut}})_1, \ldots, (X_{\textup{cut}})_M$. Thus we can apply concentration-of-measure results of Talagrand. It will be useful to factor $f_i = g_i \circ \Phi$, where $g_i : \ms{S}_N \to \R$ is given by $g_i(T) = (-1)^{\mathds{1}_{i \neq 1}} \frac{1}{N}\tr(\log_i(T))$. 

Indeed, since $\log_i$ is $(2\eta)^{-1}$-Lipschitz, we know that $g_i$ is $(\eta\sqrt{2N})^{-1}$-Lipschitz (see, e.g., \cite[Lemma 2.3.1]{AndGuiZei2010}, recalling our norm \eqref{eqn:frobenius-like-norm}), and thus $f_i$ is $\|\Phi\|_{\textup{Lip}}/(\eta\sqrt{2N})$-Lipschitz. Furthermore, since $(-1)^{\mathds{1}_{i \neq 1}} \log_i$ is convex, by Klein's lemma (see, e.g., \cite[Lemma 1.2]{GuiZei2000}) $g_i$ is also convex; since we assumed that $\Phi$ pulls back convex sets to convex sets, we conclude that $\{X_{\textup{cut}} : f_i(X_{\textup{cut}}) \leq a\}$ is a convex set of $[-\frac{N^{-\kappa}}{\|\Phi\|_{\textup{Lip}}}, \frac{N^{-\kappa}}{\|\Phi\|_{\textup{Lip}}}]^M$ for every $a \in \R$.
Then \cite[Theorem 6.6]{Tal1996} implies that
\[
    \P(\abs{f_i(X_{\textup{cut}}) - \mf{M}_{f_i}} \geq t) \leq 4 \exp\left( - \frac{t^2\eta^2N^{1+2\kappa}}{32} \right)
\]
where $\mf{M}_{f_i}$ is a median of $f_i(X_{\textup{cut}})$.
We conclude using \eqref{eqn:convexdecomposition} and the estimate
\[
    \abs{\E(f_i(X_{\textup{cut}})) - \mf{M}_{f_i}} \leq \E\abs{f_i(X_{\textup{cut}}) - \mf{M}_{f_i}} \leq 4\int_0^\infty \exp\left( - \frac{t^2\eta^2N^{1+2\kappa}}{32}\right) \diff t = \frac{8\sqrt{2\pi}}{\eta N^{\frac{1}{2}+\kappa}} = \frac{1}{3}t_0(N)
\]
to substitute the median with the mean.
\end{proof}

\subsubsection{Upper bound.}\
After establishing one more estimate, we prove the upper bound of Theorem \ref{thm:convex_functional}.

\begin{lem}
\label{lem:detconub_ec}
With the parameter choices \eqref{eqn:detcon_parameters}, we have
\[
    \lim_{N \to \infty} \frac{1}{N}\log \E[\abs{\det(H_N)}( 1 - \mathds{1}_{\mc{E}_{\textup{ss}}} \mathds{1}_{\mc{E}_{\textup{conc}}} )] = -\infty.
\]
\end{lem}
\begin{proof}
 Writing $\mc{E} = \mc{E}_{\textup{ss}} \cap \mc{E}_{\textup{conc}}$, for any $\delta>0$ H\"{o}lder's inequality gives
\begin{align*}
    \frac{1}{N}\log \E[\abs{\det(H_N)} \mathds{1}_{\mc{E}^c}] &\leq \frac{1}{(1+\delta)N} \log \E[\abs{\det(H_N)}^{1+\delta}] + \frac{\delta}{(1+\delta)N} \log \P(\mc{E}^c).
\end{align*}
For $\delta$ satisfying \eqref{eqn:detconub_NlogN}, the first term is $\OO(\log N)$. Concerning the second term, we have 
\[
    \frac{1}{N}\log \P(\mc{E}^c) \leq \frac{1}{N}\log[\P((\mc{E}_{\textup{conc}})^c) + \P((\mc{E}_{\textup{ss}})^c)] \leq -C\log N,
\]
for any $C>0$ and $N\geq N_0(C)$, where the last inequality follows from Lemma \ref{lem:prob_econc}, our parameter choices \eqref{eqn:detcon_parameters}, and our assumption \eqref{eqn:ss}. 
\end{proof}

\begin{proof}[Proof of upper bound.]
From our assumptions on $\mu_N$ we have
$
    \liminf_{N \to \infty} \int \log\abs{\lambda} \mu_N(\diff \lambda) > -\infty.
$
Thus, by Lemma \ref{lem:detconub_ec}, it suffices to prove
\begin{equation}
\label{eqn:detconub_suffices}
    \limsup_{N \to \infty} \left( \frac{1}{N}\log \E[\abs{\det(H_N)} \mathds{1}_{\mc{E}_{\textup{ss}}} \mathds{1}_{\mc{E}_{\textup{conc}}}] - \int \log\abs{\lambda} \mu_N(\diff \lambda)\right) \leq 0.
\end{equation}
On the events $\mc{E}_{\textup{ss}}$ and $\mc{E}_{\textup{conc}}$, Lemmas \ref{lem:borcapcha2011} and \ref{lem:detcon_eps1} give us
\begin{align*}
    &\int \log_\eta^K(\lambda) \hat{\mu}_{\Phi(X)}(\diff \lambda) \\
    &= \int \log_\eta^K(\lambda) (\hat{\mu}_{\Phi(X)} - \hat{\mu}_{\Phi(X_{\textup{cut}})})(\diff \lambda) + \int \log_\eta^K(\lambda) (\hat{\mu}_{\Phi(X_{\textup{cut}})} - \E[\hat{\mu}_{\Phi(X_{\textup{cut}})}])(\diff \lambda) + \int \log_\eta^K(\lambda) \E[\hat{\mu}_{\Phi(X_{\textup{cut}})}](\diff \lambda) \\
    &\leq N^{-\kappa} \log\left(1+\frac{K^2}{\eta^2}\right) + t + \int \log_\eta^K(\lambda) \E[\hat{\mu}_{\Phi(X_{\textup{cut}})}](\diff \lambda) \leq 2\epsilon_1(N) + t + \int \log_\eta^K(\lambda) \mu_N(\diff \lambda).
\end{align*}
We use this estimate to obtain
\begin{align*}
    \frac{1}{N}\log \E[\abs{\det(H_N)} \mathds{1}_{\mc{E}_{\textup{ss}}} \mathds{1}_{\mc{E}_{\textup{conc}}} ] &= \frac{1}{N}\log \E\left[\left(\prod_{i : \abs{\lambda_i} \leq K} \abs{\lambda_i} \right)\left( \prod_{i : \abs{\lambda_i} > K} \abs{\lambda_i} \right) \mathds{1}_{\mc{E}_{\textup{ss}}} \mathds{1}_{\mc{E}_{\textup{conc}}} \right] \\
    &\leq \frac{1}{N}\log \E\left[ e^{N\int \log_\eta^K\rd \hat{\mu}_{\Phi(X)}} \left( \prod_{i=1}^N (1+\abs{\lambda_i}\mathds{1}_{\abs{\lambda_i} > K})\right) \mathds{1}_{\mc{E}_{\textup{ss}}} \mathds{1}_{\mc{E}_{\textup{conc}}} \right] \\
    &\leq 2\epsilon_1(N) +t + \frac{1}{N}\log \E\left[ \prod_{i=1}^N (1+\abs{\lambda_i}\mathds{1}_{\abs{\lambda_i} > K})\right] + \int \log_\eta^K(\lambda) \mu_N(\diff \lambda).
\end{align*}
From our choice of parameters \eqref{eqn:detcon_parameters} and the assumption \eqref{eqn:detconub_coarse}, this last term is $\int \log_\eta^K(\lambda) \mu_N(\diff \lambda) + \oo(1)$. Furthermore, since the $\mu_N$'s are supported on a common compact set and $K$ increases with $N$, we have $\int \log_\eta^K(\lambda) \mu_N(\diff \lambda) = \int \log_\eta(\lambda) \mu_N(\diff \lambda)$ for $N$ large enough. Thus to prove \eqref{eqn:detconub_suffices} we need only show
\begin{equation}
\label{eqn:detconub_logeta}
    \limsup_{N \to \infty} \int (\log_\eta(\lambda) - \log\abs{\lambda}) \mu_N(\diff \lambda) \leq 0.
\end{equation}
To show this, we use
\[
    \int_\kappa^\infty (\log_\eta(\lambda) - \log\abs{\lambda})\mu_N(\diff \lambda) \leq \frac{1}{2}\log\left(1 + \frac{\eta^2}{\kappa^2}\right) 
\]
which tends to zero since $\eta$ does, and
\[
    \abs{\int_{-\kappa}^\kappa (\log_\eta(\lambda) - \log\abs{\lambda}) \mu_N(\diff \lambda)} \leq \kappa^{-1} \int_{-\kappa}^\kappa (\log\abs{\lambda} - \log_\eta(\lambda)) \abs{\lambda}^{-1+\kappa} \diff \lambda,
\]
which tends to zero by dominated convergence. This completes the proof of \eqref{eqn:detconub_logeta} and thus of the upper bound.
\end{proof}

\subsubsection{Lower bound.}\
We first collect some  estimates.

\begin{lem}
\label{lem:eps2}
We have 
\[
    \frac{1}{N}\log\E[e^{N\int (\log\abs{\lambda} - \log_\eta(\lambda))\hat{\mu}_{\Phi(X)}(\diff \lambda)} \mathds{1}_{\mc{E}_{\textup{gap}}} \mathds{1}_{\mc{E}_{\textup{ss}}} \mathds{1}_{\mc{E}_{\textup{conc}}} \mathds{1}_{\mc{E}_b}] \geq -\epsilon_2(N),
\]
where
\[
    \epsilon_2(N) = \frac{p_b}{2}\log(1+e^{2N^\epsilon} \eta^2) + \frac{\eta^2}{2w_b^2} - \frac{1}{N}\log\P(\mc{E}_{\textup{gap}}, \mc{E}_{\textup{ss}}, \mc{E}_{\textup{conc}}, \mc{E}_b).
\]
\end{lem}
\begin{proof}
On $\mc{E}_{\textup{gap}}$, for any eigenvalue $\lambda$ of $\Phi(X)$ we have
\[
    \log\abs{\lambda} - \log_\eta(\lambda) = -\frac{1}{2}\log\left(1+\frac{\eta^2}{\lambda^2}\right) \geq -\frac{1}{2}\log(1+e^{2N^\epsilon}\eta^2).
\]
Similarly, since $1-b(\lambda) \leq \mathds{1}_{\abs{\lambda} \geq w_b}$ and $\log(1+x) \leq x$ for $x>0$, we have
\[
    \int (\log\abs{\lambda} - \log_\eta(\lambda))(1-b(\lambda)) \hat{\mu}_{\Phi(X)}(\diff \lambda) \geq -\frac{1}{2} \log\left(1+\frac{\eta^2}{w_b^2}\right) \geq -\frac{\eta^2}{2w_b^2}.
\]
Thus
\begin{align*}
    &\E[e^{N\int (\log\abs{\lambda} - \log_\eta(\lambda)) \hat{\mu}_{\Phi(X)}(\diff \lambda)} \mathds{1}_{\mc{E}_{\textup{gap}}} \mathds{1}_{\mc{E}_{\textup{ss}}} \mathds{1}_{\mc{E}_{\textup{conc}}} \mathds{1}_{\mc{E}_b}] \\
    &\geq e^{-\frac{Np_b}{2}\log(1+e^{2N^\epsilon}\eta^2)}\E[e^{N\int (\log\abs{\lambda} - \log_\eta(\lambda))(1-b)(\lambda) \hat{\mu}_{\Phi(X)}(\diff \lambda)} \mathds{1}_{\mc{E}_{\textup{gap}}} \mathds{1}_{\mc{E}_{\textup{ss}}} \mathds{1}_{\mc{E}_{\textup{conc}}} \mathds{1}_{\mc{E}_b}] \\
    &\geq e^{-\frac{Np_b}{2}\log(1+e^{2N^\epsilon}\eta^2)} e^{-\frac{N\eta^2}{2w_b^2}} \P(\mc{E}_{\textup{gap}}, \mc{E}_{\textup{ss}}, \mc{E}_{\textup{conc}}, \mc{E}_b),
\end{align*}
which concludes the proof.
\end{proof}

\begin{lem}
\label{lem:prob_eb}
For $N$ large enough we have
\[
    \P((\mc{E}_b)^c) \leq \frac{2}{p_b}\left(\frac{N^{-\kappa}}{w_b} + \frac{(2w_b)^\kappa}{\kappa^2}\right).
\]
\end{lem}
\begin{proof}
By our choice \eqref{eqn:detcon_parameters} of $w_b$ tending to zero,  $\mu_N$ admits a density on $[-2w_b,2w_b]$ for $N$ large enough. Since $b(\lambda)$ is $\frac{1}{w_b}$-Lipschitz and bounded above by $\mathds{1}_{\abs{\lambda} \leq 2w_b}$, we use \eqref{eqn:dBL} to find
\begin{align*}
    \E\left[ \int b(\lambda) \hat{\mu}_{\Phi(X)}(\diff \lambda)\right] \leq \left(\frac{1}{w_b}+1\right) d_{\textup{BL}}(\E[\hat{\mu}_{\Phi(X)}],\mu_N) + \mu_N([-2w_b,2w_b]) \leq \frac{2N^{-\kappa}}{w_b} + \frac{1}{\kappa}\int_{-2w_b}^{2w_b} \abs{x}^{-1+\kappa} \diff x.
\end{align*}
The conclusion follows by evaluating this integral  and applying Markov's inequality.
\end{proof}

\begin{proof}[Proof of lower bound.] Lemmas \ref{lem:borcapcha2011}, \ref{lem:detcon_eps1}, and \ref{lem:eps2}  show that 
$N^{-1}\log\E[\abs{\det(H_N)}]$ is larger than
\begin{align}
 &\frac{1}{N} \log \E\left[e^{N\left( \int (\log\abs{\lambda} - \log_\eta(\lambda))\hat{\mu}_{\Phi(X)}(\diff \lambda) + \int \log_\eta^K(\lambda) (\hat{\mu}_{\Phi(X)} - \hat{\mu}_{\Phi(X_{\textup{cut}})} + \hat{\mu}_{\Phi(X_{\textup{cut}})} - \E[\hat{\mu}_{\Phi(X_{\textup{cut}})}])(\diff \lambda)\right)} \mathds{1}_{\mc{E}_{\textup{gap}}} \mathds{1}_{\mc{E}_{\textup{ss}}} \mathds{1}_{\mc{E}_{\textup{conc}}}\right] \notag \\
 &+ \int \log_\eta^K(\lambda) \E[\hat{\mu}_{\Phi(X_{\textup{cut}})}](\diff \lambda) \notag\\
    \geq& \, \frac{1}{N}\log\E[e^{N\int (\log\abs{\lambda} - \log_\eta(\lambda))\hat{\mu}_{\Phi(X)}(\diff \lambda)} \mathds{1}_{\mc{E}_{\textup{gap}}} \mathds{1}_{\mc{E}_{\textup{ss}}} \mathds{1}_{\mc{E}_{\textup{conc}}}] - N^{-\kappa} \log\left(1+\frac{K^2}{\eta^2}\right) - t + \int \log_\eta^K(\lambda) \E[\hat{\mu}_{\Phi(X_{\textup{cut}})}](\diff \lambda) \notag\\
    \geq&  \int \log\abs{\lambda} \mu_N(\diff \lambda) - \epsilon(N),\label{eqn:detconlb_fundamental}
\end{align}
where $
    \epsilon(N) = \epsilon_1(N) + \epsilon_2(N) + N^{-\kappa}\log\left(1+\frac{K^2}{\eta^2}\right) + t
$ and we have used 
\begin{equation}
\label{eqn:truncate_muN}
    \int \log_\eta^K(\lambda) \mu_N(\diff \lambda) \geq \int \log(\min(\abs{\lambda},K)) \mu_N(\diff \lambda) = \int \log\abs{\lambda} \mu_N(\diff \lambda)
\end{equation}
for $N$ large enough in the last inequality \eqref{eqn:detconlb_fundamental}, as the $\mu_N$'s are supported on a common compact set and $K$ grows with $N$.
It remains to check that $\epsilon(N) \to 0$. This follows immediately from our parameter choices \eqref{eqn:detcon_parameters}, except possibly for the term $\epsilon_2(N)$. For this term, we note that $\P(\mc{E}_{\textup{ss}}) \to 1$ and $\P(\mc{E}_{\textup{gap}}) \to 1$ by assumption (\eqref{eqn:ss} and \eqref{eqn:wegner}, respectively), then use Lemmas \ref{lem:prob_econc} and \ref{lem:prob_eb} to show that $\P(\mc{E}_{\textup{conc}}) \to 1$ and $\P(\mc{E}_b) \to 1$. This shows that $\epsilon_2(N) \to 0$, which concludes the proof of the lower bound and thus of \eqref{eqn:conv}. 
\end{proof}


\subsection{Proof of Theorem \ref{thm:concentrated_input}.}\
In this subsection we prove Theorem \ref{thm:concentrated_input}. The proof is largely similar to that of Theorem \ref{thm:convex_functional}, so we will omit some steps.

We make the same parameter choices as in \eqref{eqn:detcon_parameters}. We also work with the events $\mc{E}_{\textup{gap}}$ and $\mc{E}_b$ from \eqref{eqn:manyevents}, but $\mc{E}_{\textup{ss}}$ is no longer relevant, and $\mc{E}_{\textup{conc}}$ is replaced by
\[
    \mc{E}_{\textup{Lip}} = \left\{\abs{ \int \log_\eta(\lambda) (\hat{\mu}_{H_N} - \E[\hat{\mu}_{H_N}])(\diff \lambda)} \leq t\right\}.
\]

\begin{proof}[Proof of upper bound of Theorem \ref{thm:concentrated_input}]
From \eqref{eqn:lipschitztrace} and some elementary estimates, there exists a universal constant $c_{\epsilon_0}$ such that, for $N$ large enough, we have
\[
    \E[e^{N\int \log_\eta(\lambda) (\hat{\mu}_{H_N} - \E[\hat{\mu}_{H_N}])(\diff \lambda)}] \leq c_{\epsilon _0}\exp\left[ \left(\frac{2N^\zeta}{c_\zeta}\right)^{1/\epsilon_0} \left(\frac{1}{2\eta}\right)^2\right].
\]
Hence
\begin{align*}
    \frac{1}{N}\log \E[\abs{\det(H_N)}] &\leq \frac{1}{N}\log \E[e^{N \int \log_\eta(\lambda) \hat{\mu}_{H_N}(\diff \lambda)}] \leq \left(\frac{2}{4^{\epsilon_0} c_\zeta}\right)^{1/\epsilon_0} \frac{N^{\zeta/\epsilon_0-1}}{\eta^2} + \int \log_\eta(\lambda) \E[\hat{\mu}_{H_N}](\diff \lambda) \\
    &\leq \left(\frac{2}{4^{\epsilon_0} c_\zeta}\right)^{1/\epsilon_0} \frac{N^{\zeta/\epsilon_0-1}}{\eta^2} + \frac{1}{2\eta}{\rm W}_1(\E[\hat{\mu}_{H_N}],\mu_N) + \int \log_\eta(\lambda) \mu_N(\diff \lambda).
\end{align*}
For $\zeta$ small enough, the first term decays with $N$. We complete the proof by applying \eqref{eqn:wasserstein} and \eqref{eqn:detconub_logeta}. 
\end{proof}

\begin{proof}[Proof of lower bound of Theorem \eqref{thm:concentrated_input}]
Arguing as in \eqref{eqn:detconlb_fundamental}, $N^{-1}\log\E[\abs{\det(H_N)}]$ is larger than
\begin{align*}
    & \frac{1}{N}\log \E[e^{N \left( \int (\log\abs{\lambda}-\log_\eta(\lambda)) \hat{\mu}_{H_N}(\diff \lambda) + \int \log_\eta(\lambda) (\hat{\mu}_{H_N} - \E[\hat{\mu}_{H_N}])(\diff \lambda)\right)} \mathds{1}_{\mc{E}_{\textup{Lip}}} \mathds{1}_{\mc{E}_{\textup{gap}}}] + \int \log_\eta(\lambda) \E[\hat{\mu}_{H_N}](\diff \lambda) \\
    &\geq \frac{1}{N}\log \E\left[e^{N \int (\log\abs{\lambda} - \log_\eta(\lambda)) \hat{\mu}_{H_N}(\diff \lambda)} \mathds{1}_{\mc{E}_{\textup{Lip}}} \mathds{1}_{\mc{E}_{\textup{gap}}} \right] - t - \frac{1}{2\eta} {\rm W}_1(\E[\hat{\mu}_{H_N}], \mu_N) + \int \log\abs{\lambda} \mu_N(\diff \lambda).
\end{align*}
As in Lemma \ref{lem:eps2}, we have
\[
    \frac{1}{N}\log \E\left[e^{N \int (\log\abs{\lambda} - \log_\eta(\lambda)) \hat{\mu}_{H_N}(\diff \lambda)} \mathds{1}_{\mc{E}_{\textup{Lip}}} \mathds{1}_{\mc{E}_{\textup{gap}}} \mathds{1}_{\mc{E}_b} \right] \geq -\frac{p_b}{2}\log(1+e^{2N^\epsilon}\eta^2) - \frac{\eta^2}{2w_b^2} + \frac{1}{N}\log\P(\mc{E}_{\textup{Lip}}, \mc{E}_{\textup{gap}}, \mc{E}_b),
\]
so by our parameter choices \eqref{eqn:detcon_parameters} it suffices to show $\P(\mc{E}_{\textup{Lip}}, \mc{E}_{\textup{gap}}, \mc{E}_b) \to 1$. The event $\mc{E}_{\textup{gap}}$ is handled by assumption \eqref{eqn:wegner}; the event $\mc{E}_b$ is handled by Lemma \ref{lem:prob_eb} (replacing $d_{\textup{BL}}$ there with ${\rm W}_1$ here); and the event $\mc{E}_{\textup{Lip}}$ is handled by assumption \hyperlink{assn:L}{(L)}, since \eqref{eqn:lipschitztrace} gives
$
    \P(\mc{E}_{\textup{Lip}}^c) \leq \exp\left(-\frac{c_\zeta}{N^\zeta} \min\{(2Nt\eta)^2, (2Nt\eta)^{1+\epsilon_0}\}\right).
$
\end{proof}


\section{Applications to matrix models}
\label{sec:modelproofs}

In this section, we check the assumptions of our general theorems, \ref{thm:convex_functional} and \ref{thm:concentrated_input}, for our different matrix models. First we present two general and classical techniques that will help us check these assumptions. Informally speaking, the first technique shows how local laws for the Stieltjes transform along lines of the form $\{E + \ii N^{-\epsilon} : E \in [-C, C]\}$ give polynomial convergence rates of the averaged empirical spectral measure, corresponding to assumptions \hyperlink{assn:E}{(E)} and \hyperlink{assn:W}{(W)}. The second technique proves Wegner estimates of the form \eqref{eqn:wegner} using the Schur complement formula.

In the last Section \ref{sec:assumptions}, we prove the claims made just after Theorem \ref{thm:convex_functional} about the necessity of its assumptions.


\subsection{General technique: Convergence rates via local laws.}\
In this subsection, we summarize the general technique for using local laws to derive estimates like \eqref{eqn:dBL} and \eqref{eqn:wasserstein}. We will use this technique repeatedly for specific matrix models. This idea is classical; see for instance \cite{Bai1993} for the specific estimates we need.

Write $s_N(z)=\int\hat{\mu}_{H_N}(\diff \lambda)/(\lambda-z)$ for the Stieltjes transform of $\hat{\mu}_{H_N}$, and $m_N(z)=\int\mu_N(\diff \lambda)/(\lambda-z)$ for the Stieltjes transform of $\mu_N$. Define the distribution functions $F_{\E{\hat{\mu}}}(x) = \E[\hat{\mu}_{H_N}]((-\infty,x])$, $F_{\mu_N}(x) = \mu_N((-\infty,x])$.

\begin{prop}
\label{prop:bai}
Suppose the measures $\mu_N$ have densities $\mu_N(\cdot)$ on all of $\R$, not just near the origin, and $\sup_N \|\mu_N(\cdot)\|_{L^\infty} < \infty$.
Assume also that there exist fixed ($N$-independent) constants $A, \epsilon_1, \epsilon_2 > 0$ such that
\begin{align}
    \int_{-3A}^{3A} \abs{\E[s_N(E+\ii N^{-\epsilon_1})] - m_N(E+\ii N^{-\epsilon_1})} \diff E &\leq N^{-\epsilon_2}, \label{eqn:baibulk}\\
    \int_{\abs{x} > A} \abs{F_{\E[\hat{\mu}]}(x) - F_{\mu_N}(x)} \diff x &\leq N^{- \epsilon_1 - \epsilon_2}. \label{eqn:baitail}
\end{align}
Then there exists $\gamma > 0$ with
$
    d_{\textup{KS}}(\E[\hat{\mu}_{H_N}], \mu_N) =\OO(N^{-\gamma})
$. If in addition $\supp(\mu_N) \subset (-A,A)$ for each $N$, and 
\begin{equation}
\label{eqn:baidecay}
	\abs{F_{\E[\hat{\mu}]}(x) - F_{\mu_N}(x)} = o_{\abs{x} \to \infty} \left(\frac{1}{\abs{x}}\right),
\end{equation}
then there exists $\gamma' > 0$ with 
$
    d_{\textup{BL}}(\E[\hat{\mu}_{H_N}], \mu_N) \leq {\rm W}_1(\E[\hat{\mu}_{H_N}], \mu_N)=\OO(N^{-\gamma'}).
$
\end{prop}
\begin{proof}
From \cite[Theorem 2.2]{Bai1993}, we have
\begin{multline*}
    d_{\textup{KS}}(\E[\hat{\mu}_{H_N}], \mu_N) \leq\eta^{-1} \sup_x \int_{\abs{y} \leq 10\eta} \abs{F_{\mu_N}(x+y)-F_{\mu_N}(x)} \diff y 
    + 2\pi \eta^{-1} \int_{\abs{x} > A} \abs{F_{\E[\hat{\mu}]}(x) - F_{\mu_N}(x)} \diff x \\
    + \int_{-3A}^{3A} \abs{\E[s_N(E+\ii\eta)] - m_N(E+\ii\eta)} \diff E.
\end{multline*}
Since the measures $\mu_N$ have densities bounded by $S$, say, the function $F_{\mu_N}$ is $S$-Lipschitz; hence the first term is at most $100S\eta$. With the choice $\eta = N^{-\epsilon_1}$, the second and third terms are handled by assumption.

For the Wasserstein distance, let $f$ be a test function with $\|f\|_{\textup{Lip}} \leq 1$. We integrate by parts (notice \eqref{eqn:baidecay} gives us the decay at infinity necessary to do this) to find
\begin{align*}
    \abs{\int_{2A}^\infty f(x) (\E[\hat{\mu}_{H_N}] - \mu_N)(\diff x)} &= \abs{\int_{2A}^\infty f(x) \E[\hat{\mu}_{H_N}](\diff x)} \leq \int_{2A}^\infty (x-(2A-1)) \E[\hat{\mu}_{H_N}](\diff x) \\
    &\leq d_{\textup{KS}}(\E[\hat{\mu}_{H_N}], \mu_N) + \int_{2A}^\infty \abs{F_{\E[\hat{\mu}]}(x) - F_{\mu_N}(x)} \diff x \leq N^{-\gamma} + N^{-\epsilon_1-\epsilon_2}
\end{align*}
and similarly for the left tail. For the bulk, we approximate $f$ on $[-2A,2A]$ with test functions smooth enough to integrate by parts on $f$ directly, which gives
\[
    \abs{\int_{-2A}^{2A} f(x)(\E[\hat{\mu}_{H_N}] - \mu_N)(\diff x)} \leq (8A+4)d_{\textup{KS}}(\E[\hat{\mu}_{H_N}], \mu_N).
\]
This completes the proof.
\end{proof}


\subsection{General technique: Wegner estimates via Schur complements.}\
\label{subsec:wegner}
In this subsection, we summarize the classical idea of using the Schur-complement formula to derive Wegner estimates on the probability that there are no eigenvalues in a small gap around energy level $E$. These will be used to check \eqref{eqn:wegner} for a wide variety of models.

For compactness, we temporarily drop the $N$-dependence from the notation $H_N$. For any $j$ in $\llbracket 1, N \rrbracket$, write $H^{(j)}$ for the matrix obtained by erasing the $j$th column and row from $H$, write $h_j$ for the $(N-1)$-vector consisting of the $j$th column of $H$ with the entry $H_{jj}$ removed, and write $H_{\widehat{jj}}$ for the collection of every entry of $H$ except for $H_{jj}$.

\begin{prop}
\label{prop:schur}
Fix $E \in \R$ and suppose there exists a sequence $\eta = \eta_N$ tending to zero such that
\begin{equation}
\label{eqn:schurassumption}
    \sup_{j \in \llbracket 1, N \rrbracket}\E\left[ \E\left[ \left. \im \left( \frac{1}{H_{jj} - (E+\ii\eta + h_j^T(H^{(j)} - (E+\ii\eta))^{-1}h_j)} \right) \right| H_{\widehat{jj}} \right] \right] = o\left(\frac{1}{N\eta}\right).
\end{equation}
Then 
\[
    \lim_{N \to \infty} \P(H_N \text{ has no eigenvalues in } [E-\eta, E+\eta]) = 1.
\]
\end{prop}
\begin{proof}
We have
\begin{multline*}
	\P(H_N \text{ has an eigenvalue in } [E-\eta, E+\eta]) \leq \E[\#\{j : \abs{\lambda_j-E} \leq \eta \}] 
	  \leq \E\left[2\sum_{j=1}^N \frac{\eta^2}{\eta^2+(\lambda_j-E)^2}\right] \\= 2\eta \E\left[ \im\left( \sum_{j=1}^N \frac{1}{\lambda_j - E - \ii\eta} \right)\right] 
	= 2\eta\E\left[ \im\left( \Tr \frac{1}{H - (E+\ii\eta)} \right) \right] \leq 2N\eta \sup_{j \in \llbracket 1, N \rrbracket} \E[\im(((H-(E+\ii\eta))^{-1})_{jj})].
\end{multline*}
Moreover,  the Schur complement formula gives
\[
	((H-(E+\ii\eta))^{-1})_{jj} =  \frac{1}{H_{jj}-(E+\ii\eta + h_j^T(H^{(j)}-(E+\ii\eta))^{-1}h_j)},
\]
which concludes the proof by the assumption \eqref{eqn:schurassumption}. 
\end{proof}

\begin{lem}
\label{lem:schur_bdd_density}
Write $\widetilde{H}_{jj}$ for the law of $H_{jj}$ conditioned on $H_{\widehat{jj}}$. Suppose that there exists a single probability measure $\mu$ on $\R$ (independent of $N$ and $j$) with a bounded density $\mu(\cdot)$, and constants $\widetilde{\sigma}_{jj} = \widetilde{\sigma}_{jj}^{(N)}$ and $\widetilde{m}_{jj} = \widetilde{m}_{jj}^{(N)}$ such that 
\[
	\frac{\widetilde{H}_{jj}-\widetilde{m}_{jj}}{\widetilde{\sigma}_{jj}} \sim \mu
\]
for every $N$ and $j \in \llbracket 1, N \rrbracket$. If there exist $\alpha, C > 0$ with
\[
	\inf_{j \in \llbracket 1, N \rrbracket} \widetilde{\sigma}_{jj} \geq \frac{1}{C}N^{-\alpha},
\]
then \eqref{eqn:schurassumption} holds with $\eta = o(N^{-1-\alpha})$ for every $E \in \R$. 
\end{lem}
\begin{proof}
For any \emph{deterministic} $z = E+\ii\eta$, and with the notation $S:=\|\mu(\cdot)\|_{L^\infty}$, we have
\[
	\E_{\widetilde{H_{jj}}} \left[ \im\left(\frac{1}{\widetilde{H_{jj}} - z}\right) \right] = \int_{\R} \frac{\eta}{(\widetilde{\sigma}_{jj} x + \widetilde{m}_{jj} -E)^2 + \eta^2} \mu(x) \diff x \leq S\frac{1}{\widetilde{\sigma}_{jj}} \int_{\R} \frac{\eta}{x^2+\eta^2} \diff x \leq \pi SC N^{\alpha}.
\]
Define
$
    z_j = E+\ii\eta+ h_j^T(H^{(j)} - (E+\ii\eta))^{-1}h_j$, and
$    
    \widetilde{z_j} = z_j - \E[\widetilde{H}_{jj}],
$
and notice that $\widetilde{z_j}$ is measurable with respect to $H_{\widehat{jj}}$ with $\im(\widetilde{z_j}) \geq \eta$ deterministically; thus
\[
	\sup_{j \in \llbracket 1, N \rrbracket} \E\left[ \E \left[ \left. \im\left(\frac{1}{H_{jj} - z_j}\right) \right| H_{\widehat{jj}} \right] \right] = \sup_{j \in \llbracket 1, N \rrbracket} \E_{\widetilde{z_j}}\left[ \E_{\widetilde{H_{jj}}} \left[ \im \left(\frac{1}{\widetilde{H_{jj}} - \widetilde{z_j}} \right) \right] \right] \leq \pi SCN^{\alpha}
\]
which is $o(1/(N\eta))$ for our choice of $\eta$.
\end{proof}


\subsection{Wigner matrices.}\
\label{subsec:wigner}
We will use Theorem \ref{thm:convex_functional} (convexity-preserving functional) and model a Wigner matrix $W_N - E$ as $W_N - E = \Phi(X_1, \ldots, X_M)$, where $M = \frac{N(N+1)}{2}$, the $X_i$'s are independent random variables distributed according to $\mu$, and $\Phi$ is $\frac{1}{\sqrt{N}}$ times the identity map which places these entries in the upper triangle of an $N \times N$ matrix, minus $E\Id$. This $\Phi$ is trivially convex and satisfies $\|\Phi\|_{\textup{Lip}} = \frac{1}{\sqrt{N}}$.

Our assumption \eqref{eqn:subexponentialtails} that the underlying measure $\mu$ has subexponential tails is only used to check assumption \hyperlink{assn:S}{(S)}. To check the remaining conditions of Theorem \ref{thm:convex_functional}, we need only assume that $\mu$ has $2+\epsilon$ finite moments for some $\epsilon > 0$. In the interest of generality, in the following we give these minimal-assumptions proofs.

Now we check assumption \hyperlink{assn:E}{(E)} on expectations, with all $\mu_N$'s equal to the semicircle law $\rho_{\text{sc}}$. A. Tikhomirov \cite[Theorem 1.1]{Tik2009A} showed that for every $\epsilon$ in the assumption of $2+\epsilon$ finite moments, there exists $\eta = \eta(\epsilon) > 0$ with
\begin{equation}
\label{eqn:tikhomirov}
	d_{\textup{KS}}(\E[\hat{\mu}_{W_N}],\rho_{\text{sc}}) \leq N^{-\eta}.
\end{equation}
Now we transfer this inequality from $d_{\textup{KS}}$ to $d_{\textup{BL}}$: If $M > 2$ and $\|f\|_\infty \leq 1$, then
\[
	\abs{\int_{-\infty}^{-M} f(x) (\E[\hat{\mu}_{W_N}] - \rho_{\text{sc}})(\diff x)} = \abs{\int _{-\infty}^{-M} f(x) \E[\hat{\mu}_{W_N}](\diff x)} \leq \int_{-\infty}^{-M} \E[\hat{\mu}_{W_N}](\diff x) \leq N^{-\eta}
\]
from \eqref{eqn:tikhomirov}, and similarly for $\int_M^\infty$; on $[-M,M]$ we proceed exactly as in the proof of Proposition \ref{prop:bai}, to obtain \hyperlink{assn:E}{(E)}\footnote{We also briefly sketch another possible proof of assumption \hyperlink{assn:E}{(E)}. First, by following the usual Hoffman-Wielandt-based proof that two moments suffice for the Wigner semicircle law (see, e.g., \cite[Theorem 2.1.21]{AndGuiZei2010}), we can assume that the entries $W_{ij}$ are replaced with $W_{ij}\mathds{1}_{\abs{W_{ij}} \leq N^{10\epsilon}}$, if the $2+\epsilon$ moment is finite. Second, for this new matrix, one can apply the usual Stieltjes-transform-based proof of the Wigner semicircle law using Schur complements (see, e.g., \cite[Section 2.4.2]{AndGuiZei2010}); the fourth moments of the new matrix are $\OO(N^{40\epsilon})$, which is more than compensated by $1/N$ prefactors in the error terms.}.

Now we check the three estimates comprising assumption \hyperlink{assn:C}{(C)} on coarse bounds. 
\begin{itemize}
\item[\eqref{eqn:detconub_coarse}] Fix $\epsilon > 0$ and write $W = W_N = A + B = A_N + B_N$, where $A$ is defined entrywise by
$
    A_{ij} = (W_{ij})\mathds{1}_{\abs{W_{ij}} \leq \frac{1}{10N}e^{N^\epsilon}}.
$
Notice that all eigenvalues of $A$ have absolute value at most $\frac{1}{10}e^{N^\epsilon}$. The Weyl inequalities give us
\[
    \lambda_i(W) = \lambda_i(A+B) \leq \lambda_{\textup{max}}(A) + \lambda_i(B) \leq \frac{1}{10}e^{N^\epsilon} + \lambda_i(B)
\]
and similarly $\lambda_i(W) \geq \lambda_i(B) - \frac{1}{10}e^{N^\epsilon}$, so that for fixed $E$,  for large enough $N$ we have, for any $i$,
\begin{align*}
    1+\abs{\lambda_i(W-E)}\mathds{1}_{\abs{\lambda_i(W-E)} > e^{N^\epsilon}} &\leq 1+2\abs{\lambda_i(W)} \mathds{1}_{\abs{\lambda_i(W)} > \frac{1}{2}e^{N^\epsilon}} \leq 1 + 2\abs{\lambda_i(W)} \mathds{1}_{\abs{\lambda_i(B)} > \frac{1}{4}e^{N^\epsilon}} \\
    &\leq 1 + (\abs{\lambda_{\textup{max}}(A)} + \abs{\lambda_i(B)})\mathds{1}_{\abs{\lambda_i(B)} > \frac{1}{4}e^{N^\epsilon}} \leq 1 + 2\abs{\lambda_i(B)}\mathds{1}_{\abs{\lambda_i(B)} > \frac{1}{4}e^{N^\epsilon}}.
\end{align*}
For $x > 1$ we have $(1+2x) < (1+100x^2)^{1/2}$, so 
\[
    \prod_{i=1}^N (1+2\abs{\lambda_i(B)}\mathds{1}_{\abs{\lambda_i(B)} > \frac{1}{4}e^{N^\epsilon}}) \leq \prod_{i=1}^N (1+100\lambda_i(B)^2)^{1/2} = \det(\Id + 100 B^2)^{1/2}.
\]
By Fischer's inequality this can be bounded above by the product of its diagonal entries; that is,
\[
    \det(\Id + 100 B^2)^{1/2} \leq \prod_{i=1}^N \left( 1+100 \sum_{j=1}^N B_{ij}^2 \right)^{1/2} \leq \prod_{i=1}^N \left( 1 + 10 \sum_{j=1}^N \abs{B_{ij}} \right), 
\]
where for the last inequality we used $\sum a_i^2 \leq (\sum a_i)^2$ for positive numbers $a_i$. Now, for some constant $C$ we have
$
    \E[\abs{B_{ij}}], \E[\abs{B_{ij}}^2] \leq CNe^{-N^\epsilon} \leq e^{-\frac{1}{2}N^\epsilon},
$
and notice that we can calculate
$
    \E \left[ \prod_{i=1}^N \left( 1 + 10 \sum_{j=1}^N \abs{B_{ij}} \right) \right]
$
by expansion and factorization again. All matrix elements appear with a power at most two, and for any set $I$ of couples $(i,j)$ which can appear in the expansion, we have
$
    \E \left[ \prod_{\alpha \in I} \abs{B_\alpha} \right] \leq (e^{-\frac{1}{2}N^\epsilon})^{\abs{I}}
$
so that
\[
    \frac{1}{N}\log \E \left[ \prod_{i=1}^N \left( 1 + 10 \sum_{j=1}^N \abs{B_{ij}} \right) \right] \leq \frac{1}{N}\log \prod_{i=1}^N \left(1 + 10 \sum_{j=1}^N e^{-\frac{1}{2}N^\epsilon} \right) \to 0.
\]
\item[\eqref{eqn:wegner}] The existence of gaps near zero with high probability (indeed, gaps of polynomial size) was established by Nguyen \cite[Theorem 1.4]{Ngu2012}, including the case of general energy levels $E$. 
\item[\eqref{eqn:detconub_NlogN}] Fix $\delta$ so small that $\mu$ has finite $2+2\delta$ moment. Let $S_N$ be the symmetric group on $N$ letters, and for any permutation $\sigma \in S_N$ define $X_\sigma = \abs{(W-E)_{1,\sigma(1)} \cdot \ldots \cdot (W-E)_{N,\sigma(N)}}$. Then $\abs{\det(W_N-E)} \leq \sum_{\sigma} X_\sigma$, and by convexity of $x \mapsto x^{1+\delta}$ we have
\[
    \abs{\det (W_N-E)}^{1+\delta} \leq \left( \sum_{\sigma \in S_N} X_\sigma \right)^{1+\delta} \leq (N!)^{1+\delta} \frac{\sum_\sigma X_\sigma^{1+\delta}}{N!}.
\]
If $\sqrt{N}Y$ is distributed according to $\mu$, then for each $E \in \R$ there exists $c_E = c_E(\mu,\delta)$ such that
\[
    \max(\E[\abs{Y-E}^{1+\delta}], \E[\abs{Y-E}^{2+2\delta}], \E[\abs{Y}^{1+\delta}], \E[\abs{Y}^{2+2\delta}]) \leq c_E < \infty.
\]
Thus $\sup_\sigma \E[X_\sigma^{1+\delta}] \leq (c_E)^N$. Since $N!\leq N^N$, this gives
$
    \E[\abs{\det (W_N-E)}^{1+\delta}] \leq c_E^N N^{(1+\delta)N}
$
up to factors of lower order, which suffices.
\end{itemize}

To prove assumption \hyperlink{assn:S}{(S)} on spectral stability, we follow Bordenave, Caputo and Chafa\"{i}, see  \cite[Lemma C.2]{BorCapCha2011} and \cite[Lemma 2.2]{BorCap2014}. Write $W_N^{\textup{cut}} = \Phi(X_{\textup{cut}})$ for the matrix $W_N$ with entries truncated at level $N^{-\kappa}$ for some $\kappa < \frac{1}{2(2\alpha+1)}$, where $\alpha$ is from \eqref{eqn:subexponentialtails}. From interlacing (see, e.g., \cite[Theorem A.43]{BaiSil2010}) we find
\[
    d_{\textup{KS}}(\hat{\mu}_{W_N}, \hat{\mu}_{W_N^{\textup{cut}}}) \leq \frac{1}{N} \rank(W_N - W_N^{\textup{cut}}) \leq \frac{2}{N}\sum_{i \leq j} \mathds{1}_{\abs{W_{ij}} > N^{-\kappa}},
\]
where the last inequality follows since the rank of a matrix is at most the number of its nonzero entries. The $\frac{N(N+1)}{2}$ random variables $(\mathds{1}_{\abs{W_{ij}} > N^{-\kappa}})_{1 \leq i \leq j \leq N}$ are i.i.d. Bernoulli variables with parameter 
\[
    p_N = \P(\abs{W_{ij}} > N^{-\kappa}) \leq \beta \exp(-N^{(\frac{1}{2}-\kappa)\frac{1}{\alpha}}),
\]
from \eqref{eqn:subexponentialtails}. Writing $h(x) = (x+1)\log(x+1)-x$, Bennett's inequality \cite{Ben1965} gives
\[
    \P\left( \sum_{i \leq j} \mathds{1}_{\abs{W_{ij}} > N^{-\kappa}} - \frac{N(N+1)}{2} p_N \geq t \right) \leq \exp\left(-\sigma^2 h\left(\frac{t}{\sigma^2}\right)\right)
\]
with
\[
    \sigma^2 = \frac{N(N+1)}{2}p_N(1-p_N) \leq \frac{N(N+1)}{2} p_N \leq \beta \exp(-N^{(\frac{1}{2}-\kappa)\frac{1}{2\alpha}}),
\]
for $N$ large enough. With the choice $t = N^{1-\kappa} - \frac{N(N+1)}{2} p_N \geq \frac{1}{2}N^{1-\kappa}$ (for $\kappa$ small enough) we have $\frac{t}{\sigma^2} \to +\infty$, and using $h(x) \sim x \log x$ as $x \to +\infty$ (more precisely, $h(x) \geq \frac{1}{2}x\log x$ for $x$ large enough, say), we obtain 
\begin{align*}
    \log\P(d_{\textup{KS}}(\hat{\mu}_{W_N}, \hat{\mu}_{W_N^{\textup{cut}}}) > N^{-\kappa})
    &\leq -\sigma^2h\left(\frac{N^{1-\kappa}}{2\sigma^2}\right) \leq -C N^{1-\kappa}\log\left(\frac{N^{1-\kappa}}{2\sigma^2}\right) \leq -CN^{1-\kappa+(\frac{1}{2}-\kappa)\frac{1}{2\alpha}}
\end{align*}
for some constant $C$ and $N$ large enough. From our choice of $\kappa$, the last exponent is larger than one, which completes the proof of \eqref{eqn:ss}.


\subsection{Erd{\H{o}}s-R{\'e}nyi matrices.}\
We will use Theorem \ref{thm:convex_functional} (convexity-preserving functional) and model an Erd{\H{o}}s-R{\'e}nyi matrix $H_N-E$ as $H_N-E = \Phi(X_1, \ldots, X_M)$, where $M = \frac{N(N+1)}{2}$, the $X_i$'s are independent Bernoulli random variables with parameter $p_N$, and $\Phi$ is $\frac{1}{\sqrt{Np_N(1-p_N)}}$ times the identity map which places these entries in the upper triangle of an $N \times N$ matrix, minus $E\Id$. This clearly satisfies assumptions \hyperlink{assn:I}{(I)} and \hyperlink{assn:M}{(M)} with $\|\Phi\|_{\textup{Lip}} = \frac{1}{\sqrt{Np_N(1-p_N)}}$.

Now we verify assumption \hyperlink{assn:E}{(E)} with all $\mu_N$'s equal to the semicircle law $\rho_{\text{sc}}$. In the proof, we control the extreme eigenvalues (more precisely the smallest and second-largest) with results of Vu \cite{Vu2007}, improving on earlier results of F\"{u}redi-Koml\'{o}s \cite{FurKom1981}; and we control the bulk eigenvalues using the local law of Erd\H{o}s \emph{et al.} \cite{ErdKnoYauYin2013}. Often we use much weaker consequences of the results, replacing $\log N$ factors by polynomial factors and so on.

More precisely, consider $\widetilde{H_N} = H_N - \E[H_N]$. This matrix has centered entries of variance $\sigma^2 = \frac{1}{N}$, supported in $[-K,K]$ with $K = \frac{1}{\sqrt{\epsilon N^\epsilon}}$. Thus the proof of \cite[Theorem 1.3, Theorem 1.4]{Vu2007} shows that there exist $C, \gamma > 0$ with
\begin{equation}
\label{eqn:vu_erdosrenyi}
	\P\left(\|\widetilde{H_N}\| > 2 + C\frac{\log N}{(\epsilon N^\epsilon)^{1/4}}\right) \leq N^{-\gamma}
\end{equation}
for $N$ large enough. Recall we order eigenvalues as $\lambda_1 \leq \cdots \leq \lambda_N$; since $\E[H_N]$ is rank-one and positive semidefinite, interlacing tells us that $\max(\abs{\lambda_1(H_N)}, \abs{\lambda_{N-1}(H_N)}) \leq \|\widetilde{H_N}\|$, and thus we have the very coarse bound
\[
	\P(\max(\abs{\lambda_1(H_N)}, \abs{\lambda_{N-1}(H_N)}) \geq 3) \leq N^{-\gamma}
\]
for $N$ large enough. In particular, whenever $f$ is a test function with $\|f\|_\infty \leq 1$, we have
\[
	\abs{\int_3^\infty f(x) (\E[\hat{\mu}_{H_N}] - \rho_{\text{sc}})(\diff x)} \leq \frac{1}{N}\sum_{i=1}^N \P(\lambda_i(H_N) \geq 3) \leq \frac{1}{N} + N^{-\gamma},
\]
and similarly for the left tail, which is even easier because we do not need to separate out the smallest eigenvalue. 

Now we handle the bulk eigenvalues. Let $F_{\rho_{\text{sc}}}$, $F_{\hat{\mu}}$, and $F_{\E[\hat{\mu}]}$ be the distribution functions for $\rho_{\text{sc}}$, $\hat{\mu}_{H_N}$, and $\E[\hat{\mu}_{H_N}]$, respectively. Then \cite[Theorem 2.12]{ErdKnoYauYin2013} shows that there exists $\nu > 0$ such that, for $N$ large enough,
\[
	\P\left(\sup_{x \in [-3,3]} \abs{F_{\rho_{\text{sc}}}(x) - F_{\hat{\mu}}(x)} \leq N^{-1+\epsilon} \right) \geq 1-\exp(-\nu(\log N)^{5\log\log N}).
\]
Since $\sup_x \abs{F_{\rho_{\text{sc}}}(x) - F_{\hat{\mu}}(x)} \leq 2$ deterministically, this gives
\[
	\sup_{x \in [-3,3]} \abs{F_{\rho_{\text{sc}}}(x) - F_{\E[\hat{\mu}]}(x)} \leq \frac{N^\epsilon}{N} + 2\exp(-\nu(\log N)^{5\log \log N}).
\]
The proof of  \hyperlink{assn:E}{(E)} is then easily completed as in the case of Wigner matrices.

Now we check the three estimates comprising assumption \hyperlink{assn:C}{(C)} on coarse bounds.
\begin{itemize}
\item[\eqref{eqn:detconub_coarse}] We have
\begin{equation}
\label{eqn:erdosrenyi_coarse}
    \|H_N\|^2 \leq \sum_{i,j} \abs{H_{ij}}^2 \leq \frac{N}{p_N(1-p_N)} \leq \frac{1}{\epsilon} N^{2-\epsilon}
\end{equation}
almost surely, so \eqref{eqn:detconub_coarse} is trivially satisfied.
\item[\eqref{eqn:wegner}] For bulk energy levels, meaning $E \in (-2,2)$, one can show 
\[
	\P\left(H_N \text{ has no eigenvalues in } \left(E-\frac{1}{N^2}, E + \frac{1}{N^2}\right) \right) = 1 - o(1)
\]
using the bulk fixed-energy universality results of Landon-Sosoe-Yau \cite[Section 1.1.1]{LanSosYau2019}; the argument is given in our discussion below of the free-addition model. For $\abs{E} > 2$, eigenvalues other than $\lambda_N$ are handled with the result of Vu above \eqref{eqn:vu_erdosrenyi}. For $\lambda_N$ (only a concern for positive $E$ values), the Weyl inequalities give
\[
	\lambda_N(H_N) \geq \lambda_N(\E[H_N]) + \lambda_1(H_N - \E[H_N]) = \sqrt{\frac{Np_N}{1-p_N}} + \lambda_1(H_N - \E[H_N]) \geq \frac{1}{\sqrt{\epsilon}}N^{\epsilon/2} + \lambda_1(H_N - \E[H_N]).
\]
By \eqref{eqn:vu_erdosrenyi}, the last term is at least $-3$ with probability $1 - o(1)$; thus $\lambda_N$ cannot stick to any fixed $E > 2$.
\item[\eqref{eqn:detconub_NlogN}] This follows from \eqref{eqn:erdosrenyi_coarse}, using $\abs{\det(H_N-E)} \leq \|H_N-E\|^N$.
\end{itemize}

For assumption \hyperlink{assn:S}{(S)} on spectral stability, we note that the threshold for cutting is
\[
    \frac{N^{-\kappa}}{\|\Phi\|_{\textup{Lip}}} = N^{\frac{1}{2}-\kappa} \sqrt{p_N(1-p_N)} \geq \sqrt{\epsilon} N^{\frac{\epsilon}{2}-\kappa} \geq 1
\]
for $\kappa < \frac{\epsilon}{2}$ and large enough $N$. Since the $X_i=0$ or $1$, this means that $X = X_{\textup{cut}}$, and hence \eqref{eqn:ss} is trivially satisfied.


\subsection{$d$-regular matrices.}\ \label{sec:dreg}
We will prove Proposition \eqref{prop:dreg} by mimicking the proof of Theorem \ref{thm:concentrated_input}, but we will, informally speaking, prove \eqref{eqn:lipschitztrace} only for special test functions that we need to approximate the logarithm, rather than in full generality. Precisely, a careful reading of the proof of Theorem \ref{thm:concentrated_input} finds that it suffices to verify the following:
\begin{itemize}
\item Assumption \hyperlink{assn:W}{(W)} for some $\kappa > 0$, with all measures $\mu_N$ equal to the semicircle law $\rho_{\textup{sc}}$ (this is translation-invariant, so it suffices to check it at $E = 0$)
\item the Wegner estimate \eqref{eqn:wegner} around energy level $E$,
\item for the same $\kappa$, with the parameters
\begin{align*}
	\eta &= N^{-\kappa/2}, \\
	t &= N^{-\kappa/4},
\end{align*}
that (for the lower bound)
\begin{equation}
\label{eqn:dreg-lb}
	\P\left( \abs{ \int_\R \log_\eta(\lambda-E)(\hat{\mu}_{H_N} - \E[\hat{\mu}_{H_N}])(\diff \lambda)} \leq t\right) \to 1,
\end{equation}
\item and that (for the upper bound)
\begin{equation}
\label{eqn:dreg-ub}
	\limsup_{N \to \infty} \frac{1}{N} \log \E[e^{N \int \log_\eta(\lambda-E)(\hat{\mu}_{H_N} - \E[\hat{\mu}_{H_N}])(\diff \lambda)}] \leq 0.
\end{equation}
\end{itemize}
We now verify these four conditions.
\begin{itemize}
\item We will use Proposition \ref{prop:bai} and the local laws of Bauerschmidt-Knowles-Yau and Bauerschmidt-Huang-Knowles-Yau \cite{BauKnoYau2017, BauHuaKnoYau2020}. The former paper scales the $d$-regular adjacency matrix slightly differently from us; recalling that $H'_N$ is the adjacency matrix with entries in $\{0,1\}$, it considers
\[
	\widetilde{H_N} = \frac{1}{\sqrt{d-1}}\left(H'_N - \frac{d}{N}J\right)
\]
where $J$ is the $N \times N$ matrix of all ones. This normalization is close enough to ours (which was $H_N = \frac{1}{\sqrt{d(1-\frac{d}{N})}}H'_N$), in the sense that whenever $f : \R \to \R$ is $1$-Lipschitz, from Hoffman-Wielandt we have
\begin{align*}
	\abs{\int f(x) (\E[\hat{\mu}_{\widetilde{H_N}}] - \E[\hat{\mu}_{H_N}])(\diff x)} &\leq \E\left[ \frac{1}{N} \sum_{i=1}^N \abs{f(\lambda_i(\widetilde{H_N})) - f(\lambda_i(H_N))} \right] \\
	&\leq \E\left[ \frac{1}{\sqrt{N}} \left( \sum_{i=1}^N \abs{\lambda_i(\widetilde{H_N}) - \lambda_i(H_N)}^2\right)^{1/2} \right]\leq \E\left[\frac{1}{\sqrt{N}}\|\widetilde{H_N} - H_N\|_F\right]
\end{align*}
but deterministically we have
\begin{align*}
	\|\widetilde{H_N} - H_N\|_F &\leq \left\|\left( \frac{1}{\sqrt{d-1}} - \frac{1}{\sqrt{d(1-\frac{d}{N})}} \right) H'_N\right\|_F + \left\|\frac{d}{N\sqrt{d-1}} J\right\|_F = \OO(N^{\frac{1}{2}-\epsilon})
\end{align*}
by our choice of $d$. Thus
\[
	{\rm W}_1(\E[\hat{\mu}_{\widetilde{H_N}}],\E[\hat{\mu}_{H_N}]) = \OO(N^{-\epsilon}),
\]
and we can use Proposition \ref{prop:bai} to estimate ${\rm W}_1(\E[\hat{\mu}_{\widetilde{H_N}}], \rho_{\textup{sc}})$. Writing $\widetilde{s_N}(z)$ for the Stieltjes transform of $\hat{\mu}_{\widetilde{H_N}}$ and $m(z)$ for the Stieltjes transform of the semicircle law, a weaker consequence of \cite{BauKnoYau2017} shows that, for some absolute constant $C$,
\[
	\P\left(\text{there exists $z \in \C$ with $\eta \geq \frac{(\log N)^9}{N}$ such that } \abs{s_N(z) - m(z)} \geq C\sqrt{(\log N)^3\left(\frac{1}{\sqrt{N\eta}} + \frac{1}{N^{\epsilon/2}}\right)}\right) \leq e^{-(\log N)^3}.
\]
This suffices to check \eqref{eqn:baibulk}. Since $\hat{\mu}_{\widetilde{H_N}}$ is deterministically supported in $[-3\sqrt{d},3\sqrt{d}]$, \eqref{eqn:baidecay} is trivial. For \eqref{eqn:baitail}, with say $A = 10$, by separating out the largest eigenvalue we find
\[
	\int_{10}^\infty \abs{F_{\E[\hat{\mu}_{\widetilde{H_N}}]}(x) - F_{\rho_{\text{sc}}}(x)} \diff x = \int_{10}^{3\sqrt{d}} \E[\hat{\mu}_{\widetilde{H_N}}((x,\infty))] \diff x \leq \frac{4\sqrt{d}}{N} + 4\sqrt{d}\P(\lambda_{N-1}(\widetilde{H_N}) \geq 10).
\]
Now, with the same $\epsilon$ from $N^\epsilon \leq d \leq N^{2/3-\epsilon}$, (a weak consequence of) the edge-rigidity result of Bauerschmidt, Huang, Knowles, and Yau \cite[Theorem 1.1]{BauHuaKnoYau2020} gives
\begin{equation}
\label{eqn:bauhuaknoyau}
	\P\left(\max\left\{\abs{\lambda_{N-1}((d-1)^{-1/2}H'_N)-2},\abs{\lambda_1((d-1)^{-1/2}H'_N)+2}\right\} \geq 10N^{-\epsilon}\right) \leq N^{-1/\epsilon}.
\end{equation}
(This is still stronger than what we need, both in where it localizes the eigenvalues -- in a shrinking region around $\pm 2$ -- and in its right-hand side, which we only need to be $\OO(N^{-\epsilon'}d^{-1/2})$ for some $\epsilon'$.) Since $\lambda_{N-1}(\widetilde{H_N}) \leq \lambda_{N-1}((d-1)^{-1/2}H'_N)$ by interlacing, this suffices (along with analogous estimates at the left edge) to check \eqref{eqn:baitail}.
\item As in the Erd\H{o}s-R\'enyi case, for $E \in (-2,2)$ one can show
\[
	\P\left(H_N \text{ has no eigenvalues in } \left(E-\frac{1}{N^2}, E+\frac{1}{N^2}\right) \right) = 1-\oo(1)
\]
using the bulk fixed-energy universality results of Landon, Sosoe, and Yau \cite{LanSosYau2019}, with details given in the free-addition section below. For $\abs{E} > 2$, eigenvalues other than the largest one are handled with \eqref{eqn:bauhuaknoyau} (even with the slightly different normalization); the largest eigenvalue is deterministically $\sqrt{\frac{d}{1-\frac{d}{N}}} \geq \sqrt{d} \geq N^{\epsilon/2}$, i.e., cannot stick to any finite $E$.
\item We study $H_N$ by thinking of it as the adjacency matrix of an Erd\H{o}s-R\'enyi random graph, conditioned to be $d$-regular. Notice that the resulting law is indeed \emph{uniform} on $d$-regular matrices, since all graphs on $N$ vertices with a given number of edges are equiprobable under the Erd\H{o}s-R\'enyi measure. Precisely, write $\P_d$ for the law (and $\E_d$ associated expectation) of $H_N$ with the $d$-regular law \eqref{eqn:dreg-law}. Let $A'_N$ be the adjacency matrix of an Erd\H{o}s-R\'enyi random graph $G(N,\frac{d}{N})$, and consider the normalization
\[
    A_N = \frac{1}{\sqrt{d\left(1-\frac{d}{N}\right)}} A'_N
\]
whose law we denote as $\P_{ER}$ with corresponding expectation $\E_{ER}$. It will be convenient to use the notation
\[
	\mc{E}_{\textup{Lip}}(M_N,\delta) = \left\{ \abs{ \int_\R \log_\eta(\lambda-E)(\hat{\mu}_{M_N} - \E[\hat{\mu}_{M_N}])(\diff \lambda)} \leq \delta\right\}
\]
for $M_N = A_N$ or $M_N = H_N$, and for $\delta > 0$ (possibly depending on $N$). 

Now \cite[Lemma 2.1]{TraVuWan2013} shows that, if $d \to \infty$, then there exists a constant $C$ with
\begin{equation}
\label{eqn:dreg-conditioning}
    \P_{ER}(A_N \text{ is $d$-regular}) \geq \exp(-CN\sqrt{d}).
\end{equation}
Thus
\[
    \P_d((\mc{E}_{\textup{Lip}}(H_N,\delta))^c) = \P_{ER}( (\mc{E}_{\textup{Lip}}(A_N,\delta))^c \mid A_N \text{ is $d$-regular}) \leq e^{CN\sqrt{d}}\P_{ER}((\mc{E}_{\textup{Lip}}(A_N,\delta))^c).
\]
In the proof of Lemma \ref{lem:prob_econc} above, we wrote a decomposition $\log_\eta = \sum_{i=1}^3 \log_i$, where each $\log_i = \log_{i,\eta}$ was $\frac{1}{2\eta}$-Lipschitz and either convex or concave. Now
\begin{align*}
    \P_{ER}((\mc{E}_{\textup{Lip}}(A_N,\delta))^c) &= \P_{ER}\left(\abs{ \sum_{i=1}^3 \int_\R \log_i(\lambda-E)(\hat{\mu}_{A_N} - \E[\hat{\mu}_{A_N}])(\diff \lambda) } > \delta\right) \\
    &\leq \sum_{i=1}^3 \P_{ER}\left(\abs{ \int_\R \log_i(\lambda-E)(\hat{\mu}_{A_N} - \E[\hat{\mu}_{A_N}])(\diff \lambda) } > \frac{\delta}{3}\right)
\end{align*}
Since $\log_i(\cdot-E)$ is Lipschitz and convex (or concave), and since $\sqrt{N}A_N$ has entries compactly supported in, say, $[-\sqrt{\frac{2N}{d}},\sqrt{\frac{2N}{d}}]$, we can use concentration results of Guionnet and Zeitouni, namely \cite[Theorem 1.1(a)]{GuiZei2000}, which gives, for any (possibly $N$-dependent) $\delta > \delta_0(N) = 100N^{-\frac{1+\kappa-\epsilon}{2}}$, the estimate
\[
	\P_{ER}\left(\abs{\int_\R \log_i(\lambda-E)(\hat{\mu}_{A_N} - \E[\hat{\mu}_{A_N}])(\diff \lambda)} > \delta \right) \leq 4\exp\left(-\frac{dN\eta^2(\delta-\delta_0(N))^2}{32}\right).
\]
With the choice $\delta = \frac{t}{3} = \frac{1}{3}N^{-\kappa/4}$, this gives
\begin{equation}
\label{eqn:dreg-guizei-t}
    \P_d((\mc{E}_{\textup{Lip}}(H_N,t))^c) \leq 12\exp\left(CN\sqrt{d}-\frac{dN^{1-\frac{3\kappa}{2}}}{1000}\right).
\end{equation}
Since we can take $\kappa$ arbitrarily small, this tends to zero.
\item Since $\hat{\mu}_{H_N}$ is deterministically supported on $[-\frac{d}{\sqrt{d(1-d/N)}},\frac{d}{\sqrt{d(1-d/N)}}] \subset [-\sqrt{2d},\sqrt{2d}]$, we have
\[
	\abs{\int_\R \log_\eta(\lambda-E)(\hat{\mu}_{H_N} - \E[\hat{\mu}_{H_N}])(\diff \lambda)} \leq 2\max_{\abs{x} \leq \sqrt{2d}}\abs{\log_\eta(x-E)} \leq 10\log(N),
\]
almost surely, for $N \geq N_0(E)$. Thus
\begin{align*}
	&\E[e^{N \int \log_\eta(\lambda-E)(\hat{\mu}_{H_N}-\E[\hat{\mu}_{H_N}])(\diff \lambda)}] \\
	&=\E[e^{N \int \log_\eta(\lambda-E)(\hat{\mu}_{H_N}-\E[\hat{\mu}_{H_N}])(\diff \lambda)}\mathds{1}_{\mc{E}_{\textup{Lip}}(H_N,t)}] + \E[e^{N \int \log_\eta(\lambda-E)(\hat{\mu}_{H_N}-\E[\hat{\mu}_{H_N}])(\diff \lambda)}\mathds{1}_{(\mc{E}_{\textup{Lip}}(H_N,t))^c}] \\
	&\leq e^{Nt} + e^{10N\log(N)}\P((\mc{E}_{\textup{Lip}}(H_N,t))^c)
\end{align*}
From \eqref{eqn:dreg-guizei-t}, this is enough.
\end{itemize}


\subsection{Band matrices.}\
We will use Theorem \ref{thm:convex_functional} (convexity-preserving functional) and model a band matrix $H_N$ as $H_N = \Phi(X_1, \ldots, X_M)$, where $M = (W+1)N$, the $X_i$'s are independent random variable distributed according to $\mu$, and $\Phi$ is $\frac{1}{\sqrt{2W+1}}$ times the identity map which arranges these entries into a band matrix. This $\Phi$ is trivially convex and satisfies $\|\Phi\|_{\textup{Lip}} = \frac{1}{\sqrt{2W+1}}$. Throughout this section, the constant $\epsilon$ will be the same as in the assumption $W \geq N^\epsilon$. 

To check assumption \hyperlink{assn:E}{(E)} with $\mu_N\equiv \rho_{\text{sc}}$, we will use Proposition \ref{prop:bai} with $A = 3$. (By translation invariance, it suffices to check \hyperlink{assn:E}{(E)} at $E = 0$.) The bulk estimate \eqref{eqn:baibulk} follows from the stronger local law of Erd{\H{o}}s \emph{et al.} \cite{ErdYauYin2012}; the tail estimate \eqref{eqn:baitail} uses the tail estimates of Benaych-Georges/P{\'e}ch{\'e} \cite{BenPec2014}.

Write $s_N(z)$ for the Stieltjes transform of $\hat{\mu}_{H_N}$ and $m_{\text{sc}}(z)$ for the Stieltjes transform of the semicircle law. The local law \cite[Theorem 2.1]{ErdYauYin2012} gives constants $C$ and $c$ such that, if $z= E+\ii\eta$ with $E \leq 3$, $\kappa := \abs{\abs{E} - 2} \geq N^{-\delta}$ for $\delta = 3\epsilon/20$, and $\eta = N^{6\delta - \epsilon}$, then 
\[
    \P(\abs{s_N(z) - m_{\text{sc}}(z)} \geq N^{-\delta}) \leq CN^{-c(\log\log N)}. 
\]
Together with the trivial bound $\abs{\E[s_N(E+\ii\eta)] - m_{\text{sc}}(E+\ii\eta)} \leq \frac{2}{\eta}$, this gives
\[
    \abs{\E[s_N(z)] - m_{\text{sc}}(z)} \leq N^{-\delta} + 2CN^{\epsilon-6\delta-c(\log\log N)} \lesssim N^{-\delta}
\]
for such $z$ values. Writing $\epsilon_1 = \epsilon-6\delta > 0$ and using again the trivial bound for $\kappa < N^{-\delta}$, we obtain
\[
    \int_{-3}^3 \abs{\E[s_N(E+\ii N^{-\epsilon_1})] - m_{\text{sc}}(E+\ii N^{-\epsilon_1})} \diff E \lesssim 6N^{-\delta} + 8N^{\epsilon_1-\delta}.
\]
By our choice of $\delta$ we have $\epsilon_1 - \delta < 0$; this suffices to check \eqref{eqn:baibulk}.

For the tail estimate \eqref{eqn:baitail}, we note that $\abs{F_{\E[\hat{\mu}]}(x) - F_{\rho_{\text{sc}}}(x)} \leq \P(\|H_N\| \geq x)$ for, say, $x \geq 3$. The proof of \cite[Theorem 1.4]{BenPec2014} gives, for any $k \geq 1$, 
\[
    \P(\|H_N\| \geq x) \leq Nx^{-2k} 4^k \left(1-\frac{(12\alpha/e)^6 k^{12}}{W}\right)^{-1}.
\]
Choosing $k = k_N = N^{\epsilon/20}$, we verify \eqref{eqn:baidecay} and find, for $N$ large enough,
\[
    \int_3^\infty \P(\|H_N\| \geq x) \diff x \leq 6N^{1-\frac{\epsilon}{20}} \left(\frac{4}{9}\right)^{N^{\epsilon/20}},
\]
which is much faster than we need. The left tail is estimated similarly, and this verifies \eqref{eqn:baitail} and thus \eqref{eqn:dBL}. 

Now we check assumption \hyperlink{assn:C}{(C)} on coarse bounds.
\begin{itemize}
\item[\eqref{eqn:detconub_coarse}] The proof for Wigner matrices works verbatim here (in particular, $2+\epsilon$ finite moments is enough).
\item[\eqref{eqn:wegner}] Since we assumed our entries have a bounded density, this follows from Proposition \ref{prop:schur} and Lemma \ref{lem:schur_bdd_density}. 
\item[\eqref{eqn:detconub_NlogN}] The proof for Wigner matrices works verbatim here (in particular, $2+\epsilon$ finite moments is enough).
\end{itemize}

The proof of assumption \hyperlink{assn:S}{(S)} is similar to the case of Wigner matrices; in particular it holds assuming only that $\mu$ has $2+\epsilon$ finite moments.


\subsection{Sample covariance matrices.}\ \label{sec:covarianceMat}
As noted above, this model is not covered by either of our theorems directly. But it can be proved by mimicking the proof of Theorem \ref{thm:convex_functional} (convexity-preserving functional) with the following changes. We let $M = pN$, let $X_1, \ldots, X_M$ be independent copies of $\mu$, and consider the map $\Phi = \Phi_E : \R^M \to \ms{S}_p$ that places its arguments in the entries of the $p \times N$ matrix $Y = Y_{p,N}$ and returns $\frac{1}{N} YY^T - E$. There are two problems with applying Theorem \ref{thm:convex_functional} as written, but we will implement the following workarounds: 
\begin{enumerate}
\item $\Phi$ is not convex (but we will use the standard Hermitization trick that compares eigenvalues of $YY^T$ with eigenvalues of the $(p+N) \times (p+N)$ block matrix $(\begin{smallmatrix} 0 & Y \\ Y^T & 0 \end{smallmatrix})$, which \emph{is} a convex function of the entries of $Y$).
\item $\Phi$ is not Lipschitz, since it grows too quickly at infinity (but the Hermitization is Lipschitz).
\end{enumerate}

As in the Wigner case, the assumption of subexponential tails is only used to check assumption \hyperlink{assn:S}{(S)}, and we will give the remainder of the proofs only assuming that $\mu$ has $2+\epsilon$ finite moments.

Below, we will verify assumption \hyperlink{assn:E}{(E)} with some value of $\kappa$. For now, we redefine $X_{\textup{cut}}$ (for this model only), using this same $\kappa$, as 
\begin{equation}
\label{eqn:newxcut}
	(X_{\textup{cut}})_i = X_i \mathds{1}_{\abs{X_i} \leq N^{-\kappa+\frac{1}{2}}}.
\end{equation}
We choose this scaling so that each $\frac{1}{N}Y_{ij}^2$ is at most $N^{-2\kappa}$, similar to what happens in the Wigner and Erd{\H{o}}s-R{\'e}nyi cases. Later we will check assumption \hyperlink{assn:S}{(S)} with this new definition, as well as assumption \hyperlink{assn:C}{(C)}. First we show that all of these assumptions yield determinant concentration.

Much of the proof of Theorem \ref{thm:convex_functional} works verbatim in this new setting, since for example it never uses the old definition of $X_{\textup{cut}}$ directly, using instead the stability estimate \eqref{eqn:ss} which will still be true for us under the new definition. The biggest change is in the proof of Lemma \ref{lem:prob_econc}, where we applied results of Talagrand using the convexity and Lipschitz properties which no longer hold. The replacement for Lemma \ref{lem:prob_econc} is as follows. 

\begin{lem}
For every $E \in \R$, there exist $c_E, C_E > 0$ with the following properties: If we let $\widetilde{t_0}(N) = \frac{C_E}{\eta^2N^{\frac{1}{2}+\kappa}}$, then whenever $t \geq \widetilde{t_0}(N)$ we have
\[
	\P((\mc{E}_{\textup{conc}})^c) \leq 20\exp\left(-c_E(t-\widetilde{t_0}(N))^2\eta^4 N^{1+2\kappa}\right).
\]
\end{lem}
\begin{proof}
We use the classical trick of considering the $(p+N) \times (p+N)$ matrix
\[
	\mf{H}_N = \mf{H}_N(X) = \frac{1}{\sqrt{N}} \begin{pmatrix} 0_{p \times p} & Y_{p,N} \\ Y_{p,N}^T & 0_{N \times N} \end{pmatrix}.
\]
For any test function $f$, we have 
\begin{equation}
\label{eqn:blocktrick}
	\tr(f(\mf{H}_N^2)) = 2\tr(f(YY^T/N)) + (N-p)f(0).
\end{equation}
We show in Lemma \ref{lem:convex-log-decomposition} below that for every fixed $E \in \R$, there exists $C_E > 0$ such that the function $x \mapsto \log_\eta^K(x^2-E)$ can be decomposed as a sum of five functions
\[
	\log_\eta^K(x^2-E) = \sum_{i=1}^5 \widetilde{\log}_i(x) = \sum_{i=1}^5 \widetilde{\log}_{i,E,\eta,K}(x),
\]
where each $\widetilde{\log}_i$ is $\frac{C_E}{\eta^2}$-Lipschitz and either convex or concave (in fact, if $E \leq 0$ we only need three terms, and the last two will be set to zero). For simplicity, we choose powers $p_i \in \{0,1\}$ such that each $(-1)^{p_i}\widetilde{\log}_i$ is convex. From these we define functions $\widetilde{f_i} : [-N^{-\kappa+1/2}, N^{-\kappa+1/2}]^M \to \R$ given by
\[
	\widetilde{f_i}(X_{\textup{cut}}) = (-1)^{p_i} \frac{1}{N}\tr(\widetilde{\log}_i(\mf{H}_N(X_{\textup{cut}}))).
\]
Using \eqref{eqn:blocktrick} and mimicking the proof of Lemma \ref{lem:prob_econc}, we find
\begin{align*}
	\P(\mc{E}_{\textup{conc}}^c) = \P\left( \abs{ \frac{1}{2N}  \tr(\log_\eta^K(\mf{H}_N^2)) - \frac{1}{2N} \E[\tr(\log_\eta^K(\mf{H}_N^2))]} > t\right) \leq \sum_{i=1}^N \P\left( \abs{\widetilde{f_i}(X_{\textup{cut}}) - \E[\widetilde{f_i}(X_{\textup{cut}})]} > \frac{2}{5}t \right).
\end{align*}
As in the original proof, each $\widetilde{f_i}$ is $\sqrt{2(p+N)} \frac{1}{N} \frac{C_E}{\eta^2}\frac{1}{\sqrt{N}} \leq \frac{C_E}{N\eta^2}$-Lipschitz (for a new $C_E$), and since the map $X_{\textup{cut}} \mapsto \mf{H}_N$ is convex (this is the point of the Hermitization) we know that each $\widetilde{f_i}$ is convex as well. Then Talagrand's inequality gives
\[
	\P\left(\abs{\widetilde{f_i} - \mf{M}_{\widetilde{f_i}}} \geq t\right) \leq 4\exp\left(-\frac{t^2\eta^4N^{1+2\kappa}}{16(C_E)^2} \right)
\]
and we conclude as before.
\end{proof}

\begin{lem}
\label{lem:convex-log-decomposition}
For every $E \in \R$, there exists $C_E > 0$ with the following property: For every $\eta \leq \eta_0(E)$ and every $K \geq K_0(E)$, there exist functions $\widetilde{\log}_i = \widetilde{\log}_{i,E,\eta,K} : \R \to \R$ for $i = 1, 2, 3, 4, 5$ that are $\frac{C_E}{\eta^2}$-Lipschitz and either convex or concave, and such that
\[
	\log_\eta^K(x^2-E) = \sum_{i=1}^5 \widetilde{\log}_i(x).
\]
\end{lem}
\begin{proof}
The proof has two cases, according to whether $E \leq 0$ or $E > 0$.
\begin{itemize}
\item \textbf{Case $E \leq 0$:} Here we only need three functions (i.e., we set $\widetilde{\log}_4 = \widetilde{\log}_5 = 0$). One can check that there exists $0 \leq \sqrt{-E} < b = b^{(E,\eta)}$ such that the function $x \mapsto \log_\eta^K(x^2-E)$ is concave on $(-\infty,-b) \cup (b,\infty)$ and convex on $(-b,b)$. (Explicitly, it is given as the largest positive solution of the degree-six equation $2b^6 - 2Eb^4 - 2(E^2+3\eta^2)b^2 + 2E(E^2+\eta^2) = 0$, but the exact form does not matter so much as its stability for small $\eta$: As $\eta \downarrow 0$, $b \downarrow \sqrt{-E}$.) We want to pick our three functions (which we will actually call $\widetilde{\log}_1$ and $\widetilde{\log}_{2,\pm}$) in the form
\begin{align*}
	\widetilde{\log}_1(x) &= 
		\begin{cases}
			-c(x+b) + \log_\eta(b^2-E) & \text{if } x \leq -b \\ 
			\log_\eta(x^2-E) & \text{if } -b \leq x \leq b \\ 
			c(x-b) + \log_\eta(b^2-E) & \text{if } x \geq b
		\end{cases} \\
	\widetilde{\log}_{2,+}(x) &= 
		\begin{cases}
			cx & \text{if } x \leq b \\
			\log_\eta^K(x^2-E) + cb - \log_\eta(b^2-E) & \text{if } x \geq b
		\end{cases} \\
	\widetilde{\log}_{2,-}(x) &= \widetilde{\log}_{2,+}(-x)
\end{align*}
for some $c = c_{E,\eta} > 0$. It is easy to check that these functions sum to $\log_\eta^K(x^2-E)$, and they are all concave or convex as long as 
\[
	c \geq \partial_x \log_\eta(x^2-E)|_{x = b} = \frac{2b(b^2-E)}{\eta^2+(b^2-E)^2},
\]
in which case each function is $c$-Lipschitz; but since $\frac{2b(b^2-E)}{\eta^2+(b^2-E)^2} \leq \frac{2b(b^2-E)}{\eta^2} \leq \frac{C_E}{\eta^2}$ for small $\eta$, we are done. 
\item \textbf{Case $E > 0$:} This is similar to the case $E \leq 0$, except since the function is more complicated one needs five summands: One can check that there exist $0 < b_n = b_n^{(E,\eta)} < \sqrt{E} < b_f = b_f^{(E,\eta)}$ (for ``near boundary'' and ``far boundary'') such that the function $x \mapsto \log_\eta^K(x^2-E)$ is concave on $(-\infty,-b_f) \cup (-b_n,b_n) \cup (b_f,+\infty)$ and convex on $(-b_f,-b_n) \cup (b_n,b_f)$. (As before, they are given as the two positive solutions of the degree-six equation $2x^6 - 2Ex^4 - 2(E^2+3\eta^2)x^2 + 2E(E^2+\eta^2) = 0$, and they are stable for small $\eta$ in the sense that as $\eta \downarrow 0$, $b_n \uparrow \sqrt{E}$, and $b_f \downarrow \sqrt{E}$.) We want to pick our five functions (which we will actually call $\widetilde{\log}_1$, $\widetilde{\log}_{2,\pm}$, and $\widetilde{\log}_{3,\pm}$) in the form
\begin{align*}
	\widetilde{\log}_1(x) &= 
		\begin{cases}
			c_1(x+b_n) + \log_\eta(b_n^2-E) & \text{if } x \leq -b_n \\ 
			\log_\eta(x^2-E) & \text{if } -b_n \leq x \leq b_n \\ 
			-c_1(x-b_n) + \log_\eta(b_n^2-E) & \text{if } x \geq b_n
		\end{cases} \\
	\widetilde{\log}_{2,+}(x) &= 
		\begin{cases}
			-c_1(x-b_n) + \log_\eta(b_n^2-E) & \text{if } x \leq b_n \\ 
			\log_\eta(x^2-E) & \text{if } b_n \leq x \leq b_f \\
			c_3(x-b_f) + \log_\eta(b_f^2-E) & \text{if } x \geq b_f 
		\end{cases} \\
	\widetilde{\log}_{2,-}(x) &= \widetilde{\log}_{2,+}(-x) \\
	\widetilde{\log}_{3,+}(x) &= 
		\begin{cases}
			c_3(x-b_f) + \log_\eta(b_f^2-E) & \text{if } x \leq b_f \\
			\log_\eta^K(x^2-E) & \text{if } b_f \leq x
		\end{cases} \\
	\widetilde{\log}_{3,-}(x) &= \widetilde{\log}_{3,+}(-x) 
\end{align*}
for some $c_1, c_3 = c_{1,E,\eta},c_{3,E,\eta} > 0$. What are the constraints on these constants? To guarantee that all these functions are concave or convex, we need
\begin{align}
	c_1 &\geq \partial_x \log_\eta(x^2-E) |_{x = {-b_n}} = \frac{-2b_n(b_n^2-E)}{\eta^2+(b_n^2-E)^2}, \label{eqn:convex-log-decomposition-c1} \\
	c_3 &\geq \partial_x \log_\eta(x^2-E) |_{x = b_f} = \frac{2b_f(b_f^2-E)}{\eta^2+(b_f^2-E)^2}, \label{eqn:convex-log-decomposition-c3}
\end{align}
in which case each function is $\max(c_1,c_3)$-Lipschitz. Furthermore, the sum of these five functions is equal to $\log_\eta^K(x^2-E)$ as long as the constants satisfy the constraint
\begin{equation}
\label{eqn:convex-log-decomposition-freedom}
	c_1b_n + \log_\eta(b_n^2-E) - c_3b_f + \log_\eta(b_f^2-E) = 0,
\end{equation}
through which $c_1$ and $c_3$ determine each other: $c_3 = \frac{c_1b_n+\log_\eta(b_n^2-E)+\log_\eta(b_f^2-E)}{b_f}$. If we take $c_1$ very large to satisfy \eqref{eqn:convex-log-decomposition-c1}, then through \eqref{eqn:convex-log-decomposition-freedom} we find that $c_3$ must be very large, hence satisfies \eqref{eqn:convex-log-decomposition-c3}; thus it is possible to find \emph{some} pair $(c_1, c_3)$ satisfying all the constraints. Now we want pairs with small value, i.e., such that $\max(c_1,c_3) \leq\frac{C_E}{\eta^2}$: For small $\eta$ one can show that the lower bounds on the right-hand sides of \eqref{eqn:convex-log-decomposition-c1} and \eqref{eqn:convex-log-decomposition-c3} are upper-bounded by $\frac{100E^2}{\eta^2}$; and if we choose $c_1 = \frac{100E^2}{\eta^2}$, then through \eqref{eqn:convex-log-decomposition-freedom} we find that $c_3 \lesssim \frac{1}{\eta^2}$ as well, which finishes the proof.
\end{itemize}
\end{proof}

It remains to check assumptions \hyperlink{assn:E}{(E)}, \hyperlink{assn:C}{(C)}, and \hyperlink{assn:S}{(S)} (the latter under the new definition \eqref{eqn:newxcut}). The only assumption that is not translation-invariant (i.e., that depends on the energy level $E$) is assumption \hyperlink{assn:C}{(C)}.

For assumption \hyperlink{assn:E}{(E)},  \cite{Tik2009B} proved that if $\mu$ has $2+\epsilon$ moments then there exists (explicit) $\kappa(\epsilon) > 0$ such that
\[
	d_{\textup{KS}}(\E[\hat{\mu}_{\frac{1}{N}YY^T}], \mu_{\textup{MP},\frac{p_N}{N}}) \lesssim N^{-\kappa(\gamma)}.
\]
From this Kolmogorov-Smirnov distance information we evaluate $d_{\textup{BL}}$ in the same way as for Wigner matrices with $2+\epsilon$ moments. It remains only to understand $d_{\textup{BL}}(\mu_{MP, \frac{p_N}{N}}, \mu_{MP, \gamma})$, and this is only necessary in the case $\gamma < 1$ (since when $\gamma = 1$ we assumed $p_N = N$). If $\gamma_1, \gamma_2 \in [\epsilon,1-\epsilon]$, then  the difference between the densities gives
\[
	d_{\textup{BL}}(\mu_{\textup{MP}, \gamma_1}, \mu_{\textup{MP}, \gamma_2}) = O_\epsilon\left(\sqrt{\abs{\gamma_1-\gamma_2}}\right).
\]
Since we assumed in \eqref{eqn:wishartspeed} that $\abs{\frac{p_N}{N}-\gamma}$ is polynomially small, this suffices to prove \eqref{eqn:dBL}.

We check the three estimates of assumption \hyperlink{assn:C}{(C)} as follows:
\begin{itemize}
\item[\eqref{eqn:detconub_coarse}] This follows the proof of the Wigner case, but using the Weyl inequalities for singular values instead of those for eigenvalues. We write out the beginning of the argument because some of the powers change. For some $\epsilon > 0$, write $Y/\sqrt{N} = A + B$, where $A$ is defined entrywise by
\[
	A_{ij} = \frac{1}{\sqrt{N}}Y_{ij} \mathds{1}_{\frac{1}{\sqrt{N}} \abs{Y_{ij}} \leq \frac{1}{10N} e^{\frac{1}{2}N^\epsilon}}.
\]
Then $A$ has singular values at most $\frac{1}{10}e^{\frac{1}{2}N^\epsilon}$, and the Weyl inequalities give 
\[
	\sigma_i(Y/\sqrt{N}) \leq \sigma_{\textup{max}}(A) + \sigma_i(B) \leq \frac{1}{10}e^{\frac{1}{2}N^\epsilon} + \sigma_i(B)
\]
and similarly $\sigma_i(Y/\sqrt{N}) \geq \sigma_i(B) - \frac{1}{10}e^{\frac{1}{2} N^\epsilon}$, so that for each $i$ we have
\begin{align*}
	1+\abs{\lambda_i(YY^T/N - E)}\mathds{1}_{\abs{\lambda_i(YY^T/N - E)} > e^{N^\epsilon}} &\leq 1 + 2\lambda_i(YY^T/N) \mathds{1}_{\lambda_i(YY^T/N) > \frac{1}{2}e^{N^\epsilon}} \\
	&= 1+2\sigma_i^2(Y/\sqrt{N}) \mathds{1}_{\sigma_i(Y/\sqrt{N}) > \frac{1}{\sqrt{2}}e^{\frac{1}{2}N^\epsilon}} \leq 1+ 8\sigma_i^2(B) \mathds{1}_{\sigma_i(B) > \frac{1}{2}e^{\frac{1}{2}N^\epsilon}}.
\end{align*}
Then from Fischer's inequality we have
\[
	\prod_{i=1}^p (1+8\sigma_i^2(B)\mathds{1}_{\sigma_i(B) > \frac{1}{2}e^{\frac{1}{2}N^\epsilon}}) \leq \det(\Id + 8BB^T) \leq \prod_{i=1}^p \left(1+8\sum_{j=1}^N B_{ij}^2 \right).
\]
Since $B$ is non-Hermitian with independent entries, the same argument as in the Wigner case goes through here: when we expand and factor, each matrix entry appears at a power at most two.
\item[\eqref{eqn:wegner}] We mimic the proofs from Section \ref{subsec:wegner}, making the following changes. We closely follow the proof of some Wegner estimates for complex Wigner matrices from \cite[Theorem 3.4]{ErdSchYau2010}, as adapted in \cite[Proposition B.1]{BouErdYauYin2016} to the symmetric case.
Our estimates below will be coarser as we can afford any polynomial error, contrary to the optimal estimates from these references.
Let $E,\eta>0$, $\eta=\e/N$, $I=[E-\eta,E+\eta]$, $z=E+\ii\eta$ and $\mathcal{N}_I=|\{\mu_i\in I\}|$.  In the covariance matrix setting, the Schur complement formula gives, for any $1\leq j\leq N$ and defining $X=Y/\sqrt{N}$ and $H=YY^*/N$ (see e.g. \cite[Equation (3.8)]{BloErdKnoYauYin2014}) 
$$
((H-z)^{-1})_{ii}=\big(-z-z X_i^*R_i(z)X_i \big)^{-1}
$$
where we define $X_i=(X_{ij})_{j}$, $X^{(i)}_{jk}=X_{jk}\1_{j\neq i}$ and  $R_i(z)=((X^{(i)})^*X^{(i)}-z)^{-1}$. 
This implies, by the Cauchy-Schwarz inequality, 
$$
\E[\mathcal{N}_I^2\mathds{1}_{A}]\leq C(N\eta)^2\E\left[\Big(\im\frac{1}{-z-\frac{z}{N}\sum_{\alpha=1}^{N}\frac{\xi_\alpha}{\lambda_\alpha-z}}\Big)^2\mathds{1}_{A}\right]\leq C \e^2\E\left[\big(
(\sum_{\alpha=1}^Nc_\alpha \xi_\alpha)^2+(E-\sum_{\alpha=1}^{N} d_\alpha\xi_\alpha)^2
\big)^{-1}\mathds{1}_{A}\right]
$$
for any event $A$,
where
$$
d_\alpha=-\frac{1}{N}-\frac{N\lambda_\alpha(E-\lambda_\alpha)}{N^2(\lambda_\alpha-E)^2+\e^2},\ \ c_{\alpha}=\frac{\lambda_\alpha\e}{N^2(\lambda_\alpha-E)^2+\e^2}, \ \
$$
with $(\lambda_\alpha)_{1\leq \alpha\leq N-1}$ the eigenvalues of $(X^{(1)})^*X^{(1)}$, with corresponding $L^2$-normalized eigenvectors $u_\alpha$'s, and $\xi_\alpha=|u_\alpha\cdot Y_1|^2$. 

Let $(\gamma_k)_{1\leq k\leq N}$ be implicitly defined through  $\int_0^{\gamma_k} \mu_{\textup{MP},\gamma}(\diff x)=k/N$, with   $\mu_{\textup{MP},\gamma}$ 
from \eqref{eqn:MP}. If $E<\gamma_{\lfloor N/2\rfloor}$, we define $m=\lfloor 3N/4\rfloor$. If $E>\gamma_{\lfloor N/2\rfloor}$, let $m=\lfloor N/4\rfloor$.
Convergence of $N^{-1}\sum_{k=1}^N\delta_{\mu_k}$ ($\mu_1,\dots,\mu_N$ are the eigenvalues of $H$) to $\mu_{\textup{MP},\gamma}$  under the minimal assumption of finite second moment of the entries \cite{Wac1978} has the following elementary consequence: For any $c>0$,  $\mathbb{P}(A_N)=1-\oo(1)$ where $A_N=\cap_{N/7<k<8N/7}\{|\mu_k-\gamma_k|<c\}$. By interlacing, on $A_N$ the $(d_{m+\ell})_{0\leq \ell\leq 3}$ all have the same sign and absolute value greater than $N^{-2}$, and $c_{m},c_{m+1}>c\e/N^2$. Hence we can apply \cite[Equation (B.4)]{BouErdYauYin2016}\footnote{The assumption \eqref{eqn:smooth} is exactly the needed input for \cite[Lemma B.4]{BouErdYauYin2016}. Note that although this Lemma assumes $\mu$ has finite moments of all orders, this is actually not used in its proof.} with $\tau=0, r=p=2$ (and either $E$ or $-E$ depending on the sign of the $d_{m+\ell}$'s) to obtain, on $A_N$,
$$
\E_{Y_1}\left[\big(
(\sum_{\alpha=1}^Nc_\alpha \xi_\alpha)^2+(E-\sum_{\alpha=1}^{N-1} d_\alpha\xi_\alpha)^2
\big)^{-1}\right]
\leq
\frac{C}{\sqrt{c_{m}c_{m+1}}\min(d_{m+1},d_{m+2},d_{m+3})}\leq C \frac{N^{10}}{\e},
$$
so that $\E[\mathcal{N}_{I}^2\mathds{1}_{A_N}]\leq N^{10}\e$ and in particular $\mathbb{P}(\mathcal{N}_I\geq 1)\to 0$ for $\e=e^{-N^\e}$.
\item[\eqref{eqn:detconub_NlogN}] This proof has the same idea as the one for Wigner matrices; the only difference is that the product of entries associated to one permutation is estimated as follows. Fix $\delta$ so small that $\mu$ has finite $2+2\delta$ moment. For any permutation $\sigma \in S_p$ define $X_\sigma = \abs{(YY^T/N - E)_{1, \sigma(1)} \cdot \ldots \cdot (YY^T/N - E)_{p, \sigma(p)}}$. Let
\[
	c_\delta = \max(\E[\abs{Y_{1,1}}^{1+\delta}], \E[\abs{Y_{1,1}}^{2+2\delta}]) < \infty. 
\]
Then from convexity of $x \mapsto x^{1+\delta}$ we have
\[
	\left(\frac{1}{N^p} \sum_{j_1, \ldots, j_p = 1}^N \prod_{i=1}^p \abs{Y_{i,j_i}Y_{\sigma(i),j_i} - E\delta_{i,\sigma(i)}} \right)^{1+\delta} \leq \frac{1}{N^p} \sum_{j_1, \ldots, j_p = 1}^N \left( \prod_{i=1}^p \abs{Y_{i,j_i} Y_{\sigma(i),j_i} - E\delta_{i,\sigma(i)}} \right)^{1+\delta},
\]
and thus 
\begin{align*}
	\E[X_\sigma^{1+\delta}] &= \E\left[ \left( \prod_{i=1}^p \abs{\frac{1}{N} \sum_{j=1}^N (Y_{i,j} Y_{\sigma(i),j} - E\delta_{i,\sigma(i)})} \right)^{1+\delta} \right] \leq \frac{1}{N^p} \sum_{j_1, \ldots, j_p = 1}^N \E\left[ \left( \prod_{i=1}^p (\abs{Y_{i,j_i} Y_{\sigma(i), j_i}} + \abs{E} ) \right)^{1+\delta} \right] \\
	&=: \frac{1}{N^p} \sum_{j_1, \ldots, j_p = 1}^N \E[(Z_{j_1, \ldots, j_p})^{1+\delta}].
\end{align*}
Now each $Z_{j_1, \ldots, j_p}$ is the sum of $2^p$ terms, each of the form $\abs{E}^{p-k} \prod_{\ell=1}^k |Y_{i_\ell, j_{i_\ell}} Y_{\sigma(i_\ell), j_{i_\ell}}|$ for some $k \in \llbracket 1, p \rrbracket$ and some collection of \emph{distinct} integers $i_1, \ldots, i_k \in \llbracket 1, p \rrbracket$. Since they are distinct, each entry of the matrix $Y$ appears with power at most two in such a term; since these entries are independent, we have
\[
	\E\left[\left( \abs{E}^{p-k} \prod_{\ell=1}^k |Y_{i_\ell, j_{i_\ell}} Y_{\sigma(i_\ell), j_{i_\ell}}| \right)^{1+\delta}\right] \leq \abs{E}^{(p-k)(1+\delta)} c_\delta^{2k} \leq \max(\abs{E}, c_\delta, 1)^{2p(1+\delta)} =: c_{\delta,E}^{2p(1+\delta)} 
\]
Then Minkowski's inequality in $L^{1+\delta}$ gives
\[
	\E[X_\sigma^{1+\delta}] \leq \sup_{j_1, \ldots, j_p} \E[(Z_{j_1, \ldots, j_p})^{1+\delta}] \leq (2c_{\delta,E}^2)^{p(1+\delta)}.
\]
The rest of the proof is similar to the Wigner case. 
\end{itemize}

Finally we check assumption \hyperlink{assn:S}{(S)} with the new definition \eqref{eqn:newxcut}. Write $Y_{\textup{cut}} = \Phi(X_{\textup{cut}})$ for the $p \times N$ matrix $Y$ with entries truncated at level $N^{-\kappa+1/2}$; then it is classical that
\[
	d_{\textup{KS}}(\hat{\mu}_{YY^T/N}, \hat{\mu}_{Y_{\textup{cut}}Y_{\textup{cut}}^T/N}) \leq \frac{1}{p} \rank(Y - Y_{\textup{cut}})
\]
(this follows from interlacing of singular values; see, e.g., \cite[Theorem A.44]{BaiSil2010}). The rest of the argument with Bennett's inequality goes through from here; note that $\P(\abs{W_{ij}} > N^{-\kappa})$ and $P(\abs{Y_{ij}} > N^{-\kappa+1/2})$ are of similar order because $Y$ has order-one entries but the Wigner matrix $W$ has order-$\frac{1}{\sqrt{N}}$ entries.


\subsection{Gaussian matrices with a (co)variance profile.}\ \label{sec:GausCovProof}
We will use Theorem \ref{thm:concentrated_input} (concentrated input) to prove Corollary \ref{cor:GaussianCovarianceB}. First we need the following sequence of lemmas establishing consequences of our model assumptions (such as the log-Sobolev inequality and  tail decay estimates). 

\begin{lem}
\label{lem:covariance_lb}
Let $\mc{C} = \mc{C}_N$ be the covariance matrix of the upper triangle of $H_N$ considered as a Gaussian vector, i.e., $\mc{C}$ is an $\frac{N(N+1)}{2} \times \frac{N(N+1)}{2}$ matrix with entries
\[
    \mc{C}_{(i,j),(k,\ell)} = \Cov(H_{ij},H_{k\ell}) = \Cov(W_{ij}, W_{k\ell}).
\]
Let $p$ be as in the weak-fullness assumption \hyperlink{assn:wF}{(wF)}. Then, in the sense of quadratic forms,
\[
    \mc{C} \geq N^{-1-p}\Id.
\]
\end{lem}
\begin{proof}
We claim that
\begin{equation}
\label{eqn:fullness_decomp}
    W \overset{(d)}{=} N^{-\frac{p}{2}} W^{(\text{GOE})} + W'
\end{equation}
where $W_{\text{GOE}}$ is distributed as a GOE matrix (i.e., independent Gaussian entries up to symmetry with $\E[(W^{(\text{GOE})}_{ij})^2] = \frac{1+\delta_{ij}}{N}$) and where $W'$ is some real symmetric Gaussian matrix independent of $W^{(\text{GOE})}$. 

Indeed, consider the $N^2 \times N^2$ covariance matrix $\ms{C}_W$ of the \emph{full} matrix $W$ (not just the upper triangle). We will index this by matrix locations, i.e., $\ms{C}_W$ has entries $(\ms{C}_W)_{(i,j),(k,\ell)}$. Write $\ms{C}_{\text{GOE}}$ for the covariance matrix for GOE. We index a vector $B \in \R^{N^2}$ similarly, writing $B_{(i,j)}$, and associate with it the matrix $\tilde{B} \in \R^{N \times N}$ defined by
\[
    \tilde{B}_{ij} = B_{(i,j)}.
\]
Notice that the matrix $\tilde{B}$ need not be symmetric. Whenever $B$ has unit norm, we have
\begin{align*}
    \ip{B,\ms{C}_{\text{GOE}}B} &= \frac{1}{N} \sum_{i,j,k,\ell} B_{(i,j)} (\delta_{ik}\delta_{j\ell}+\delta_{i\ell}\delta_{jk}) B_{(k,\ell)} 
    = \frac{1}{N}\Tr(\tilde{B}\tilde{B}^{T} + \tilde{B}^2) = \frac{1}{N}\Tr\left(\left(\frac{\tilde{B}+\tilde{B}^{T}}{2}\right)^2\right).
\end{align*}
Thus by the weak-fullness assumption \hyperlink{assn:wF}{(wF)} we have
\begin{align*}
    \ip{B,\ms{C}_WB}
    = \E\left[ (\Tr(\tilde{B}W))^2 \right] = \E\left[ \left(\Tr\left(\left(\frac{\tilde{B}+\tilde{B}^{T}}{2}\right)W\right)\right)^2\right] \geq N^{-1-p} \Tr\left(\left(\frac{\tilde{B}+\tilde{B}^{T}}{2}\right)^2\right) = \ip{B,N^{-p}\ms{C}_{\text{GOE}}B}.
\end{align*}
To complete the proof of \eqref{eqn:fullness_decomp}, we write
\[
    \ms{C}_W = N^{-p}\ms{C}_{\text{GOE}} + (\ms{C}_W - N^{-p}\ms{C}_{\text{GOE}})
\]
and interpret the matrix in parentheses on the right-hand side, which we just showed is positive semi-definite, as the covariance matrix for $W'$.

Now we consider the $\frac{N(N+1)}{2} \times \frac{N(N+1)}{2}$ covariance matrix $\mc{C} = \mc{C}_W$ of the \emph{upper triangle} of $W$, and define $\mc{C}_{\text{GOE}}$ and $\mc{C}_{W'}$ similarly. Then whenever $v \in \R^{\frac{N(N+1)}{2}}$ is indexed with upper-triangular entries we have
\begin{align*}
    \ip{v,\mc{C}_Wv} &= \ip{v,N^{-p}\mc{C}_{\text{GOE}}v} + \ip{v,\mc{C}_{W'}v} \geq N^{-p} \ip{v,\mc{C}_{\text{GOE}}v} = N^{-1-p}\left(\sum_{i \leq j} v_{(i,j)}^2 + \sum_i v_{(i,i)}^2\right) \geq N^{-1-p} \|v\|_2^2
\end{align*}
which concludes the proof.
\end{proof}

\begin{lem}\label{lem:LSI}
For every $\zeta > 0$, there exists $c_\zeta > 0$ such that the law of the upper triangle of $H_N$, considered as a vector, satisfies the logarithmic-Sobolev inequality with constant $c_\zeta \frac{N^\zeta}{N}$.
\end{lem}
\begin{proof}
Since the logarithmic-Sobolev inequality is preserved under translations, it suffices to prove the statement with $H_N = W_N + \E[H_N]$ replaced by $W_N$. This is essentially an exercise in spelling out our model assumptions, which come from \cite{ErdKruSch2019}. 

The upper triangle of $W_N$ is a Gaussian vector with covariance matrix $\mc{C}$. Define the matrix $\abs{\mc{C}}$ by $\abs{\mc{C}}_{(i,j),(k,\ell)} = \abs{\mc{C}_{(i,j),(k,\ell)}}$, and whenever $u \in \R^{\frac{N(N+1)}{2}}$ is a unit vector, define the unit vector $\abs{u}$ by $\abs{u}_{(i,j)} = \abs{u_{(i,j)}}$. Then 
\[
    \ip{u,\mc{C}u} \leq \ip{\abs{u},\abs{\mc{C}}\abs{u}} \leq \| \abs{\mc{C}} \|.
\]
But our assumptions \hyperlink{assn:D}{(D)} on correlation decay imply that $\|\abs{\mc{C}}\| \leq_\zeta \frac{N^\zeta}{N}$; see \cite[(6b), Assumption (C)]{ErdKruSch2019}, specifically noting that $\vertiii{\kappa}_2^{\textup{av}}$ in their notation is the same as $N\|\abs{\mc{C}}\|$ in ours (the factor $N$ appears since their normalization is $H_N = A_N + \frac{1}{\sqrt{N}} W_N$ to our $H_N = A_N + W_N$).

Since $\mc{C}$ is invertible by Lemma \ref{lem:covariance_lb}, this implies the  log-Sobolev inequality via the Bakry-\'Emery criterion.
\end{proof}

\begin{lem}
\label{lem:mfv}
The flatness assumption \hyperlink{assn:F}{(F)} implies, for each $i, j, N$, 
\[
    \frac{1}{pN} \leq \Var((W_N)_{ij}) \leq \frac{p}{N}.
\]
\end{lem}
\begin{proof}
By writing $e_j$ for the $j$th canonical basis vector, understood as a column, and writing $(\cdot)^T$ for transposition, we have
$
	\E[W_{ij}^2] = \E[W_{ij}W_{ji}] = \E[We_j(e_j)^TW]_{ii} = (e_i)^T\E[We_j(e_j)^TW]e_i,
$
but by the flatness assumption \hyperlink{assn:F}{(F)} we have
$
	\frac{1}{pN} = \frac{1}{pN} \Tr(e_j(e_j)^T) \leq (e_i)^T\E[We_j(e_j)^TW]e_i \leq \frac{p}{N}\Tr(e_j(e_j)^T) = \frac{p}{N}.
$
\end{proof}

\begin{lem}
\label{lem:corr_L1}
We have $\sup_N \E[\|H_N\|] < \infty$. 
\end{lem}
\begin{proof}
Since we assumed $\sup_N \|A_N\| < \infty$, we need only check $\sup_N \E[\|W_N\|] < \infty$ where $W=W_N=H_N-\E(H_N)$. We apply the relevant local law  from \cite{ErdKruSch2019}. This local law provides a sequence of measures $\widetilde{\mu_N}$ which well-approximate the empirical measure of $W$. The exact form of $\widetilde{\mu_N}$ does not matter for our purpose; what does matter is \cite[Proposition 2.1, Equation (4.2)]{AjaErdKru2019}, which we combine to obtain $\supp(\widetilde{\mu_N}) \subset [-2\sqrt{2p}, 2\sqrt{2p}]$ uniformly in $N$. Then the local law  \cite[Corollary 2.3]{ErdKruSch2019} implies that eigenvalues of $W$ stick to $\supp(\widetilde{\mu_N})$ in the sense that, for some constant $C$, we have
\[
    \P(\|W\| \geq 2\sqrt{2p}+1) \leq CN^{-100}.
\]
Thus
\begin{align*}
    \E[\|W\|^2] &\leq (2\sqrt{2p}+1)^2 + \E[\|W\|^2 \mathds{1}_{\|W\| \geq 2\sqrt{2p}+1}] \leq (2\sqrt{2p}+1)^2 + \sqrt{\E[\|W\|^4] \P(\|W\| \geq 2\sqrt{2p}+1)}
\end{align*}
and the last term is $\oo(1)$ provided  $\E[\|W\|^4]$ satisfies some weak bound: Since the entries $W_{ij}$ are centered Gaussian with variance at most $\frac{p}{N}$ by Lemma \ref{lem:mfv}, H\"{o}lder's inequality gives
$$
    \E[\|W\|^4] \leq \E[\Tr(W^4)] \leq \sum_{i,j,k,\ell} \E[W_{ij}W_{jk}W_{k\ell}W_{\ell i}] 
    \leq \sum_{i,j,k,\ell} (\E[W_{ij}^4] \E[W_{jk}^4] \E[W_{k\ell}^4] \E[W_{\ell i}^4])^{1/4} \leq 3p^2N^2,
$$
which is sufficient. 
\end{proof}

\begin{lem}
\label{lem:corr_tails}
There exists $C$ such that, for every $t > 0$, we have
\[
	\P(\|H_N\| \geq t) \leq e^{-\sqrt{N}\max(t-C,0)}.
\]
\end{lem}
\begin{proof}
For definiteness,  we consider the logarithmic Sobolev inequality from Lemma \ref{lem:LSI} with constant $cN^{-1/2}$, $c=c_{1/2}$.  We apply Herbst's lemma with the map $H_N \mapsto\|H_N\|$, which is 
Lipschitz with constant $\sqrt{2}$ (by the Hoffman-Wielandt inequality),  to obtain for any $\alpha>0$
\[
	\E[e^{\alpha \|H_N\|}] \leq e^{\alpha \sup_N\E[\|H_N\|] + \frac{c}{2}N^{-1/2} \alpha^2}.
\]
To finish,  we bound $\E\|H_N\|$ with Lemma \ref{lem:corr_L1}, choose $\alpha = \sqrt{N}$, and apply Markov's inequality, so that the result applies for any $C>\sup_N\E[\|H_N\|]+c/2$.
\end{proof}

\begin{proof}[Proof of Corollary \ref{cor:GaussianCovarianceB}]
By the Herbst argument, Lemma \ref{lem:LSI} implies assumption \hyperlink{assn:L}{(L)} on Lipschitz concentration.

We now check assumption \hyperlink{assn:W}{(W)}, with the measures $\mu_N$ given as the solutions of the Matrix Dyson Equation. Most of this argument consists of importing results of Ajanki \emph{et al.} and Erd{\H o}s \emph{et al.} Indeed, combining \cite[Proposition 2.1, Equation (4.2)]{AjaErdKru2019}, we find that the supports of the measures $\mu_N$ satisfy
\begin{equation}
\label{eqn:gauss_cov_bounded_support}
    \supp(\mu_N) \subseteq (-(\|A_N\| + 2\sqrt{2p}), \|A_N\| + 2\sqrt{2p}).
\end{equation}
Since the right-hand side is uniformly bounded in $N$, so is the left-hand side. Furthermore, \cite[Proposition 2.2]{AjaErdKru2019} shows that each $\mu_N$ admits a density $\mu_N$ with respect to Lebesgue measure (on all of $\R$), and that these densities are $c$-H\"{o}lder continuous for some universal $c$; hence they are bounded, uniformly in $N$.

To check \eqref{eqn:wasserstein}, we use Proposition \ref{prop:bai}. Write $s_N$ for the (random) Stieltjes transform of $\hat{\mu}_{H_N}$. For the Stieltjes-transform estimate \eqref{eqn:baibulk}, we use the local law \cite[Theorem 2.1(4b)]{ErdKruSch2019}, which implies that there exists a universal constant $c$ such that, for every sufficiently small $\epsilon > 0$, there exists $C_\epsilon > 0$ with
\[
    \P\left( \abs{s_N(E+\ii N^{-c\epsilon}) - m_N(E+\ii N^{-c\epsilon})} \geq N^{\epsilon(1+2c)-1} \text{ for some }\abs{E} \leq N^{100} \right) \leq C_{\epsilon}N^{-100}.
\]
Using the trivial bound $\frac{1}{\eta}$ for a Stieltjes transform evaluated at $E+\ii\eta$, we obtain
\begin{align*}
    \abs{\E[s_N(E+\ii N^{-c\epsilon})] - m_N(E+\ii N^{-c\epsilon})} \leq N^{\epsilon(1+2c)-1} + 2C_{\epsilon}N^{c\epsilon-100}
\end{align*}
for all $\abs{E} \leq N^{100}$, which suffices to check \eqref{eqn:baibulk}.
Moreover,  if $x > \max \supp(\mu_N)$ we have
\[
	\abs{F_{\E[\hat{\mu}]}(x) - F_{\mu_N}(x)} = 1-F_{\E[\hat{\mu}]}(x) \leq \P(\|H_N\| > x) \leq e^{-\sqrt{N}\max(x-C,0)}
\]
from Lemma \ref{lem:corr_tails}, and similarly for the left edge, which gives \eqref{eqn:baitail} and \eqref{eqn:baidecay}. This verifies assumption \hyperlink{assn:W}{(W)}.

Finally we check the Wegner estimate,  with the general Schur-complement strategy. Recall we wrote $\mc{C}$ for the covariance matrix of the upper triangle of $H = H_N$ (we will drop the subscript $N$ for the remainder of this proof). Now we will write $\mc{C}_{\widehat{jj}}$ for its minor obtained by erasing the column and row corresponding to $H_{jj}$. Since $\mc{C}$ is invertible by Lemma \ref{lem:covariance_lb} (and positive semidefinite), so is its minor $\mc{C}_{\widehat{jj}}$ by interlacing. Conditioned on $H_{\widehat{jj}}$, we have that $H_{jj}$ is a Gaussian random variable with (an explicit mean that does not matter now and) variance
\[
    \left(\widetilde{\sigma_{jj}}\right)^2 := \Var(H_{jj}) - \sum_{\substack{k \leq \ell, k' \leq \ell' \\ (k,\ell) \neq (j,j) \neq (k',\ell')}} \mc{C}_{(j,j),(k,\ell)}  ((\mc{C}_{\widehat{jj}})^{-1})_{(k,\ell),(k',\ell')} \mc{C}_{(k',\ell'),(j,j)} = \frac{1}{(\mc{C}^{-1})_{jj}} \geq \lambda_{\min{}}(\mc{C}) \geq N^{-1-p},
\]
where we used Lemma \ref{lem:covariance_lb} in the last step. By Lemma \ref{lem:schur_bdd_density} and Proposition \ref{prop:schur}, this proves \eqref{eqn:wegner}.
\end{proof}


\subsection{Block-diagonal Gaussian matrices.}\ As in subsection \ref{sec:GausCovProof},
we will use Theorem \ref{thm:concentrated_input} (concentrated input). Considered as a vector, the upper triangle of $H_N$ satisfies log-Sobolev with constant $\frac{p}{N}$, since it consists of independent (possibly degenerate) Gaussians with variance at most $\frac{p}{N}$. This implies the Lipschitz-concentration assumption \hyperlink{assn:L}{(L)}.

Now we check assumption \hyperlink{assn:W}{(W)}. We assumed in \hyperlink{assn:R}{(R)} that the MDE measures $\mu_N$ have a bounded density; they lie in a common compact set by the estimate \cite[(3.32a)]{AltErdKruNem2019} and arguments like those around \eqref{eqn:gauss_cov_bounded_support}, so it remains only to check \eqref{eqn:wasserstein}, through Proposition \ref{prop:bai}. If $s_N$ denotes the Stieltjes transform of $H_N$, then the local law \cite[(B.5)]{AltErdKruNem2019} implies that there exist universal constants $\delta > 0$ and $P \in \N$ such that, for every $0 < \gamma < \delta$, there exists $C_\gamma$ with
\[
	\P\left(\abs{s_N(E+iN^{-\gamma}) - m_N(E+iN^{-\gamma})} \geq \frac{N^{\gamma P}}{N} \text{ for some } E \in \R\right) \leq C_\gamma N^{-100}.
\]
For the tail estimate \eqref{eqn:baitail}, we essentially mimic the proof in the case of Gaussian matrices with a (co)variance profile, with the following differences: Here the estimate $\sup_N\E[\|W_N\|^2] <\infty$ is easier, since (recall that $W_N$ is block-diagonal with blocks $X_1, \ldots, X_K$) we have $\E[\|W_N\|^2]^{1/2} \leq \sum_{i=1}^K \E[\|X_i\|^2]^{1/2}$, and it is classical that $\sup_N \E[\|X_i\|^2] < \infty$ since $X_i$ is a Gaussian matrix whose entries all have variance order $\frac{1}{N}$, by assumption \hyperlink{assn:MF}{(MF)}. Since the log-Sobolev constant is now at most $p/N$, we obtain $\P(\|H_N\| \geq t) \leq e^{-cN\max(0,t-C)}$ for some constants $c,C>0$, which verifies \eqref{eqn:baitail} and \eqref{eqn:baidecay}. This completes the proof of \eqref{eqn:wasserstein}.

Finally we check the Wegner estimate \eqref{eqn:wegner} with Proposition \ref{prop:schur}. Here Lemma \ref{lem:schur_bdd_density} applies immediately, since the conditioning is trivial, and we assumed in \hyperlink{assn:MF}{(MF)} that the variances on the diagonal are all at least of order $\frac{1}{N}$.


\subsection{Free addition.}\
We will use Theorem \ref{thm:concentrated_input} (concentrated input). Write $H_N = E + A_N + O_NB_NO_N^T$. Concentration for Lipschitz test functions follows from classical results of Gromov-Milman: If $S = E + \sup_{N \geq 1}(\|A_N\| + \|B_N\|)$ and $f : \R \to \R$ is Lipschitz, then (see, e.g., \cite[Corollary 4.4.30]{AndGuiZei2010})
\[
    \P\left( \abs{\frac{1}{N}\Tr(f(H_N)) - \frac{1}{N}\E[\Tr(f(H_N))]} \geq \delta\right) \leq 2\exp\left(-\frac{\delta^2N^2}{128S^2\|f\|_{\textup{Lip}}^2}\right),
\]
which suffices to check \eqref{eqn:lipschitztrace} and thus assumption \hyperlink{assn:L}{(L)}.

For assumption \hyperlink{assn:E}{(E)} with the reference measure $\mu_N \equiv \mu_A \boxplus \mu_B$, we will use the local law of Bao \emph{et al}, \cite[Corollary 2.8]{BaoErdSch2020}:  for every $\epsilon > 0$ and all $N \geq N_0(\epsilon)$, we have
\[
    \P\left(d_{\textup{KS}}(\hat{\mu}_{H_N}, \mu_A \boxplus \mu_B) > N^{-1+\epsilon} \right) \leq N^{-100}.
\]
This implies $d_{\textup{KS}}(\E[\hat{\mu}_{H_N}],\mu_A \boxplus \mu_B) \lesssim N^{-1+\epsilon}$. We obtain the same estimate for ${\rm W}_1$ as in the proof of Proposition \ref{prop:bai} (there are no tail estimates because all the measures $\hat{\mu}_{H_N}$ and $\mu_A \boxplus \mu_B$ are supported on a common compact set).

It remains only to check the Wegner estimate \eqref{eqn:wegner}. The argument is different depending if $E$ is in the bulk of $\mu_A \boxplus \mu_B$ (meaning in the interior of the single-interval support), or if $E$ is outside the support. In the first case, we prove the Wegner estimate with the much stronger fixed-energy universality results of Che-Landon \cite[Theorem 2.1]{CheLan2019}. This result implies
\[
    \lim_{N \to \infty} \P\left(H_N \text{ has no eigenvalues in } \left(E-\frac{\epsilon}{N(\mu_A \boxplus \mu_B)(E)}, E+\frac{\epsilon}{N(\mu_A \boxplus \mu_B)(E)}\right)\right) = 1-F(\epsilon),
\]
where $F(\epsilon)$ is a special function found by solving the Painlev{\'e} V equation satisfying $\lim_{\epsilon \downarrow 0} F(\epsilon) = 0$. Thus
\[
    \liminf_{N \to \infty} \P\left(H_N \text{ has no eigenvalues in } \left(E-\frac{1}{N^2}, E+\frac{1}{N^2}\right)\right) \geq 1 - \limsup_{\epsilon \downarrow 0} F(\epsilon) = 1.
\]
In the second case (if $E$ is outside the support of $\mu_A \boxplus \mu_B$), the Wegner estimate is much easier, since indeed $\P(\text{no eigenvalues in } (E-\delta,E+\delta)) \to 1$ for small enough $\delta$. This follows, e.g., from the large-deviations principle for the extremal eigenvalues of this model established by Guionnet and Ma{\"i}da \cite{GuiMai2020}, or from the edge rigidity of Bao \emph{et al.} \cite{BaoErdSch2020}.


\subsection{Proofs of examples showing necessity of assumptions.}\
\label{sec:assumptions}
In this subsection we show the importance of two of the trickier assumptions of Theorem \ref{thm:convex_functional}. Precisely, for each of \eqref{eqn:detconub_coarse} and assumption \hyperlink{assn:S}{(S)}, we give an explicit example satisfying all the assumptions of that theorem except for the one in question, for which the conclusion fails. All notations refer back to that section.

Our example where \eqref{eqn:detconub_coarse} fails and determinant concentration fails is the following: Let $(X_{ij})_{1 \leq i \leq j \leq N}$ be centered, i.i.d. with variance $1$ and a compactly supported and bounded density; choose some $\theta\in(0,1)$ (e.g. $\theta = 1/8$ works) and let $A$ be deterministic, diagonal, and defined through $A_{ii} = e^{N^\theta} \mathds{1}_{i < N^{1-\theta}}$ with all other entries zero; and define symmetric $H = \Phi(X)$ as $H_{ij} = \Phi(X)_{ij} = \frac{X_{ij}}{\sqrt{N}} + A_{ij}$ for $i \leq j$. In this example, $\mu_N = \rho_{\text{sc}}$. 

Our example where assumption \hyperlink{assn:S}{(S)} fails and determinant concentration fails is the following: Let $(X_{ij})_{1 \leq i \leq j \leq N}$ be as above, include the additional random variable $X_0$ with $\P(X_0 = N) = N^{-1} = 1 - \P(X_0 = 0)$, and define $A = X_0\Id_N$; then we let $H = \Phi(X)$ be symmetric defined by $H_{ij} = \Phi(X)_{ij} = \frac{X_{ij}}{\sqrt{N}} + A_{ij}$ for $i \leq j$. In this example, $\mu_N = \rho_{\text{sc}}$. 

In the remainder of this subsection, we prove that these examples have the claimed properties. 

\subsubsection{Necessity of bounds on large eigenvalues.}\

Write the compact support of the $X_i$'s as $[-T,T]$. This proof is essentially an application of the Weyl inequalities. Note that $\|\Phi\|_{\textup{Lip}} = N^{-1/2}$; since the $X_i$'s are compactly supported, this means $X_{\textup{cut}} = X$ for $\kappa < 1/2$ and $N$ large enough, and hence \hyperlink{assn:S}{(S)} is trivially satisfied. Equation \eqref{eqn:wegner} holds by Lemma \ref{lem:schur_bdd_density}. If $\kappa < \theta$, then \hyperlink{assn:E}{(E)} holds with $\mu_N = \rho_{\text{sc}}$ by interlacing; indeed, defining the matrix $G$ by $G_{ij} = \frac{X_{ij}}{\sqrt{N}}$, we have $d_{\textup{KS}}(\hat{\mu}_G, \hat{\mu}_H) \leq \frac{1}{N} \rank(A) = N^{-\theta}$. Since $G$ is a Wigner matrix with all moments finite, \cite[Theorem 4.1]{Bai1993} shows $d_{\textup{KS}}(\E[\hat{\mu}_G], \rho_{\text{sc}}) \leq N^{-1/4}$, and thus if $\theta < 1/4$ we have
\[
	d_{\textup{KS}}(\E[\hat{\mu}_H], \rho_{\text{sc}}) \lesssim N^{-\theta}.
\]
We transfer this from $d_{\textup{KS}}$ to $d_{\textup{BL}}$ in the same way as for Wigner matrices, above. For \eqref{eqn:detconub_NlogN}, the Weyl inequalities give deterministically
\[
	\abs{\det(H_N)} = \prod_{i=1}^N \abs{\lambda_i(H_N)} \leq \prod_{i=1}^N (\lambda_i(A) + T\sqrt{N}) = (e^{N^\theta}+T\sqrt{N})^{N^{1-\theta}}(T\sqrt{N})^{N-N^{1-\theta}} \leq (2e^{N^\theta})^{N^{1-\theta}} (T\sqrt{N})^N
\]
which suffices to check \eqref{eqn:detconub_NlogN} (with any $\delta > 0$). 

On the other hand,  \eqref{eqn:detconub_coarse} fails. Indeed, by the Weyl inequalities the $N^{1-\theta}$ large eigenvalues of $H$ satisfy
\begin{equation}
\label{eqn:example_large_evals}
	\lambda_i \geq e^{N^\theta} - T\sqrt{N} \geq \frac{1}{2}e^{N^\theta},
\end{equation}
so for $\epsilon < \theta$ the failure of \eqref{eqn:detconub_coarse}  follows from the deterministic estimate
\[
	\prod_{i=1}^N (1+\abs{\lambda_i}\mathds{1}_{\abs{\lambda_i} > e^{N^\epsilon}}) \geq \prod_{i=N-N^{1-\theta}+1}^N \abs{\lambda_i}\mathds{1}_{\abs{\lambda_i} > e^{N^\epsilon}} \geq \left(\frac{1}{2}e^{N^\theta}\right)^{N^{1-\theta}} = 2^{-N^{1-\theta}}e^N.
\] 

The proof that determinant concentration fails is somewhat involved, but mimics the proof of the lower bound of Theorem \ref{thm:convex_functional}. The idea is that the largest $N^{1-\theta}$ eigenvalues contribute a factor of size $e^N$, as above, and the rest of the eigenvalues behave as if semicircular (this is the difficulty), so we get a lower bound for the determinant asymptotics that is order-one above what the semicircle would predict. We now sketch how to prove this rigorously. Since $X = X_{\textup{cut}}$, we simplify our notation and write $\hat{\mu} = \hat{\mu}_{\Phi(X)}$. Recall our eigenvalues are ordered $\lambda_1 \leq \cdots \leq \lambda_N$; we decompose this measure as
\[
	\hat{\mu} = \hat{\mu}^{\textup{trunc}} + \hat{\mu}^{\textup{r. tail}}, \quad \hat{\mu}^{\textup{trunc}} = \frac{1}{N} \sum_{i=1}^{N-N^{1-\theta}} \delta_{\lambda_i}, \quad \hat{\mu}^{\textup{r. tail}} = \frac{1}{N} \sum_{i=N-N^{1-\theta}+1}^N \delta_{\lambda_i}.
\]
Notice that $\hat{\mu}^{\textup{trunc}}$ has mass $1-N^{-\theta}$ and $\hat{\mu}^{\textup{r. tail}}$ has mass $N^{-\theta}$. Compared to \eqref{eqn:manyevents}, the event $\mc{E}_{\textup{ss}}$ is no longer necessary; the events $\mc{E}_{\textup{gap}}$ and $\mc{E}_b$ remain the same (since they clearly imply the analogues for $\hat{\mu}^{\textup{trunc}}$), and each still has probability $1-o(1)$; the event $\mc{E}_{\textup{conc}}$ is replaced with
\[
	\mc{E}_{\textup{conc}}^{\textup{trunc}} = \left\{ \abs{ \int \log_\eta^K(\lambda) (\hat{\mu}^{\textup{trunc}} - \E[\hat{\mu}^{\textup{trunc}}])(\diff \lambda)} \leq t\right\}.
\]
This is a likely event, since 
\begin{equation}
\label{eqn:log_righttail}
	\abs{ \int \log_\eta^K(\lambda) (\hat{\mu}^{\textup{r. tail}} - \E[\hat{\mu}^{\textup{r. tail}}])(\diff \lambda)} \leq 2 \log_\eta(K) \hat{\mu}^{\textup{r. tail}}(\R) \lesssim N^{\epsilon - \theta} < \frac{t}{2}
\end{equation}
(here it matters that $\theta$ not be too small), and thus if $\epsilon < \theta$
\[
	1-\P(\mc{E}_{\textup{conc}}^{\textup{trunc}}) \leq \P\left(\abs{ \int \log_\eta^K(\lambda) (\hat{\mu} - \E[\hat{\mu}])(\diff \lambda)} \geq \frac{t}{2}\right),
\]
but the right-hand probability is $o(1)$ by arguments as in the proof of Lemma \ref{lem:prob_econc}. By mimicking \eqref{eqn:detconlb_fundamental} but handling the large eigenvalues instead with \eqref{eqn:example_large_evals},  $\frac{1}{N} \log \E[\abs{\det(H_N)}]$ is larger than
\begin{align*}
	&1 - \frac{\log 2}{N^{\theta}} + \frac{1}{N} \E\left[e^{N \int (\log\abs{\lambda} - \log_\eta(\lambda)) \hat{\mu}^{\textup{trunc}}(\diff \lambda)} \mathds{1}_{\mc{E}_{\textup{gap}}} \mathds{1}_{\mc{E}_{\textup{conc}}^{\textup{trunc}}} \right] - t + \int \log_\eta^K(\lambda) \E[\hat{\mu}^{\textup{trunc}}](\diff \lambda) \\
	&= 1 + \int \log_\eta^K(\lambda) \E[\hat{\mu}^{\textup{trunc}}](\diff \lambda) - \oo(1),
\end{align*}
where the last equality follows by mimicking Lemma \ref{lem:eps2}. Now $\E[\hat{\mu}^{\textup{trunc}}] = \E[\hat{\mu}] - \E[\hat{\mu}^{\textup{r. tail}}]$, and by \eqref{eqn:truncate_muN} and arguments as in the proof of Lemma \ref{lem:detcon_eps1}, we have 
$
	\int \log_\eta^K(\lambda) \E[\hat{\mu}](\diff \lambda) \geq \int \log\abs{\lambda} \rho_{\text{sc}}(\lambda) \diff \lambda + \oo(1).
$
The term $\int \log_\eta^K(\lambda) \E[\hat{\mu}^{\textup{r. tail}}](\diff \lambda)$ is handled as in \eqref{eqn:log_righttail}. Overall, this gives
$
	\liminf_{N \to \infty} \frac{1}{N} \log \E[\abs{\det(H_N)}] \geq 1 + \int \log\abs{\lambda} \rho_{\text{sc}}(\lambda) \diff \lambda
$
which contradicts \eqref{eqn:conv}. 

\subsubsection{Necessity of spectral stability.}\
With $T$ as above, the eigenvalues of $H$ are at most $N + T\sqrt{N}$ deterministically; this implies \eqref{eqn:detconub_coarse} and \eqref{eqn:detconub_NlogN}. For \eqref{eqn:wegner}, we note that on the event $\{X_0 = N\}$, the eigenvalues are at least $N - T\sqrt{N} > 0$, so there are clearly no eigenvalues near zero; on the event $\{X_0 = 0\}$, the matrix $H$ is just a Wigner matrix, for which we proved \eqref{eqn:wegner} above. Now we claim that assumption \hyperlink{assn:E}{(E)} holds with $\mu_N = \rho_{\text{sc}}$. Indeed, for test functions $f$ with $\|f\|_{\textup{Lip}} + \|f\|_{L^\infty} \leq 1$ we have
\[
	\E\left[\abs{\int f(x) (\hat{\mu}_{H} - \rho_{\text{sc}})(\diff x)} \mathds{1}_{X_0 = N}\right] \leq 2\P(X_0 = N) = \frac{2}{N},
\]
and on the event $\mathds{1}_{X_0 = 0}$ we revert to the Wigner case studied above. 

On the other hand, notice that $(X_{\textup{cut}})_0$ is always zero, so on the event $\{X_0 = N\}$ the measure $\hat{\mu}_{\Phi(X_{\textup{cut}})}$ is supported on $[-T\sqrt{N},T\sqrt{N}]$ while $\hat{\mu}_{\Phi(X)}$ is supported on $[N-T\sqrt{N}, N+T\sqrt{N}]$. For large enough $N$ these are disjoint, so the measures are one apart in KS distance, and thus
$
	\P(d_{\textup{KS}}(\hat{\mu}_{\Phi(X)}, \hat{\mu}_{\Phi(X_{\textup{cut}})}) > N^{-\kappa}) \geq \frac{1}{N},
$
which shows that \hyperlink{assn:S}{(S)} fails. 

Finally, since
$
	\E[\abs{\det(H)}] \geq \E[\abs{\det(H)}\mathds{1}_{X_0 = N}] \geq (N-T\sqrt{N})^N \P(X_0 = N),
$
we have $\frac{1}{N} \log \E[\abs{\det(H)}] \to +\infty$ and determinant concentration fails.


\section{Variational principles and long-range correlations}
\label{sec:longrange}


\subsection{General scheme.}\
In this section, we study expected determinants in the presence of long-range matrix correlations. The prototypical example to keep in mind is
\[
	H_N = W_N + \xi \Id,
\] 
where $W_N$ is drawn from the Gaussian Orthogonal Ensemble (GOE), and $\xi \sim \mc{N}(0,1/N)$ is independent of $W_N$. Matrices of this style are very common in the landscape-complexity program, but our main theorems do not apply directly because of the presence of long-range correlations (here, along the diagonal of $H_N$). Nevertheless, there is still a general procedure to understand the determinant asymptotics for such matrices, which we illustrate in the case of this example. We first notice
\[
	\E[\abs{\det(H_N)}] = \frac{1}{\sqrt{2\pi/N}} \int_\R e^{-N \frac{u^2}{2}} \E[\abs{\det(W_N + u)}] \diff u.
\]
Our determinant asymptotics do apply to $W_N + u$, giving $\E[\abs{\det(W_N + u)}] \approx e^{N\Sigma(u)}$ for some constants $\Sigma(u)$; then the Laplace method suggests
\begin{equation}
\label{eqn:longrangedetcon}
	\lim_{N \to \infty} \frac{1}{N} \log \E[\abs{\det(H_N)}] = \sup_{u \in \R} \left\{ \Sigma(u) - \frac{u^2}{2}\right\}.
\end{equation}
This method has appeared before in special cases, for example in \cite{AufChe2014} and \cite{FyoLeD2020}. In Section \ref{subsec:unrestricted}, we prove results of this type without reference to any particular matrix model. In Section \ref{subsec:restricted}, we prove extensions necessary to understand asymptotics of the form
\[
	\lim_{N \to \infty} \frac{1}{N} \log \E[\abs{\det(H_N)} \mathds{1}_{H_N \geq 0}].
\]
In complexity computations, these ``restricted determinants'' correspond to counting just the local minima among all critical points. The upshot is that this limit is also a variational problem as in \eqref{eqn:longrangedetcon}, but restricted to $u$ in some good set instead of all Euclidean space.


\subsection{Variational principles for unrestricted determinants.}\
\label{subsec:unrestricted}
For applications to complexity, we will need not just one matrix $H_N$, but a field of matrices $H_N(u)$ for $u \in \R^m$ (here $m$ is independent of $N$), with approximating measures $\mu_N(u)$. 
\begin{thm}
\label{thm:variational_formula}
Assume the following:
\begin{itemize}
\item \textbf{(Assumptions locally uniform in $u$)} Each $H_N(u)$ satisfies all the assumptions of Theorem \ref{thm:convex_functional}, or all the assumptions of Theorem \ref{thm:concentrated_input}. In addition, all limits, powers, and rates in these assumptions are uniform over compact sets of $u$.\footnote{For example, writing $(\lambda_i(u))_{i=1}^N$ for the eigenvalues of $H_N(u)$, the condition \eqref{eqn:detconub_coarse} becomes: for every compact $K \subset \R^m$,
$
    \lim_{N \to \infty} \sup_{u \in K} \frac{1}{N}\log \E\left[ \prod_{i=1}^N (1+\abs{\lambda_i(u)} \mathds{1}_{\abs{\lambda_i(u)} > e^{N^\epsilon}}) \right] = 0
$
.}
\item \textbf{(Limit measures)} There exist probability measures $\mu_\infty(u)$ such that 
\begin{align*}
    d_{\textup{BL}}(\mu_N(u), \mu_\infty(u)) \leq N^{-\kappa} \qquad &\text{if we are in the setting of Theorem \ref{thm:convex_functional}, or} \\
    {\rm W}_1(\mu_N(u), \mu_\infty(u)) \leq N^{-\kappa} \qquad &\text{if we are in the setting of Theorem \ref{thm:concentrated_input}}
\end{align*}
for $\kappa = \kappa(u)>0$ that can, again, be chosen uniformly on compact sets of $u$. These measures also admit densities $\mu_\infty(u,\cdot)$ on $[-\kappa,\kappa]$ that satisfy $\mu_\infty(u,x) < \kappa^{-1}\abs{x}^{-1+\kappa}$ for all $\abs{x} < \kappa$. 
\item \textbf{(Continuity and decay in $u$)} For each $N$, the map $u \mapsto H_N(u)$ is entrywise continuous. Furthermore, there exists $C > 0$ such that
\begin{equation}
\label{eqn:det_sublinear_in_u}
    \E[\abs{\det(H_N(u))}] \leq (C \max(\|u\|,1))^N.
\end{equation}
\end{itemize}
Then for any $\alpha > 0$, any fixed $p \in \R$, and any $\mf{D} \subset \R^m$ with positive Lebesgue measure that is the closure of its interior, we have
\[
    \lim_{N \to \infty} \frac{1}{N} \log \int_{\mf{D}} e^{-(N+p)\alpha\|u\|^2}\E[\abs{\det(H_N(u))}] \diff u = \sup_{u \in \mf{D}} \left\{ \int_\R \log\abs{\lambda}\mu_\infty(u)(\diff \lambda) - \alpha \|u\|^2 \right\}.
\]
\end{thm}

\begin{rem}
\label{rem:closure_of_interior}
A close inspection of the proof shows that the condition ``$\mf{D}$ is the closure of its interior'' is only necessary for the lower bound in Theorem \ref{thm:variational_formula}. For the upper bound, it suffices to assume that $\mf{D}$ is simply closed (and has positive measure). We will use this below.
\end{rem}

The proof of this theorem relies on the following two lemmas, in addition to determinant concentration in the form of Theorem \ref{thm:convex_functional} or \ref{thm:concentrated_input}. We postpone their proofs until after the proof of the theorem. Recall that $B_R$ is the ball of radius $R$ around zero in $\R^m$. 

\begin{lem}
\label{lem:restricttocompacts}
\[
    \lim_{R \to \infty} \limsup_{N \to \infty} \frac{1}{N} \log \int_{B_R^c} e^{-(N+p) \alpha \|u\|^2} \E[\abs{\det(H_N(u))}] \diff u = -\infty.
\]
\end{lem}

\begin{lem}
\label{lem:Siscts}
The function
\[
    \mc{S}_\alpha[u] = \int_\R \log\abs{\lambda} \mu_\infty(u) (\diff \lambda) - \alpha \|u\|^2,
\]
is continuous, and $
    \lim_{\|u\| \to +\infty} \mc{S}_\alpha[u] = -\infty.$
\end{lem}

\begin{proof}[Proof of Theorem \ref{thm:variational_formula}]
First we prove the upper bound. We apply Theorem \ref{thm:convex_functional} or \ref{thm:concentrated_input} with the reference measures $\mu_\infty(u)$.  Since all inputs are uniform over compact sets of $u$, so is the conclusion; that is, for all $R$, we have
\[
    \limsup_{N \to \infty} \frac{1}{N} \log  \sup_{u \in B_R} \left\{  \E[\abs{\det(H_N(u))}]e^{-N\int_\R \log\abs{\lambda} \mu_\infty(u) (\diff \lambda)} \right\} \leq 0
\]
and a matching lower bound we will use momentarily. If $R$ is large enough that $\abs{B_R \cap \mf{D}} > 0$, then we conclude
\begin{align*}
    \limsup_{N \to \infty} \frac{1}{N} \log \int_{B_R \cap \mf{D}} e^{-(N+p) \alpha \|u\|^2} \E[\abs{\det(H_N(u))}] \diff u &\leq \limsup_{N \to \infty} \frac{1}{N}\log e^{\abs{p}\alpha R^2} \int_{B_R \cap \mf{D}} e^{-N \alpha \|u\|^2 + N\int_\R \log\abs{\lambda} \mu_\infty(u,\lambda) \diff \lambda} \diff u \\
    &\leq \sup_{u \in \mf{D}} \mc{S}_\alpha[u] + \limsup_{N \to \infty} \left[ \frac{\log(\abs{B_R \cap \mf{D})}}{N} \right].
\end{align*}
An application of Lemma \ref{lem:restricttocompacts} finishes the proof of the upper bound. 

Now we prove the lower bound. Lemma \ref{lem:Siscts} tells us that $\sup_{u \in \mf{D}} \mc{S}_\alpha[u]$ is achieved at some (possibly not unique) $u_0$. Since $\mc{S}_\alpha$ is continuous, for every $\epsilon > 0$ there exists a bounded neighborhood $\mc{U}_\epsilon$ of $u_0$ on which $\mc{S}_\alpha[u] \geq \mc{S}_\alpha[u_0] - \epsilon$. Since $\mf{D}$ is the closure of its interior, we have $\abs{\mc{U}_\epsilon \cap \mf{D}} > 0$.

For each $R$, applying Theorem \ref{thm:convex_functional} or \ref{thm:concentrated_input} with arguments as above yields
\[
    \liminf_{N \to \infty} \frac{1}{N} \log  \inf_{u \in B_R} \left\{  \E[\abs{\det(H_N(u))}]e^{-N\int_\R \log\abs{\lambda} \mu_\infty(u) (\diff \lambda)} \right\} \geq 0.
\]
If $R$ is so large that $\mc{U}_\epsilon \subset B_R$, then
\begin{align*}
    &\liminf_{N \to \infty} \frac{1}{N}\log \int_{\mf{D}} e^{-(N+p) \alpha \|u\|^2} \E[\abs{\det(H_N(u))}] \diff u \\
    &\geq \liminf_{N \to \infty} \frac{1}{N} \log \left\{ e^{-\abs{p}\alpha R^2} \int_{\mc{U}_\epsilon \cap \mf{D}} e^{-N \alpha \|u\|^2} \E[\abs{\det(H_N(u))}] \diff u \right\}  \\
    &\geq \liminf_{N \to \infty} \frac{1}{N} \log \int_{\mc{U}_\epsilon \cap \mf{D}} e^{N\mc{S}_\alpha[u]} \diff u \geq \liminf_{N \to \infty} \frac{1}{N} \log \int_{\mc{U}_\epsilon \cap \mf{D}} e^{N(\mc{S}_\alpha[u_0] - \epsilon)} \diff u \geq \mc{S}_\alpha[u_0] - \epsilon + \liminf_{N \to \infty} \frac{\log(\abs{\mc{U}_\epsilon \cap \mf{D}})}{N}.
\end{align*}
Letting $\epsilon \to 0$ completes the proof. 
\end{proof}

\begin{proof}[Proof of Lemma \ref{lem:restricttocompacts}]
If $\omega_m$ is the surface area of the unit ball in $\R^m$, then from \eqref{eqn:det_sublinear_in_u} we have
\begin{align*}
    \int_{B_R^c} e^{-(N+p) \alpha \|u\|^2} \E[\abs{\det(H_N(u))}] \diff u &\leq \int_{B_R^c}  e^{N\left(\log(C \|u\|) - \alpha \|u\|^2 \right)-p\alpha \|u\|^2} \diff u = \omega_m \int_R^\infty e^{N\left(\log(C r) - \alpha r^2 \right)-p\alpha r^2} r^{m-1} \diff r
\end{align*}
which suffices by the Laplace method.
\end{proof}

\begin{proof}[Proof of Lemma \ref{lem:Siscts}]
Fix $N$. We assumed that $H_N(u)$ is an entrywise continuous function of $u$. Since the determinant is a continuous function of the matrix entries, dominated convergence (with dominating function given by \eqref{eqn:det_sublinear_in_u}) says that $\E[\abs{\det(H_N(u))}]$ is continuous in $u$, hence so is $\frac{1}{N}\log \E[\abs{\det(H_N(u))}]$.  Then Theorem \ref{thm:convex_functional} or \ref{thm:concentrated_input} shows
\begin{equation}
\label{eqn:detconlimit}
    \lim_{N \to \infty} \sup_{u \in B_R} \abs{ \frac{1}{N}\log \E[\abs{\det(H_N(u))}] - \int_\R \log\abs{\lambda} \mu_\infty(u) (\diff \lambda)} = 0,
\end{equation}
and $\int_\R \log\abs{\lambda} \mu_\infty(u)(\diff \lambda)$ is the locally uniform limit of continuous functions. Thus $\mc{S}_\alpha[u]$ is continuous. 

The decay at infinity follows from $
    \int_\R \log\abs{\lambda} \mu_\infty(u) (\diff \lambda) \leq \liminf_{N \to \infty} \frac{1}{N}\log \E[\abs{\det(H_N(u))}] \leq \log(C\max(\|u\|,1))
$, obtained by \eqref{eqn:detconlimit} and \eqref{eqn:det_sublinear_in_u}.
\end{proof}


\subsection{Variational principles for restricted determinants.}\
\label{subsec:restricted}
Let $\mc{G} \subset \R^m$ be the set of ``good'' $u$ values 
\begin{equation}
\label{eqn:def_good_set}
    \mc{G} = \{u \in \R^m : \mu_\infty(u)((-\infty,0)) = 0\} = \{u \in \R^m : \mathtt{l}(\mu_\infty(u)) \geq 0\}.
\end{equation}
For each $\epsilon > 0$, consider the following inner and outer approximations of $\mc{G}$:
\begin{align}
\label{eqn:def_gplusminus}
\begin{split}
    \mc{G}_{+\epsilon} &= \{u \in\R^m: \mathtt{l}(\mu_\infty(u)) \geq 2\epsilon\}, \\
    \mc{G}_{-\epsilon} &= \{u \in \R^m : \mu_\infty(u)((-\infty,-\epsilon)) \leq \epsilon\}.
\end{split}
\end{align}

\begin{thm}
\label{thm:minima}
Fix some $\mf{D} \subset \R^m$, and suppose that $\mf{D}$ and the matrices $H_N(u)$ satisfy the following.
\begin{itemize}
\item All the assumptions of Theorem \ref{thm:variational_formula}.
\item \textbf{(Superexponential concentration)} For every $R > 0$ and every $\epsilon > 0$, we have
\begin{equation}
\label{eqn:concentration_NlogN}
    \lim_{N \to \infty} \frac{1}{N\log N}\log \left[ \sup_{u \in B_R} \P(d_{\textup{BL}}(\hat{\mu}_{H_N(u)}, \mu_\infty(u)) > \epsilon) \right] = -\infty.
\end{equation}
\item \textbf{(No outliers)} For every $R > 0$ and every $\epsilon > 0$, we have
\begin{equation}
\label{eqn:nooutliers}
    \lim_{N \to \infty} \inf_{u \in \mf{D} \cap \mc{G}_{+\epsilon} \cap B_R} \P(\Spec(H_N(u)) \subset [\mathtt{l}(\mu_\infty(u)) - \epsilon, \mathtt{r}(\mu_\infty(u)) + \epsilon]) = 1.
\end{equation}
\item \textbf{(Topology)} Each $\mc{G}_{+\epsilon}$ is convex; $\mf{D}$ is convex and closed; the set $\mf{D} \cap \mc{G}_{+1}$ has positive Lebesgue measure; and
\begin{equation}
\label{eqn:minima_closure}
    \overline{\mf{D} \cap \left(\bigcup_{\epsilon > 0} \mc{G}_{+\epsilon} \right)} = \mf{D} \cap \mc{G}.
\end{equation}
\end{itemize}
Then for any $\alpha > 0$ and any fixed $p \in \R$, we have
\[
    \lim_{N \to \infty} \frac{1}{N} \log \int_{\mf{D}} e^{-(N+p)\alpha \|u\|^2} \E[\abs{\det(H_N(u))} \mathds{1}_{H_N(u) \geq 0}] \diff u = \sup_{u \in \mf{D} \cap \mc{G}} \left\{ \int_\R \log\abs{\lambda} \mu_\infty(\diff \lambda) - \alpha \|u\|^2 \right\}.
\]
\end{thm}

We prove the upper and lower bounds separately in the next two subsubsections.

\subsubsection{Upper bound.}\
The proof of the upper bound of Theorem \ref{thm:minima} relies on the following three lemmas, which we will prove after.

\begin{lem}
\label{lem:minima_topology}
Each $\mc{G}_{-\epsilon}$ is closed, and $\mc{G}$ is closed.
\end{lem}

\begin{lem}
\label{lem:min_restriction_ub}
For every $\epsilon > 0$, we have
\[
    \lim_{N \to \infty} \frac{1}{N}\log \int_{(\mc{G}_{-\epsilon})^c} e^{-(N+p)\alpha \|u\|^2} \E[\abs{\det(H_N(u))} \mathds{1}_{H_N(u) \geq 0}] \diff u = -\infty.
\]
\end{lem}

\begin{lem}
\label{lem:g_minus_eps}
We have
\[
    \lim_{\epsilon \downarrow 0} \sup_{u \in \mf{D} \cap \mc{G}_{-\epsilon}} \mc{S}_\alpha[u] \leq \sup_{u \in \mf{D} \cap \mc{G}} \mc{S}_\alpha[u].
\]
\end{lem}

\begin{proof}[Proof of the upper bound in Theorem \ref{thm:minima}]
For each $\epsilon > 0$, Lemma \ref{lem:min_restriction_ub} yields
\begin{align*}
    &\limsup_{N \to \infty} \frac{1}{N} \log \int_{\mf{D}} e^{-(N+p)\alpha \|u\|^2} \E[\abs{\det(H_N(u))}\mathds{1}_{H_N(u) \geq 0}] \diff u \\
    &\leq \limsup_{N \to \infty} \frac{1}{N}\log \int_{\mf{D} \cap \mc{G}_{-\epsilon}} e^{-(N+p) \alpha \|u\|^2} \E[\abs{\det(H_N(u))}\mathds{1}_{H_N(u) \geq 0}] \diff u \\
    &\leq \limsup_{N \to \infty} \frac{1}{N}\log \int_{\mf{D} \cap \mc{G}_{-\epsilon}} e^{-(N+p) \alpha \|u\|^2} \E[\abs{\det(H_N(u))}] \diff u \leq \sup_{u \in \mf{D} \cap \mc{G}_{-\epsilon}} \mc{S}_\alpha[u].
\end{align*}
The last inequality holds by Theorem \ref{thm:variational_formula} applied to $\mf{D} \cap \mc{G}_{-\epsilon}$, which is closed (by Lemma \ref{lem:minima_topology}) and has positive measure (as a superset of $\mf{D} \cap \mc{G}_{+1}$, which has positive measure by assumption). By Remark \ref{rem:closure_of_interior}, these are the only conditions we need to check. Letting $\epsilon$ tend to zero and applying Lemma \ref{lem:g_minus_eps} completes the proof.
\end{proof}

\begin{proof}[Proof of Lemma \ref{lem:minima_topology}]
Since we assumed that $u \mapsto H_N(u)$ is entrywise continuous and the spectrum is a continuous function of matrix entries, we have that $u \mapsto \hat{\mu}_{H_N(u)}$ is almost surely continuous with respect to the bounded-Lipschitz distance: $d_{\textup{BL}}(\hat{\mu}_{H_N(u)}, \hat{\mu}_{H_N(u')}) \leq \frac{1}{N}\sum_{i=1}^N \min(2,\abs{\lambda_i(u)-\lambda_i(u')})$. By dominated convergence, this means that $u \mapsto \E[\hat{\mu}_{H_N(u)}]$ is continuous with respect to $d_{\textup{BL}}$. But $d_{\textup{BL}}(\E[\hat{\mu}_{H_N(u)}], \mu_\infty(u)) \to 0$ uniformly on compact sets of $u$ by assumption (here we use $d_{\textup{BL}} \leq {\rm W}_1$ for the concentrated-input case), so we conclude that $u \mapsto \mu_\infty(u)$ is continuous with respect to $d_{\textup{BL}}$, as well. Since $d_{\textup{BL}}$ metrizes weak convergence, and since the defining properties of $\mc{G}$ and $\mc{G}_{-\epsilon}$ can be stated in terms of distribution functions of $\mu_\infty(u)$, which are continuous since each $\mu_\infty(u)$ has a density with respect to Lebesgue, the lemma follows.
\end{proof}

\begin{proof}[Proof of Lemma \ref{lem:min_restriction_ub}]
From Lemma \ref{lem:restricttocompacts}, it suffices to show
\[
    \lim_{N \to \infty} \frac{1}{N}\log \int_{(\mc{G}_{-\epsilon})^c \cap B_R} e^{-(N+p)\alpha \|u\|^2} \E[\abs{\det(H_N(u))} \mathds{1}_{H_N(u) \geq 0}] \diff u = -\infty
\]
for each $R > 0$. If $H_N(u) \geq 0$ and $u \in (\mc{G}_{-\epsilon})^c$, then by taking some $\frac{1}{2}$-Lipschitz $f_\epsilon$ satisfying
$
    \frac{\epsilon}{2} \mathds{1}_{x \leq 0} \geq f_\epsilon(x) \geq \frac{\epsilon}{2} \mathds{1}_{x \leq -\epsilon}
$
we obtain 
$
    d_{\textup{BL}}(\hat{\mu}_{H_N(u)},\mu_\infty(u)) \geq \frac{\epsilon}{2}\mu_\infty(u)((-\infty,-\epsilon)) \geq \frac{\epsilon^2}{2}.
$
For small $\delta > 0$, this gives (for $N > -p$)
\begin{align*}
    &\int_{(\mc{G}_{-\epsilon})^c \cap B_R} e^{-(N+p)\alpha \|u\|^2} \E[\abs{\det(H_N(u))} \mathds{1}_{H_N(u) \geq 0}] \diff u \\
    &\leq \abs{B_R} \left(\sup_{u \in B_R} \E[\abs{\det(H_N(u))}^{1+\delta}]^{\frac{1}{1+\delta}} \right) \left( \sup_{u \in B_R} \P(d_{\textup{BL}}(\hat{\mu}_{H_N(u)},\mu_\infty(u)) \geq \frac{\epsilon^2}{2}) \right)^{\frac{\delta}{1+\delta}}.
\end{align*}
This suffices by \eqref{eqn:detconub_NlogN} and \eqref{eqn:concentration_NlogN}.
\end{proof}

\begin{proof}[Proof of Lemma \ref{lem:g_minus_eps}]
From their definitions, we have
$
    \bigcap_{\epsilon > 0} \mc{G}_{-\epsilon} = \mc{G}.
$
We take the intersection of both sides with $\mf{D}$. Next, we claim that there exists some $R > 0$ such that 
\begin{equation}
\label{eqn:minimaub_restricttocpct}
    \sup_{u \in \mf{D} \cap \mc{G}} \mc{S}_\alpha[u] = \max_{u \in (\mf{D} \cap \mc{G} \cap B_R)} \mc{S}_\alpha[u] \quad \text{and} \quad \sup_{u \in \mf{D} \cap \mc{G}_{-\epsilon}} \mc{S}_\alpha[u] = \max_{u \in (\mf{D} \cap \mc{G}_{-\epsilon} \cap B_R)} \mc{S}_\alpha[u]
\end{equation}
for every $\epsilon > 0$. Indeed, the proof of Lemma \ref{lem:Siscts} shows that
\[
    \mc{S}_\alpha[u] \leq \log(C\|u\|) - \alpha \|u\|^2
\]
on $\R^m$. Since $\mc{S}_\alpha$ is continuous and $\mf{D} \cap \mc{G}$ is closed by Lemma \ref{lem:minima_topology}, let $u_\ast \in \mf{D} \cap \mc{G}$ satisfy $\sup_{u \in \mf{D} \cap \mc{G}} \mc{S}[u] = \mc{S}[u_\ast]$, and let $R > 1$ be so large that $\log(C R) - \alpha R^2 < \mc{S}_\alpha[u_\ast]$.

For each $\epsilon$, since $\mf{D} \cap \mc{G}_{-\epsilon}$ is closed (again by Lemma \ref{lem:minima_topology}), let $u_\epsilon$ be such that $\mc{S}_\alpha[u_\epsilon] = \sup_{u \in \mf{D} \cap \mc{G}_{-\epsilon}} \mc{S}_\alpha[u]$. Then $u_\epsilon \in B_R$; otherwise, we would have
\[
    \max_{u \in \mf{D} \cap \mc{G}_{-\epsilon}} \mc{S}_\alpha[u] = \mc{S}_\alpha[u_\epsilon] \leq \log(C R) - \alpha R^2 < \mc{S}_\alpha[u_\ast] = \max_{u \in \mf{D} \cap \mc{G}} \mc{S}_\alpha[u] \leq \max_{u \in \mf{D} \cap \mc{G}_{-\epsilon}} \mc{S}_\alpha[u].
\]
This verifies \eqref{eqn:minimaub_restricttocpct}.

Since the $\{u_\epsilon\}$ lie in a compact set, they have a limit point $u_0$ up to extraction. Furthermore, $u_0 \in \mf{D} \cap \mc{G} = \cap_\epsilon (\mf{D} \cap \mc{G}_{-\epsilon})$. Indeed, otherwise a neighborhood of $u_0$ would be contained in $(\mf{D} \cap \mc{G}_{-\epsilon_1})^c$ for some $\epsilon_1$, hence in $(\mf{D} \cap \mc{G}_{-\epsilon})^c$ for every $\epsilon < \epsilon_1$ (since the sets are nested). But then $u_0$ could not be a limit point of $\{u_\epsilon\}$.

Thus by continuity of $S_\alpha$ we have
$
    \lim_{\epsilon \downarrow 0} \sup_{u \in \mf{D} \cap \mc{G}_{-\epsilon}} \mc{S}_\alpha[u] = \lim_{\epsilon \downarrow 0} \mc{S}_\alpha[u_\epsilon] = \mc{S}_\alpha[u_0] \leq \sup_{u \in \mf{D} \cap \mc{G}} \mc{S}_{\alpha}[u]
$.
\end{proof}

\subsubsection{Lower bound.}\
The proof of the lower bound in Theorem \ref{thm:minima} relies on the following two lemmas, which we will prove after.
\begin{lem}
\label{lem:minima_inf_logpotential}
For each $u \in \R^m$ and each $\delta, \epsilon > 0$, define the set of probability measures
\[
    M(u,\delta,\epsilon) = \{ \mu : d_{\textup{BL}}(\mu,\mu_\infty(u)) < \delta \text{ and } \supp(\mu) \subset [\mathtt{l}(\mu_\infty(u))-\epsilon,\mathtt{r}(\mu_\infty(u))+\epsilon] \}
\]
that are close to $\mu_\infty(u)$ both in $d_{\textup{BL}}$ and in support. For all $R$, all $\delta$, and all $\epsilon$ sufficiently small depending on $R$, we have
\[
    \inf_{u \in \mf{D} \cap \mc{G}_{+\epsilon} \cap \overline{B_R}} \left( \inf_{\mu \in M(u,\delta,\epsilon)} \int \log\abs{\lambda} \mu(\diff \lambda) - \int\log\abs{\lambda} \mu_\infty(u)(\diff \lambda) \right) \geq -\frac{2\delta}{\epsilon}.
\]
\end{lem}

\begin{lem}
\label{lem:g_plus_eps}
Each $\mc{G}_{+\epsilon}$ is closed, and for all $R$ large enough we have
\begin{equation}
\label{eqn:minimalb_restricttocpct}
    \sup_{u \in \mf{D} \cap \mc{G}} \mc{S}_\alpha[u] = \max_{u \in \mf{D} \cap \mc{G} \cap \overline{B_R}} \mc{S}_\alpha[u] \quad \text{and} \quad \sup_{u \in \mf{D} \cap \mc{G}_{+\epsilon}} \mc{S}_\alpha[u] = \max_{u \in \mf{D} \cap \mc{G}_{+\epsilon} \cap \overline{B_R}} \mc{S}_\alpha[u]
\end{equation}
for every $0 < \epsilon < 1$. Furthermore,
\begin{equation}
\label{eqn:g_plus_eps}
    \lim_{\epsilon \downarrow 0} \sup_{u \in \mf{D} \cap \mc{G}_{+\epsilon}} \mc{S}_\alpha[u] = \sup_{u \in \mf{D} \cap \mc{G}} \mc{S}_\alpha[u].
\end{equation}
\end{lem}

\begin{proof}[Proof of the lower bound in Theorem \ref{thm:minima}]
Since
\[
    \P(\hat{\mu}_{H_N(u)} \not\in M(u,\delta,\epsilon)) \leq \P(\Spec(H_N(u)) \not\subset [\mathtt{l}(\mu_\infty(u)) - \epsilon, \mathtt{r}(\mu_\infty(u)) + \epsilon]) + \P(d_{\textup{BL}}(\hat{\mu}_{H_N(u)}, \mu_\infty(u)) > \delta), 
\]
\eqref{eqn:nooutliers} and \eqref{eqn:concentration_NlogN} tell us that
\begin{equation}
\label{eqn:minimalb_goodprobmeasures}
    \lim_{N \to \infty} \frac{1}{N} \log \left( \inf_{u \in \mf{D} \cap \mc{G}_{+\epsilon} \cap \overline{B_R}} \P(\hat{\mu}_{H_N(u)} \in M(u,\delta,\epsilon)) \right) = 0.
\end{equation}

Let $R$ satisfy Lemma \ref{lem:g_plus_eps}, and additionally be so large that $\abs{\mf{D} \cap \mc{G}_{+1} \cap \overline{B_R}} > 0$. If $\epsilon < 1$ is sufficiently small depending on $R$, then by Lemma \ref{lem:minima_inf_logpotential} we have
\begin{align*}
    &\int_{\mf{D}} e^{-(N+p)\alpha\|u\|^2} \E[\abs{\det(H_N(u))} \mathds{1}_{H_N(u) \geq 0}] \diff u \\
    &\geq e^{-\abs{p}\alpha R^2} \int_{\mf{D} \cap \mc{G}_{+\epsilon} \cap \overline{B_R}} e^{-N\alpha\|u\|^2} \exp\left(N \inf_{\mu \in M(u,\delta,\epsilon)} \int \log\abs{\lambda} \mu(\diff \lambda) \right) \P(\hat{\mu}_{H_N(u)} \in M(u,\delta,\epsilon)) \diff u \\
    &\geq e^{-\abs{p}\alpha R^2} \left( \inf_{u \in \mf{D} \cap \mc{G}_{+\epsilon} \cap \overline{B_R}} \P(\hat{\mu}_{H_N(u)} \in M(u,\delta,\epsilon)) \right) \exp\left(-\frac{2N\delta}{\epsilon}\right) \int_{\mf{D} \cap \mc{G}_{+\epsilon} \cap \overline{B_R}} e^{N\mc{S}_\alpha[u]} \diff u.
\end{align*}
Now we take the logarithm of both sides, divide by $N$, let $N \to \infty$, and then let $\delta \downarrow 0$. The set $\mf{D} \cap \mc{G}_{+\epsilon} \cap \overline{B_R}$ is closed and convex, as the finite intersection of such sets. Since closed convex sets in Euclidean space have empty interior if and only if they lie in a lower-dimensional affine space, we conclude that $\mf{D} \cap \mc{G}_{+\epsilon} \cap \overline{B_R}$ has nonempty interior from the fact that it has positive measure. Since a closed convex set with nonempty interior is the closure of its interior, we can apply Theorem \ref{thm:variational_formula} to this set. From this theorem and from \eqref{eqn:minimalb_goodprobmeasures}, we have
\[
    \liminf_{N \to \infty} \int_{\mf{D}} e^{-(N+p)\alpha\|u\|^2} \E[\abs{\det(H_N(u))} \mathds{1}_{H_N(u) \geq 0}] \diff u \geq \sup_{u \in \mf{D} \cap \mc{G}_{+\epsilon} \cap B_R} \mc{S}_\alpha[u] = \sup_{u \in \mf{D} \cap \mc{G}_{+\epsilon}} \mc{S}_\alpha[u].
\]
By \eqref{eqn:g_plus_eps}, this suffices.
\end{proof}

\begin{proof}[Proof of Lemma \ref{lem:minima_inf_logpotential}]
Consider the function $f_u$ defined on $[\mathtt{l}(\rho_\infty(u)) - \epsilon, \mathtt{r}(\rho_\infty(u)) + \epsilon]$ by
$
    f_u(\lambda) = \log\abs{\lambda}.
$
If $u \in \mf{D} \cap \mc{G}_{+\epsilon} \cap \overline{B_R}$, then
\[
    \|f_u\|_{\text{Lip}} + \|f_u\|_{L^\infty} \leq \frac{1}{\epsilon} + \max\left\{ \abs{\log(\epsilon)}, \abs{\log(\mathtt{r}(\rho_\infty(u))+\epsilon)} \right\} \leq \frac{2}{\epsilon}
\]
where the last inequality holds for $\epsilon$ sufficiently small, uniformly over $u \in \overline{B_R}$, since $\supp(\rho_\infty(u))$ is compactly supported uniformly over $u \in \overline{B_R}$. This implies that whenever $\mu \in M(u,\delta,\epsilon)$, we have
\[
	\abs{\int \log\abs{\lambda} \mu(\diff \lambda) - \int \log\abs{\lambda} \mu_\infty(u)(\diff \lambda)}  \leq \frac{2}{\epsilon} d_{\textup{BL}}(\mu,\mu_\infty(u)) \leq \frac{2\delta}{\epsilon}.
\]
\end{proof}

\begin{proof}[Proof of Lemma \ref{lem:g_plus_eps}]
The proof of Lemma \ref{lem:minima_topology} shows that the map $u \mapsto \rho_\infty(u)$ is continuous with respect to weak convergence; thus each $\mc{G}_{+\epsilon}$ is closed.

The proof of Lemma \ref{lem:Siscts} shows that $\mc{S}_\alpha[u] \leq \log(C\|u\|) - \alpha \|u\|^2$ on $\R^m$ and that $\mc{S}_\alpha$ is continuous. Since $\mf{D} \cap \mc{G}$ is closed by Lemma \ref{lem:minima_topology}, and each $\mf{D} \cap \mc{G}_{+\epsilon}$ is closed by the argument above, we can write $\sup_{u \in \mf{D} \cap \mc{G}} \mc{S}_\alpha[u] = \mc{S}_\alpha[u_\ast]$ for some $u_\ast$ and $\sup_{u \in \mf{D} \cap \mc{G}_{+\epsilon}} \mc{S}_\alpha[u] = \mc{S}_\alpha[u_\epsilon]$ for some $u_\epsilon$.

Let $R > 1$ be so large that $\log(C R) - \alpha R^2 < \mc{S}_\alpha[u_1]$. Then $u_\epsilon \in B_R$ for each $\epsilon < 1$; else we would have
\[
    \max_{u \in \mf{D} \cap \mc{G}_{+\epsilon}} \mc{S}_\alpha[u] = \mc{S}_\alpha[u_\epsilon] \leq \log(CR) - \alpha R^2 < \mc{S}_\alpha[u_1] = \max_{u \in \mf{D} \cap \mc{G}_{+1}} \mc{S}_\alpha[u] \leq \max_{u \in \mf{D} \cap \mc{G}_{+\epsilon}} \mc{S}_\alpha[u].
\]
This verifies \eqref{eqn:minimalb_restricttocpct}. 

For each $\epsilon > 0$, let 
\[
    f_{+\epsilon}(u) = \begin{cases} \mc{S}_\alpha[u] & \text{if } u \in \mf{D} \cap \mc{G}_{+\epsilon}, \\ -\infty & \text{otherwise,} \end{cases} \qquad f_{+0}(u) = \sup_{\epsilon > 0} f_{+\epsilon}(u) = \begin{cases} \mc{S}_\alpha[u] & \text{if } u \in \mf{D} \cap (\cup_{\epsilon > 0} \mc{G}_{+\epsilon}), \\ -\infty & \text{otherwise.} \end{cases}
\]
Since the $\mc{G}_{+\epsilon}$'s are nested and $\mc{S}_\alpha$ is continuous, we have
\begin{align*}
    \lim_{\epsilon \downarrow 0} \sup_{u \in \mf{D} \cap \mc{G}_{+\epsilon}} \mc{S}_\alpha[u] &= \sup_{\epsilon > 0} \sup_{u \in \mf{D} \cap \mc{G}_{+\epsilon}} \mc{S}_\alpha[u] = \sup_{\epsilon > 0} \sup_{u \in \R^m} f_{+\epsilon}(u) = \sup_{u \in \R^m} \sup_{\epsilon > 0} f_{+\epsilon}(u) = \sup_{u \in \R^m} f_{+0}(u) \\
    &= \sup_{u \in \mf{D} \cap (\cup_{\epsilon > 0} \mc{G}_{+\epsilon})} \mc{S}_\alpha[u] = \sup_{u \in \overline{\mf{D} \cap (\cup_{\epsilon > 0} \mc{G}_{+\epsilon})}} \mc{S}_\alpha[u] = \sup_{u \in \mf{D} \cap \mc{G}} \mc{S}_\alpha[u],
\end{align*}
where the last equality follows from \eqref{eqn:minima_closure}.
\end{proof}


\setcounter{equation}{0}
\setcounter{thm}{0}
\renewcommand{\theequation}{A.\arabic{equation}}
\renewcommand{\thethm}{A.\arabic{thm}}
\appendix
\setcounter{secnumdepth}{0}
\hypertarget{sec:multipoint}{}
\section[Appendix A\ \ \ Extensions to products of determinants]
{Appendix A\ \ \ Extensions to products of determinants}

In this section, we are interested in expectations of products of determinants like $\E[\prod_{i=1}^\ell |\det(H_N^{(i)})|]$, where $\ell$ is independent of $N$. In the landscape complexity program, these asymptotics help understand the $\ell$th moment of the number of critical points of some high-dimensional random function. Everything essentially is the same as in the case $\ell = 1$, and we obtain leading-order determinant asymptotics consistent with
\begin{equation}
\label{eqn:exponential_factoring}
	\E\left[ \prod_{i=1}^\ell |\det(H_N^{(i)})| \right] \approx \prod_{i=1}^\ell \E[|\det(H_N^{(i)})|],
\end{equation}
on exponential scale in $N$.
This is true \emph{no matter the correlation structure between the $H_N^{(i)}$'s}, which is perhaps surprising at first glance. However, note that \eqref{eqn:exponential_factoring} should hold at ``both ends of the correlation spectrum,'' so to speak: On the one hand, it holds with exact equality if the $H_N^{(i)}$'s are independent; on the other hand, if we believe in concentration then \eqref{eqn:exponential_factoring} is very plausible when the $H_N^{(i)}$'s are the same as each other.

However, \eqref{eqn:exponential_factoring} does require slightly stronger moment assumptions, which are encapsulated in the following generalization of assumption \hyperlink{assn:C}{(C)} (notice that \hyperlink{assn:C^ell}{(C${}^\ell$)} with $\ell = 1$ is the same as \hyperlink{assn:C}{(C)}).

\begin{enumerate}
\item[\hypertarget{assn:C^ell}{(C${}^\ell$)}] In addition to the Wegner assumption \eqref{eqn:wegner}, we require
\begin{equation}
\label{eqn:detconub_coarses}
	\lim_{N \to \infty} \frac{1}{N} \log \E\left[ \prod_{i=1}^N (1+|\lambda_i|\mathds{1}_{|\lambda_i| > e^{N^\epsilon}})^\ell\right] = 0
\end{equation}
for every $\epsilon > 0$ and 
\begin{equation}
\label{eqn:detconub_NlogNs}
	\limsup_{N \to \infty} \frac{1}{N \log N} \log \E[|\det(H_N)|^{\ell(1+\delta)}] < \infty \quad \text{for each $i$},
\end{equation}
for all sufficiently small $\delta > 0$. 

\end{enumerate}

Here is the analogue of Theorem \ref{thm:convex_functional}.

\begin{thm}
\label{thm:convex_functionals}
\textbf{(Convexity-preserving functionals)}
Fix $\ell \in \N$, and consider $\ell$ collections $(X^{(i)})_{i=1}^\ell$ each consisting of $M$ arbitrary independent entries. The collections can have any correlation structure with respect to each other. Consider matrices $H_N^{(i)} = \Phi^{(i)}(X^{(i)})$ that each satisfy Assumptions \hyperlink{assn:I}{(I)}, \hyperlink{assn:M}{(M)}, \hyperlink{assn:E}{(E)}, \hyperlink{assn:C^ell}{(C${}^\ell$)}, and \hyperlink{assn:S}{(S)} with reference measures $\mu_N^{(i)}$. 
Then
\begin{equation}
\label{eqn:detscon}
	\lim_{N \to \infty} \left( \frac{1}{N}\log \E\left[ \prod_{i=1}^\ell |\det(H_N^{(i)})|\right] - \sum_{i=1}^\ell \int_\R \log\abs{\lambda} \mu_N^{(i)}(\diff \lambda) \right) = 0.
\end{equation}
\end{thm}

\begin{proof}
We refer freely to objects from the proof of Theorem \ref{thm:convex_functional}, adding a parenthetical index $(i)$ to indicate their corresponding matrix. For example, 
\[
	\mc{E}_{\textup{ss}}^{(i)} = \{d_{\textup{KS}}(\hat{\mu}_{\Phi^{(i)}(X^{(i)})}, \hat{\mu}_{\Phi^{(i)}(X^{(i)}_{\textup{cut}})}) \leq N^{-\kappa}\}
\]
and so on. The main estimate in the upper bound is
\begin{align*}
	&\frac{1}{N}\log\E\left[ \prod_{i=1}^\ell |\det(H_N^{(i)})| \mathds{1}_{\mc{E}_{\textup{ss}}^{(i)}} \mathds{1}_{\mc{E}_{\textup{conc}}^{(i)}} \right] \\
	&\leq \frac{1}{N} \log \E\left[ e^{\sum_{i=1}^\ell N \int \log_\eta^K(\lambda) \hat{\mu}_{\Phi^{(i)}(X^{(i)})}(\diff \lambda)} \prod_{i=1}^\ell \left(\prod_{j=1}^N (1+|\lambda_j^{(i)}|\mathds{1}_{|\lambda_j^{(i)}| > K}) \right) \mathds{1}_{\mc{E}_{\textup{ss}}^{(i)}} \mathds{1}_{\mc{E}_{\textup{conc}}^{(i)}} \right] \\
	&\leq \ell(2 \epsilon_1(N) + t) + \sum_{i=1}^\ell \int_\R \log_\eta^K(\lambda)\mu_N^{(i)}(\diff \lambda) + \sum_{i=1}^\ell \frac{1}{\ell N}\log \E\left[ \prod_{j=1}^N (1+|\lambda_j^{(i)}|\mathds{1}_{|\lambda_j^{(i)}| > K})^\ell \right]
\end{align*}
where we use H{\" o}lder's inequality in the last line. Using the assumption \eqref{eqn:detconub_coarses} and arguments as in the one-determinant case, we use this to find
\[
	\limsup_{N \to \infty} \left( \frac{1}{N} \log \E\left[ \prod_{i=1}^\ell |\det(H_N^{(i)})| \mathds{1}_{\mc{E}_{\textup{ss}}^{(i)}} \mathds{1}_{\mc{E}_{\textup{conc}}^{(i)}} \right] - \sum_{i=1}^\ell \int \log\abs{\lambda} \mu_N^{(i)}(\diff \lambda) \right) \leq 0.
\]
To conclude the upper bound, write $\mc{E}^{(i)} = \mc{E}^{(i)}_{\textup{ss}} \cap \mc{E}^{(i)}_{\textup{conc}}$. We expand
\[
	\E\left[ \prod_{i=1}^\ell |\det(H_N^{(i)})| (\mathds{1}_{\mc{E}^{(i)}} + \mathds{1}_{(\mc{E}^{(i)})^c}), \right]
\]
as a sum over $2^\ell$ terms, each of which has a product of $\ell$ determinants and a product of $\ell$ indicators. We just studied the term with every indicator on $\mc{E}^{(i)}$, and now claim that any term with at least one indicator on the complement of $\mc{E}^{(i)}$ does not contribute. Indeed, suppose for concreteness that the indicator $\mathds{1}_{(\mc{E}^{(1)})^c}$ appears; then the term is bounded above by
\[
	\E\left[ \prod_{i=1}^\ell |\det(H_N^{(i)})| \mathds{1}_{(\mc{E}^{(1)})^c} \right] \leq \left( \prod_{i=1}^\ell \E[|\det(H_N^{(i)})|^{\ell(1+\delta)}]^{\frac{1}{\ell(1+\delta)}} \right) \P((\mc{E}^{(1)})^c)^{\frac{\delta}{1+\delta}}
\]
according to H{\"o}lder's. Using the new assumption \eqref{eqn:detconub_NlogNs}, we proceed as in the proof of Lemma \ref{lem:detconub_ec} to complete the proof of the upper bound.

The lower bound is easier to generalize; by following the proof of Lemma \ref{lem:eps2}, we find
\[
	\frac{1}{N} \log \E\left[\prod_{i=1}^\ell e^{ N \int(\log\abs{\lambda} - \log_\eta(\lambda)) \hat{\mu}_{\Phi^{(i)}(X^{(i)})}(\diff \lambda)} \mathds{1}_{\mc{E}^{(i)}_{\textup{gap}}} \mathds{1}_{\mc{E}^{(i)}_{\textup{ss}}} \mathds{1}_{\mc{E}^{(i)}_{\textup{conc}}} \right] \geq -\widetilde{\epsilon_2}(N)
\]
with
\[
	\widetilde{\epsilon_2}(N) = \ell\left(\frac{p_b}{2} \log(1+e^{2N^\epsilon}\eta^2) + \frac{\eta^2}{2w_b^2}\right) - \frac{1}{N}\log\P\left( \bigcap_{i=1}^\ell \mc{E}_{\textup{gap}}^{(i)}, \mc{E}_{\textup{ss}}^{(i)}, \mc{E}_{\textup{conc}}^{(i)}, \mc{E}_b^{(i)} \right),
\]
which tends to zero since each of the events $\mc{E}_{\cdots}^{(i)}$ has probability tending to one.
\end{proof}

Here is the analogue of Theorem \ref{thm:concentrated_input}.

\begin{thm}
\label{thm:concentrated_inputs}
\textbf{(Concentrated inputs)}
Fix $\ell \in \N$, and suppose that each of the matrices $(H_N^{(i)})_{i=1}^\ell$ satisfies the assumptions of the one-determinant Theorem \ref{thm:concentrated_input} with measures $\mu_N^{(i)}$. Then \eqref{eqn:detscon} holds.
\end{thm}
\begin{proof}
For the upper bound, we mimic the proof of the one-determinant case, using H{\"o}lder's to obtain terms of the form $\E[e^{\ell N \int \log_\eta(\lambda)(\hat{\mu}_{H_N^{(i)}} - \E[\hat{\mu}_{H_N^{(i)}}])(\diff \lambda)}]^{1/\ell}$; we simply absorb this $\ell$ into the Lipschitz constant of $\log_\eta$. The lower bound is generalized as in the convexity-preserving-functional case, Theorem \ref{thm:convex_functionals}.
\end{proof}

We give two corollaries.

\begin{cor}
\textbf{(Products of $\ell$ Wigner matrices with subexponential tails)}
Let $(\mu^{(i)})_{i=1}^\ell$ be a collection of centered probability measures on $\R$ with unit variance and subexponential tails in the sense of \eqref{eqn:subexponentialtails} (the constants $\alpha$ and $\beta$ can depend on $i$). Let $W_N^{(i)}$ be a real symmetric Wigner matrix corresponding to $\mu^{(i)}$. Then for every collection $(E^{(i)})_{i=1}^\ell$ we have
\[
	\lim_{N \to \infty} \frac{1}{N} \log \E\left[\prod_{i=1}^\ell |\det(W_N^{(i)}-E^{(i)})|\right] = \sum_{i=1}^\ell \int_\R \log|\lambda - E^{(i)}|\rho_{\text{sc}}(\lambda) \diff \lambda.
\]
\end{cor}
\begin{proof}
We use Theorem \ref{thm:convex_functionals}, verifying its assumptions as in the case of one Wigner matrix. The verification of \eqref{eqn:detconub_coarses} and \eqref{eqn:detconub_NlogNs} is as follows: Dropping $(\cdots)^{(i)}$ from the notation and arguing as in the one-point case, we find
\[
	\E\left[ \prod_{j=1}^N (1+|\lambda_j|\mathds{1}_{|\lambda_j| > e^{N^\epsilon}})^\ell\right] \leq \E\left[ \prod_{j=1}^N \left(1+10\sum_{k=1}^N \abs{B_{jk}} \right)^{\ell}\right],
\] 
where $B$ is the matrix of tails, which has independent entries up to symmetry. When we expand and factor the right-hand side, entries of $B$ now appear with power at most $2\ell$ (instead of $2$ before). Similarly, to verify \eqref{eqn:detconub_NlogNs} we mimic the original notation and find 
\[
	\abs{\det(W_N+E)}^{\ell(1+\delta)} \leq (N!)^{\ell(1+\delta)} \frac{\sum_{\sigma} X_\sigma^{\ell(1+\delta)}}{N!}.
\]
\end{proof}

\begin{cor}
\textbf{(Products of $\ell$ non-invariant Gaussian matrices)}
If $(H_N^{(i)})_{i=1}^\ell$ are Gaussian matrices with a (co)variance profile satisfying the requirements of Corollary \ref{cor:GaussianCovarianceB}, or block-diagonal Gaussian matrices satisfying the requirements of Corollary \ref{cor:GaussianBlock} -- or a mixture of both -- and $\mu_N^{(i)}$ are the corresponding MDE measures, then
\[
	\lim_{N \to \infty} \left( \frac{1}{N} \log \E\left[ \prod_{i=1}^\ell |\det(H_N^{(i)})| \right] - \sum_{i=1}^\ell \int_\R \log\abs{\lambda} \mu_N^{(i)}(\lambda) \diff \lambda \right) = 0.
\]
\end{cor}


\setcounter{equation}{0}
\setcounter{thm}{0}
\renewcommand{\theequation}{B.\arabic{equation}}
\renewcommand{\thethm}{B.\arabic{thm}}
\appendix
\setcounter{secnumdepth}{0}
\hypertarget{sec:secondmoment}{}
\section[Appendix B\ \ \ Second moments and exponential finiteness]
{Appendix B\ \ \ Second moments and exponential finiteness}

In the main text we found \emph{sufficient} conditions to compute the exact value of $\lim_{N \to \infty} \frac{1}{N} \log \E[\abs{\det(H_N)}]$. In this section we ask just when this limit is finite, but the conditions we find are both necessary and sufficient.

For the theorem statement, we fix some probability measure $\mu$, and let $H_N$ be a Wigner matrix associated with $\mu$ in the sense of Section \ref{subsec:wignerstatement}. We emphasize that there are no regularity assumptions on $\mu$ other than the moment ones made in the theorem statement.

\begin{thm}
\label{thm:pthmoment}
Fix $p \geq 1$ and $E \in \R$. If $\mu$ has infinite $(2p)$-th moment, then
\[
	\E[\abs{\det(H_N-E)}^p] = +\infty
\]
for every finite $N$. On the other hand, if $\mu$ has finite $(2p)$-th moment, then
\[
	\limsup_{N \to \infty} \frac{1}{N} \log \E[\abs{\det(H_N-E)}^p] < \infty.
\]
\end{thm}
The proof relies on the following elementary lemma, whose proof will be given at the end:
\begin{lem}
\label{lem:lphadamardconvexity}
For any $a, b \geq 0$ and $p \geq 1$, the function $f : \R_+ \to \R_+$ given by $f(x) = (a+x^{1/p})^{p/2}(b+x^{1/p})^{p/2}$ is concave.
\end{lem}

\begin{proof}[Proof of Theorem \ref{thm:pthmoment}.]
First, suppose that $\mu$ has infinite $(2p)$-th moment. We can write $\det(H_N-E) = X H_{12}^2+Y H_{12}+Z$, where the random vector $(X,Y,Z)$ is independent of $H_{12}$ and $X = -\det((H_{ij}-E)_{3 \leq i, j \leq N})$.  Denoting $\mu_{X,Y,Z}$ the joint law of $(X,Y,Z)$ on $\R^3$ and $\E_{H_{12}}$ denote expectation with respect to $H_{12}$, we have
\[
	\E[\abs{\det(H_N)}^p] \geq \int_{\R^3} \mu_{X,Y,Z}(\diff x, \diff y, \diff z) \mathds{1}_{x \neq 0} \E_{H_{12}}[\abs{xH_{12}^2 + yH_{12} + z}^p].
\]
We have $\P(X = 0) < 1$, so 
it suffices to show that, for every deterministic $(x,y,z) \in \R^3$ with $x \neq 0$, we have $\E[\abs{xH_{12}^2 + yH_{12} + z}^p] = +\infty$. Indeed, suppose without loss of generality that $x > 0$; then there exists a threshold $t_{x,y,z}$ such that, for $\abs{H_{12}} > t_{x,y,z}$, we have $xH_{12}^2 + yH_{12} + z \geq \frac{x}{2}H_{12}^2$; thus
\[
	\E[\abs{xH_{12}^2 + yH_{12} + z}^p] \geq \left(\frac{x}{2}\right)^p \E[\abs{H_{12}}^{2p}\mathds{1}_{\abs{H_{12}} > t_{x,y,z}}] = +\infty.
\]

Now suppose that $\mu$ has finite $(2p)$-th moment, and write the columns of $H_N-E$, which are $N$ vectors each of length $N$, as $(V_i)_{i=1}^N = (V_i^{(E)})_{i=1}^N$. Write $\|\cdot\|$ for the 2-norm on vectors, and for simplicity of notation write $H$ for $H_N$. By Hadamard's inequality $\abs{\det(H-E)} \leq \prod_{i=1}^N \|V_i\|$, it suffices to prove, for some finite $C$, that $\E\left[ \prod_{i=1}^N \|V_i\|^p \right] \leq C^N$.

The problem is that the $V_i$'s are correlated (since $H$ is symmetric); to surmount this, we will apply a Lindeberg-style argument. Fix some ordering of the strict-upper-triangular \emph{positions}, i.e., a map $\ell : \{(i,j)\}_{1 \leq i < j \leq N} \to \llbracket 1, v_{\textup{max}} \rrbracket$ (of course, $v_{\textup{max}} = \frac{N(N-1)}{2}$). Then, for $v \in \llbracket 0, v_{\textup{max}} \rrbracket$, define the symmetric matrix $H^{(v)}$ entrywise through
\begin{align*}
	&\text{for } i < j, \qquad (H^{(v)})_{ij} = \begin{cases} H_{ij} & \text{if } v < \ell((i,j)), \\ \frac{1}{\sqrt{N}} \E[(\sqrt{N}H_{ij})^{2p}]^{\frac{1}{2p}} & \text{if } v \geq \ell((i,j)), \end{cases} \\
	&\text{for } i = j, \qquad (H^{(v)})_{ij} = H_{ij}, \\
	&\text{for } i > j, \qquad (H^{(v)})_{ij} = (H^{(v)})_{ji}.
\end{align*}
In other words, $H^{(0)} = H$, and in $H^{(v)}$ we have replaced $v$ of the entries in the (strict) upper triangle, along with their reflections in the lower triangle, with deterministic numbers. 

Write $(V_i^{(v)})_{1 \leq i \leq N}$ for the columns of $H^{(v)}-E$. For any $1 \leq v \leq v_{\textup{max}}$ we claim that
\begin{equation}
\label{eqn:lplindeberg}
	\E\left[ \prod_{i=1}^N \|V_i^{(v-1)}\|^p \right] \leq \E\left[ \prod_{i=1}^N \|V_i^{(v)}\|^p \right].
\end{equation}
This naturally follows from
\[
	\E_v\left[ \prod_{i=1}^N \|V_i^{(v-1)}\|^p \right] \leq \prod_{i=1}^N \|V_i^{(v)}\|^p,
\]
where $\E_v$ denotes the integration with respect to the single entry $H_{\ell(v)}$ (and of course its reflection in the lower triangle). Writing $H_{\ell(v)} = N^{-1/2}X$ for some random variable $X \sim \mu$ which has finite $(2p)$-th moment, the above equation can be rewritten as
\[
	\E\left[ (a+N^{-1}(X^{2p})^{\frac{1}{p}})^{\frac{p}{2}}(b+N^{-1}(X^{2p})^{\frac{1}{p}})^{\frac{p}{2}} \right] \leq (a+N^{-1}\E[X^{2p}]^{\frac{1}{p}})^{\frac{p}{2}}(b+N^{-1}\E[X^{2p}]^{\frac{1}{p}})^{\frac{p}{2}}
\]
for some $a, b \geq 0$. But this rewriting follows from Jensen's inequality and Lemma \ref{lem:lphadamardconvexity}.

Thus \eqref{eqn:lplindeberg} holds and by iterations we end up with
\[
	\E\left[ \prod_{i=1}^N \|V_i\|^p \right] \leq \E\left[ \prod_{i=1}^N \|V_i^{(v_{\textup{max}})}\|^p \right].
\]
On the right-hand side, all off-diagonal entries were replaced by deterministic ones, so the expectation splits. We have thus proved that
\[
	\E\left[ \prod_{i=1}^N \|V_i\|^p \right] \leq \left(\E\left[ \left((H_{11}-E)^2 + (N-1) \frac{1}{N}\E[(\sqrt{N} H_{12})^{2p}]^{\frac{1}{p}} \right)^{\frac{p}{2}} \right] \right)^{N} \leq C^N,
\]
which concludes the proof.
\end{proof}

\begin{proof}[Proof of Lemma \ref{lem:lphadamardconvexity}]
One can compute
\[
	2f'(x)=(a+x^{\frac{1}{p}})^{\frac{p}{2}-1}(b+x^{\frac{1}{p}})^{\frac{p}{2}}x^{\frac{1}{p}-1} + (a+x^{\frac{1}{p}})^{\frac{p}{2}} (b+x^{\frac{1}{p}})^{\frac{p}{2}-1}x^{\frac{1}{p}-1}.
\]
Thus
\begin{multline*}
2f''(x)=\left[(a+x^{\frac{1}{p}})^{\frac{p}{2}-1}(b+x^{\frac{1}{p}})^{\frac{p}{2}}
+
(a+x^{\frac{1}{p}})^{\frac{p}{2}} (b+x^{\frac{1}{p}})^{\frac{p}{2}-1}\right]\left(\frac{1}{p}-1\right)x^{\frac{1}{p}-2}
+
\left[
(a+x^{\frac{1}{p}})^{\frac{p}{2}-1}(b+x^{\frac{1}{p}})^{\frac{p}{2}-1}x^{\frac{1}{p}-1}\right.\\
\left.+
\left(\frac{p}{2}-1\right)(a+x^{\frac{1}{p}})^{\frac{p}{2}-2}\frac{1}{p}x^{\frac{1}{p}-1}(b+x^{\frac{1}{p}})^{\frac{p}{2}}
+
(a+x^{\frac{1}{p}})^{\frac{p}{2}} \left(\frac{p}{2}-1\right)(b+x^{\frac{1}{p}})^{\frac{p}{2}-2}\frac{1}{p}x^{\frac{1}{p}-1}\right]x^{\frac{1}{p}-1},
\end{multline*}
i.e.
\begin{multline*}
2(a+x^{\frac{1}{p}})^{1-\frac{p}{2}} (b+x^{\frac{1}{p}})^{1-\frac{p}{2}}x^{2-\frac{2}{p}}\, f''(x)=\left[(a+x^{\frac{1}{p}})
+
(b+x^{\frac{1}{p}})\right]\left(\frac{1}{p}-1\right)x^{-\frac{1}{p}}
+
1
+
\left(\frac{1}{2}-\frac{1}{p}\right)\left(\frac{(a+x^{\frac{1}{p}})}{(b+x^{\frac{1}{p}})}+\frac{(b+x^{\frac{1}{p}})}{(a+x^{\frac{1}{p}})}\right)
\end{multline*}
and $f''(x)$ has the same sign as 
\[
	\left(\frac{1}{2}-\frac{1}{p}\right)\left(\frac{(a+x^{\frac{1}{p}})}{(b+x^{\frac{1}{p}})}+\frac{(b+x^{\frac{1}{p}})}{(a+x^{\frac{1}{p}})}-2\right)+\left(\frac{1}{p}-1\right)x^{-\frac{1}{p}}(a+b).
\]
When we change variables as $u=ax^{-1/p}, v=bx^{-1/p}$, the above becomes
\[
	\left(\frac{1}{2}-\frac{1}{p}\right)\left(\frac{1+u}{1+v}+\frac{1+v}{1+u}-2\right)+\left(\frac{1}{p}-1\right)(u+v),
\]
and it suffices to show this is nonpositive for all $u, v \geq 0$ and $p \geq 1$. Since $y + y^{-1} - 2 \geq 0$, it is nonpositive when $p = 1$; it thus suffices to show that it is non-increasing in $p$ for every fixed $u$ and $v$. Its partial derivative in $p$ is $\frac{1}{p^2}\left(\frac{1+u}{1+v}+\frac{1+v}{1+u} - 2 - (u+v)\right)$, so we just want to show that $g(u,v) = \frac{1+u}{1+v}+\frac{1+v}{1+u} - 2 - (u+v)$ is nonpositive for all $u, v \geq 0$; but $g$ has the properties 
(i) $g(0,v) = \frac{1}{1+v}+1+v-2-v = \frac{1}{1+v}-1 \leq 0$, and 
(ii) $\partial_{u} g(u,v) = -1+\frac{1}{1+v}-\frac{1+v}{(1+u)^2} \leq -1+\frac{1}{1+v} = \frac{-v}{1+v} \leq 0$, which completes the proof.
\end{proof}

\addcontentsline{toc}{section}{References}
\bibliographystyle{alpha-abbrvsort}
\bibliography{complexitybib}

\end{document}